\documentclass[10pt,reqno]{amsart}
\usepackage{amssymb,mathrsfs,graphicx,extpfeil,amsmath,bbm}
\usepackage{ifthen,latexsym,float,colortbl}
\usepackage[margin=1in]{geometry}
\usepackage{caption}
\usepackage{sidecap}
\usepackage{rotating,dsfont,stackengine}
\usepackage{enumitem}
 
\usepackage{subcaption}
\usepackage{colortbl}
\definecolor{black}{rgb}{0.0, 0.0, 0.0}
\definecolor{red}{rgb}{1.0, 0.5, 0.5}
\provideboolean{shownotes} 
\setboolean{shownotes}{true} 
\newcommand{\margnote}[1]{
\ifthenelse{\boolean{shownotes}}%
{\marginpar{\raggedright\tiny\texttt{#1}}}%
{}%
}
\newcommand{\hole}[1]{
\ifthenelse{\boolean{shownotes}}%
{\begin{center} \fbox{ \rule {.25cm}{0cm} \rule[-.1cm]{0cm}{.4cm}
\parbox{.85\textwidth}{\begin{center} \texttt{#1}\end{center}} \rule
{.25cm}{0cm}}\end{center}} {} }

\graphicspath{{figures/}}

\title[From kinetic mixtures to compressible two-phase flow]
{From kinetic mixtures to compressible two-phase flow: \\ A BGK-type model and rigorous derivation}

\author[Cho]{Seung Yeon Cho}
\address[Seung Yeon Cho]{\newline Department of Mathematics\newline
Gyeongsang National University, Jinju 52828, Republic of Korea}
\email{chosy89@gnu.ac.kr}

\author[Choi]{Young-Pil Choi}
\address[Young-Pil Choi]{\newline Department of Mathematics\newline
Yonsei University, Seoul 03722, Republic of Korea}
\email{ypchoi@yonsei.ac.kr}

\author[Hwang]{Byung-Hoon Hwang}
\address[Byung-Hoon Hwang]{\newline Department of Mathematics Education\newline
Sangmyung University, Seoul 03016, Republic of Korea}
\email{bhhwang@smu.ac.kr}

\author[Song]{Sihyun Song}
\address[Sihyun Song]{\newline Department of Mathematics\newline
Yonsei University, Seoul 03722, Republic of Korea}
\email{ssong@yonsei.ac.kr}

\numberwithin{equation}{section}

\newtheorem{theorem}{Theorem}[section]
\newtheorem{lemma}{Lemma}[section]

\newtheorem{proposition}{Proposition}[section]
\newtheorem{remark}{Remark}[section]

\newcommand{\R}{\mathbb R}
 
\newcommand{\N}{\mathbb N}

\newcommand{\T}{\mathbb T}

\newcommand{\bq}{\begin{equation}}
\newcommand{\eq}{\end{equation}}
\newcommand{\e}{\varepsilon}
\newcommand{\lt}{\left}
\newcommand{\rt}{\right}

\newcommand{\pa}{\partial}

\newcommand{\intr}{\int_{\R^n}}

\makeatletter
\def\moverlay{\mathpalette\mov@rlay}
\def\mov@rlay#1#2{\leavevmode\vtop{%
   \baselineskip\z@skip \lineskiplimit-\maxdimen
   \ialign{\hfil$\m@th#1##$\hfil\cr#2\crcr}}}
\newcommand{\charfusion}[3][\mathord]{
    #1{\ifx#1\mathop\vphantom{#2}\fi
        \mathpalette\mov@rlay{#2\cr#3}
      }
    \ifx#1\mathop\expandafter\displaylimits\fi}
\makeatother

\newcommand{\calX}{\mathcal X}
\newcommand{\calM}{\mathcal M}
\newcommand{\calL}{\mathcal L}
\DeclareMathOperator*{\argmin}{arg\,min}
\newcommand{\calE}{\mathcal E}
\newcommand{\p}{\partial}
\newcommand{\ve}{\varepsilon}

\begin{document}
\allowdisplaybreaks

\date{\today}


\begin{abstract} 
We propose a BGK-type kinetic model for a binary gas mixture, designed to serve as a kinetic formulation of compressible two-phase fluid dynamics. The model features species-dependent adiabatic exponents, and the relaxation operator is constructed by solving an entropy minimization problem under moments constraints. Starting from this model, we derive the compressible two-phase Euler equations via a formal Chapman--Enskog expansion and identify dissipative corrections of Navier--Stokes type. We then rigorously justify the Euler limit using the relative entropy method, establishing quantitative convergence estimates under appropriate regularity assumptions. Finally, we present numerical experiments based on an implicit-explicit Runge--Kutta method, which confirm the asymptotic preserving property and demonstrate the convergence from the BGK model to the isentropic two-phase Euler system in the hydrodynamic regime.
\end{abstract}

\maketitle \centerline{\date}

\tableofcontents

%
%
%
%
\section{Introduction}

The dynamics of compressible fluids without entropy variations are governed by the \emph{isentropic Euler equations}. Such dynamics arise in a variety of physical settings, including the propagation of sound waves in an ideal gas and the expansion of gases through nozzles under smooth flow conditions. These equations describe the conservation of mass and momentum, with a barotropic pressure law of the form $p \propto \rho^\gamma$, where $\gamma \in [1, \frac{n+2}{n}]$ denotes the adiabatic index and $n$ is the spatial dimension.

To analyze isentropic flow from a kinetic perspective, a line of work initiated by Lions, Perthame, and Tadmor \cite{LPT91,LPT94,LPT94-2} introduced a kinetic formulation of the isentropic Euler equations by solving an auxiliary entropy identity. This approach encodes the macroscopic conservation laws into kinetic equations via a suitable velocity distribution function. In particular, these kinetic representations are instrumental in analyzing entropy solutions, constructing numerical schemes, and justifying hydrodynamic limits.

Building upon this, Bouchut \cite{B99} proposed a general framework for BGK-type kinetic models relaxing toward isentropic Euler equations. The classical BGK model, introduced by Bhatnagar, Gross, and Krook \cite{BGK54}, approximates the Boltzmann equation by a relaxation process toward a local Maxwellian distribution. Bouchut's contribution was to characterize such equilibria via an entropy minimization principle under mass and momentum constraints, thereby ensuring thermodynamic consistency with the target isentropic fluid system.

On the other hand, significant progress has also been made in the kinetic modeling of multi-component gas mixtures. In a seminal work, Andries, Aoki, and Perthame \cite{AAP02} proposed the consistent BGK-type model for inert mixtures. Since then, various strategies have been explored to approximate the Boltzmann collision operator by a BGK relaxation term. Broadly speaking, two main modeling approaches exist: one uses a \emph{single BGK operator} shared across species, while the other mimics the Boltzmann structure via a \emph{sum of interspecies operators}. 

The single-operator approach offers modeling simplicity, provides a framework for modeling more general gaseous systems, and is possibly justified via entropy minimization. Several developments have followed in this direction, for instance, the works of Groppi, Bisi, Brull, and their collaborators for reactive mixtures, polyatomic molecules, and relativistic corrections  \cite{BMS18,BT20,Brull15,BPS12,GMS11,GRS09,HLY24}. The multi-operator strategy, inspired by the original Gross--Krook model \cite{GK56}, allows more flexibility in modeling detailed interspecies dynamics but may be more complex to analyze or simulate. Recently, the consistent models following the multi-operator strategy have been derived by Bobylev, Haack, Klingenberg, Pirner, and  others \cite{BBGSP18,HHM17,KPP17}. See \cite{BKPY21} for the extension to quantum mixtures. For both types of BGK models, we refer to \cite{BKPY22,KLY21,KP18,Pirner18} for the existence of solutions and \cite{CBGR22,CKP20,GRS18} for numerical simulations.

Formal derivations of macroscopic two-phase fluid equations from such kinetic mixture models have also been attempted, often based on asymptotic expansions or moment closure techniques. While these studies provide valuable insights into the structure of the resulting fluid systems, they are typically limited to formal arguments and do not establish rigorous convergence from the kinetic level.

On the other hand, a number of works have rigorously addressed the derivation of macroscopic two-phase models starting from \emph{kinetic-fluid systems}, where a kinetic equation governs the dispersed particle phase and the fluid is described by compressible or incompressible Navier--Stokes or Euler equations. The coupling is usually modeled through drag forces. Among early contributions in this direction, the works of Goudon, Jabin, and Vasseur \cite{GJV04a,GJV04b} analyzed different asymptotic regimes, namely, the light and fine particle limits, which lead to either passive transport models or coupled two-phase Navier--Stokes systems. Han-Kwan and Michel \cite{HKM24} developed a unified framework for high-friction hydrodynamic limits from the Vlasov--Navier--Stokes system using modulated energy and entropy techniques. Related results include \cite{MV08}, which considered the asymptotic behavior of a Vlasov--Fokker--Planck/compressible Navier--Stokes system, and \cite{SWYZ23}, where the relative entropy method was applied to derive a coupled system consisting of an inhomogeneous two-phase fluid system from weak solutions of a kinetic-fluid model. Other notable works include \cite{CCK16,CG06,CJ21,CJ23}, which addressed similar limits in the presence of nonlinear Fokker--Planck operators or electrostatic interactions.

In contrast to these approaches, the present work addresses the derivation of two-phase models starting from a  \emph{kinetic description for both phases} in a binary gas mixture. Specifically, we consider a BGK-type relaxation model in which both species evolve via kinetic equations and interact through the common velocity for mixtures. While formal derivations of two-phase fluid systems from kinetic mixtures have appeared in the literature mentioned above, a complete rigorous justification has, to the best of our knowledge, remained open. This paper provides the first such derivation of the isentropic two-phase Euler equations from a unified kinetic BGK model, using the relative entropy method and making precise the connection between mesoscopic relaxation mechanisms and macroscopic fluid dynamics.

To this end, we construct a single BGK-type model for a binary mixture of gases that serves as a kinetic formulation of the isentropic two-phase fluid system.  Each species is characterized by its own adiabatic exponent $\gamma_i > 1$ and particle mass $m_i > 0$, and the two components are assumed to interact solely through relaxation to a shared bulk velocity field. Our primary aim is to derive, from this kinetic model, a macroscopic system that reflects the structure of the compressible two-phase Euler or Navier-Stokes equations, depending on the scaling regime.

The proposed kinetic system describes a binary mixture through two distribution functions $f_i(t,x,v)$, defined on $(x,v) \in \Omega \times \R^n$ for $t>0$, where $\Omega = \T^n$ or $\R^n$. These distributions evolve according to BGK-type kinetic equations:
\begin{equation}\label{main}
\partial_t f_i + v \cdot \nabla_x f_i = \nu_i \left( \mathcal{M}_i[n_i, u] - f_i \right), \quad i=1,2,
\end{equation}
where $\nu_i > 0$ denotes the collision frequency (or the inverse of the mean relaxation time) for species $i$, which may depend on $n_i$. The right-hand side models the tendency of $f_i$ to relax toward a local equilibrium distribution $\mathcal{M}_i[n_i, u]$, characterized by the number density $n_i$ of species $i$ and a common bulk velocity $u$ shared by both species.

The equilibrium distribution $\mathcal{M}_i[n_i, u](v)$ is designed to be consistent with the thermodynamic structure of isentropic gas dynamics and is obtained via an entropy minimization principle. It takes the form
\[
\mathcal{M}_i[n_i, u](v) =
\begin{cases}
\displaystyle c_i \, \mathbf{1}_{|u-v|^2 \le \frac{2\gamma_i}{\gamma_i - 1} n_i^{\frac 2n}} & \text{if } \gamma_i = \frac{n+2}{n}, \\[2mm]
\displaystyle c_i \left( m_i^{\frac{n}{2}(\gamma_i - 1)} \frac{2\gamma_i}{\gamma_i - 1} n_i^{\gamma_i - 1} - m_i |v - u|^2 \right)_+^{\frac{d_i}{2}} & \text{if } \gamma_i \in \left(1, \frac{n+2}{n} \right).
\end{cases}
\]
Here, the constants $c_i$ and $d_i$ are chosen so that the internal energy and entropy associated with $\mathcal{M}_i$ match the corresponding expressions in isentropic gas dynamics. A detailed derivation of $\mathcal{M}_i$ is provided in Appendix \ref{app_deri}, following the framework of entropy minimization introduced in \cite{B99, LPT94} under suitable constraints. The constants are explicitly given by 
\[
c_i = \left( \frac{2\gamma_i}{\gamma_i - 1} \right)^{-\frac{1}{\gamma_i - 1}} \frac{\Gamma\left( \frac{\gamma_i}{\gamma_i - 1} \right)}{\pi^{n/2} \Gamma\left( \frac{d_i}{2} + 1 \right)}, \qquad
d_i = \frac{2}{\gamma_i - 1} - n,
\]
where $\Gamma$ denotes the Gamma function.  

The macroscopic quantities are defined as follows. The number density $n_i$, mass density $\rho_i$, and bulk velocity $u_i$ of species $i$ are defined by
\[
n_i = \int_{\mathbb{R}^n} f_i(v)\,dv, \qquad \rho_i=m_in_i,\qquad u_i = \frac{1}{n_i} \int_{\mathbb{R}^n} v f_i(v)\,dv,
\]
respectively, with the convention $u_i = 0$ when $n_i = 0$. Then the mass density $\rho$ and bulk velocity $\mathbf{u}$ of mixtures are given by 
\begin{equation}\label{total} 
 \rho=\rho_1+\rho_2,\qquad {\bf{u}}=\frac{\rho_1u_1+\rho_2u_2}{\rho}.
\end{equation}
We define the common bulk velocity $u$, which is the auxiliary parameter of $\mathcal{M}$, as
\[
u=\frac{\sum_{i=1}^2  \nu_i \rho_i u_i}{\sum_{i=1}^2 \nu_i \rho_i}.
\]
To ensure thermodynamic consistency, we introduce the total kinetic entropy
\[
H(f,v) := \sum_{i=1}^2 \nu_i h_i(f_i,v),
\]
where the kinetic entropy $h_i$ for each species is defined by
\[
h_i(f_i,v) =
\begin{cases}
\displaystyle \frac{m_i}{2} |v|^2 f_i + \infty \cdot \mathbf{1}_{f_i > c_i}(v) & \text{if } \gamma_i = \frac{n+2}{n}, \\[2mm]
\displaystyle \frac{m_i}{2} |v|^2 f_i + \frac{1}{2c_i^{2/d_i}} \frac{f_i^{1+2/d_i}}{1+2/d_i} & \text{if } \gamma_i \in \left(1, \frac{n+2}{n} \right).
\end{cases}
\]
It can be verified that the entropy of the equilibrium state satisfies
\[
\int_{\mathbb{R}^n} h_i(\mathcal{M}_i[n_i, u](v), v)\,dv = \frac{\rho_i}{2}  |u|^2 + \frac{m_i^{\frac{n}{2}(\gamma_i - 1)}}{\gamma_i - 1} n_i^{\gamma_i},
\]
which recovers the internal energy and pressure terms in the isentropic gas model. Furthermore, the following entropy minimization principle holds:
\[
\sum_{i=1}^2 \int_{\R^n}  \nu_i h_i(f_i,v)\,dv \geq \sum_{i=1}^2 \int_{\mathbb{R}^n} \nu_i h_i(\mathcal{M}_i[n_i, u](v), v)\,dv,
\]
which yields the dissipation rate in the total entropy and is crucial for the entropy inequality discussed in Section \ref{sec_pre}.

 Our main objective is to rigorously justify the derivation of macroscopic two-phase fluid equations from this kinetic model. We begin by performing a formal Chapman--Enskog expansion under a hydrodynamic scaling, which serves to identify the macroscopic structure associated with the model. This expansion yields the compressible two-phase Euler equations at leading order, and a Navier--Stokes-type correction with density-dependent viscosity at the next order, thus capturing both the conservative and dissipative effects induced by interspecies momentum relaxation.

To complement the formal analysis, we carry out a rigorous derivation of the isentropic two-phase Euler equations using the relative entropy method. This involves constructing a suitable modulated entropy functional and establishing convergence of the kinetic solutions to their macroscopic counterparts under appropriate regularity assumptions. This justifies, for the first time, the hydrodynamic limit from a BGK-type kinetic mixture to a two-phase Euler system in a mathematically rigorous framework.

The rest of the paper is organized as follows. In Section \ref{sec_pre}, we examine the structural properties of the kinetic model, including conservation of mass and momentum, and the entropy dissipation. Section \ref{sec_ce} is devoted to a formal asymptotic analysis via the Chapman--Enskog expansion, which yields a macroscopic two-phase fluid system and reveals the structure of dissipative corrections. In Section \ref{sec_re}, we rigorously derive the isentropic two-phase Euler equations from the kinetic model using the relative entropy method, constructing a modulated entropy functional and establishing convergence under suitable regularity assumptions. Section \ref{sec_numer} presents the numerical scheme, emphasizing the implicit-explicit Runge--Kutta (IMEX-RK) time integration and the asymptotic preserving property of the method. In Section \ref{sec_numer2}, we perform numerical tests comparing solutions of the kinetic model and the isentropic Euler system in the hydrodynamic limit, confirming the consistency of the scheme and supporting the theoretical analysis. Appendix \ref{app_deri} provides the derivation of the truncated Maxwellians via an entropy minimization principle. Appendix \ref{app_vm} contains the detailed proof of Lemma \ref{lem_vm}, establishing key velocity moment identities for the truncated Maxwellians, which are essential for the Chapman--Enskog expansion. Appendix \ref{app_gwp} establishes the global existence of mild solutions to the kinetic model, which is required for the application of the relative entropy method in the rigorous hydrodynamic limit.

\section{Structural properties: Conservation laws and entropy dissipation}\label{sec_pre}

The kinetic model \eqref{main} possesses several fundamental structural properties that are consistent with the underlying physics of two-phase flows. In this section, we demonstrate that the system conserves mass and momentum and satisfies an entropy dissipation principle. These properties are essential to ensure that the kinetic formulation behaves consistently with the expected macroscopic dynamics in the hydrodynamic limit.

\begin{enumerate}[label=(\roman*)]
	\item \emph{Mass conservation for each species}: For each species $i = 1,2$, the equation \eqref{main} preserves the total mass over time. Indeed, integrating both sides of \eqref{main} against $m_i$ over velocity and space, and using the fact that the collision operator conserves mass, we obtain
	\[
	\frac{d}{dt} \iint_{\Omega \times \mathbb{R}^n} m_i f_i \,dx dv = 0,\quad i=1,2.
	\]
	This confirms that the individual mass of each species remains constant.

	\item \emph{Conservation of total momentum}: The mixture evolves under the constraint that total momentum is conserved. By multiplying \eqref{main} by $m_i v$ and integrating in $(x,v)$, we use the fact that the relaxation term preserves the total momentum because
	\[
\nu_i m_i  \int_{\mathbb{R}^n} v 	\left( \mathcal{M}_i[n_i, u] - f_i \right) dv =  \nu_i\rho_i (u - u_i).
	\]
	Since the definition of $u$ is constructed to ensure $\sum_{i=1}^2 \nu_i \rho_i u_i = \sum_{i=1}^2 \nu_i \rho_i u$, we have
	\[
	\frac{d}{dt} \sum_{i=1}^2 \iint_{\Omega \times \mathbb{R}^n} m_i v f_i \,dx dv = 0,
	\]
	which confirms the global momentum conservation in the absence of external forces.

	\item \emph{Entropy dissipation $(${\rm H}-theorem$)$}: Entropy dissipation is a fundamental feature of BGK-type models. We define the total entropy functional
	\[
	\mathcal{H}(f) := \iint_{\Omega \times \mathbb{R}^n} H(f,v)\,dxdv
	\]
	and compute its time derivative:
	\[
	\frac{d}{dt} \mathcal{H}(f) + \mathcal{D}(f) \leq 0,
	\]
	where the entropy dissipation rate is
	\[
	\mathcal{D}(f) := \sum_{i=1}^2 \nu_i  \iint_{\Omega \times \mathbb{R}^n} \left( h_i(f_i,v) - h_i(\mathcal{M}_i[n_i, u](v), v) \right)dxdv.
	\]
	The non-negativity $\mathcal{D}(f) \geq 0$ follows from the convexity of $h_i$ and the fact that $\mathcal{M}_i$ minimizes the entropy functional under mass and momentum constraints, see Appendix \ref{app_deri} and Remark \ref{rem: each h}.

	We now verify this inequality in two cases depending on the value of the adiabatic index $\gamma_i$.

\emph{Case $\gamma_i = \frac{n+2}{n}$}: This corresponds to a uniform truncation of the distribution function. Since 
	\[
	\int_{\R^n} h_i(f_i,v)\,dv = \int_{\R^n}\frac{m_i}{2}|v|^2 f_i + \infty \cdot \mathbf{1}_{f_i > c_i}\,dv, 
	\]
	the entropy is minimized precisely when $f_i = \mathcal{M}_i$, by construction (see Remark \ref{rem: each h}). Taking the time derivative of the total entropy then yields
\bq\label{end point}
	\frac{d}{dt} \sum_{i=1}^2 \iint_{\Omega \times \mathbb{R}^n} \frac{m_i}{2} |v|^2 f_i \,dx\,dv = \sum_{i=1}^2 \frac{\nu_i m_i}{2} \iint_{\Omega \times \mathbb{R}^n} |v|^2 (\mathcal{M}_i - f_i)\,dx\,dv \leq 0,
\eq
which confirms the entropy inequality.

\emph{Case $\gamma_i \in \left(1, \frac{n+2}{n} \right)$}: In this case, we use the power-law form of the kinetic entropy:
	\[
	h_i(f_i,v) = \frac{m_i}{2} |v|^2 f_i + \frac{1}{2c_i^{2/d_i}} \frac{f_i^{1+2/d_i}}{1+2/d_i}.
	\]
	Multiplying \eqref{main} by $f_i^{2/d_i}/(2c_i^{2/d_i})$, integrating over $(x,v)$, and summing over $i$ give
\[
\frac{d}{dt} \sum_{i=1}^2 \iint_{\Omega \times \mathbb{R}^n} \frac{1}{2c_i^{2/d_i}} \frac{f_i^{1+2/d_i}}{1+2/d_i} 
\,dxdv = \sum_{i=1}^2 \frac{\nu_i}{2c_i^{2/d_i}} \iint_{\Omega \times \mathbb{R}^n} \left( \mathcal{M}_i f_i^{2/d_i} - f_i^{1+2/d_i} \right) dx dv.
\]
Applying Young's inequality with exponents $p = \frac{d_i + 2}{d_i}$ and $q = \frac{d_i + 2}{2}$ yields
\begin{align*}
&\frac{d}{dt} \sum_{i=1}^2 \iint_{\Omega \times \R^n} \frac{1}{2c_i^{2/d_i}} \frac{f_i^{1+2/d_i}}{1+2/d_i}\,dxdv\cr
&\quad \le \sum_{i=1}^2 \frac{\nu_i}{2c_i^{2/d_i}} \iint_{\Omega \times \R^n} \left( \frac{d_i}{d_i+2} \mathcal{M}_i^{1+2/d_i} + \frac{2}{d_i+2} f_i^{1+2/d_i} - f_i^{1+2/d_i} \right) dx dv \\
&\quad = \sum_{i=1}^2 \frac{\nu_i}{2c_i^{2/d_i}} \iint_{\Omega \times \R^n} \left( \frac{\mathcal{M}_i^{1+2/d_i}}{1+2/d_i} - \frac{f_i^{1+2/d_i}}{1+2/d_i} \right) dx dv.
\end{align*}
Combining this with the entropy identity \eqref{end point}, we conclude that the total entropy is nonincreasing, as guaranteed by the entropy minimization principle.
\end{enumerate}

\section{Formal derivation: Chapman--Enskog method}\label{sec_ce}

In this section, we carry out a formal derivation of the macroscopic two-phase fluid model from the proposed kinetic BGK system \eqref{main}, using the Chapman--Enskog method. This asymptotic procedure is based on a hyperbolic scaling, which assumes that both temporal and spatial variations occur on a slow scale. Precisely, we apply the scalings $t \mapsto \varepsilon t$ and $x \mapsto \varepsilon x$, leading to the scaled system:
\begin{equation}\label{CE main}
\partial_t f_i + v\cdot \nabla_x f_i = \frac{\nu_i}{\varepsilon} \left( \mathcal{M}_i[n_i, u] - f_i \right), \qquad i=1,2,
\end{equation}
where $\varepsilon$ represents the Knudsen number.  In this setting, we assume $\varepsilon$ is small and proceed with a formal expansion of the distribution function $f_i$ in powers of $\varepsilon$:
\begin{equation}\label{decomp}
f_i = f_i^{(0)} + \varepsilon f_i^{(1)}.
\end{equation}
The expansion is performed under the assumption that the macroscopic quantities $n_i$ and $\rho_1 u_1 + \rho_2 u_2$ are not expanded in $\varepsilon$. This leads to the so-called \emph{compatibility conditions}:
\begin{align}\label{compatibility}
\int_{\mathbb{R}^n} f_i\,dv = \int_{\mathbb{R}^n} f_i^{(0)}\,dv, \ i=1,2 \quad \mbox{and} \quad \int_{\mathbb{R}^n} v \left( m_1 f_1 + m_2 f_2 \right) dv = \int_{\mathbb{R}^n} v \left( m_1 f_1^{(0)} + m_2 f_2^{(0)} \right) dv.
\end{align}
These relations guarantee that the macroscopic variables $n_i$, $\rho_i$, and $\rho_1u_1+\rho_2u_2$ are independent of $\varepsilon$.
Let us consider the expansions of collision frequenties, bulk velocity of species $i$, common bulk velocity, and equilibrium distributions: 	 
	$$
\nu_i=\nu_i^{(0)}+\ve \nu_i^{(1)},\quad u_i=u_i^{(0)}+\ve u_i^{(1)}\quad u=u^{(0)}+\ve u^{(1)},\quad 	\mathcal{M}_i=\mathcal{M}_i^{(0)}+\ve\mathcal{M}_i^{(1)}.
	$$
It is straightforward that
\begin{align*}
u&=\frac{\sum_{i=1}^2\nu_i \rho_iu_i}{\sum_{i=1}^2\nu_i \rho_i}\cr 
&=\frac{\sum_{i=1}^2\nu_i^{(0)} \rho_iu_i^{(0)}+\ve\left( \sum_{i=1}^2\nu_i^{(0)} \rho_iu_i^{(1)}+\sum_{i=1}^2\nu_i^{(1)} \rho_iu_i^{(0)}\right)+\ve^2\sum_{i=1}^2\nu_i^{(1)} \rho_iu_i^{(1)}}{\sum_{i=1}^2\nu_i^{(0)} \rho_i+\ve \sum_{i=1}^2\nu_i^{(1)} \rho_i},
\end{align*}
which implies
\begin{equation*}
u^{(0)}=u|_{\ve=0}=\frac{\sum_{i=1}^2\nu_i^{(0)} \rho_iu_i^{(0)}}{\sum_{i=1}^2\nu_i^{(0)} \rho_i}.
\end{equation*}
On the other hand, since $n_i$ is independent of $\ve$, $\mathcal{M}_i^{(0)}$ and $\mathcal{M}_i^{(1)}$ can be understood as
\begin{equation*}
\mathcal{M}_i^{(0)}=\left\{\mathcal{M}_i[n_i,u]\right\}_{\ve=0}=\mathcal{M}_i[n_i,u^{(0)}],
\end{equation*}
and
\begin{align}\label{M1}\begin{split}
\mathcal{M}_i^{(1)}&=\frac{d}{d\ve}\left\{\mathcal{M}_i[n_i,u]\right\}_{\ve=0}\cr 
&=\left\{\nabla_u \mathcal{M}_i[n_i,u]\cdot \frac{d u}{d\ve}\right\}_{\ve=0}\cr 
&=d_im_ic_i^{\frac{2}{d_i}}(v-u^{(0)})\mathcal{M}_i^{\frac{d_i-2}{d_i}}[n_i,u^{(0)}]\cr
&\hspace{2.5cm} \cdot \left\{ \frac{ \sum_{i=1}^2\nu_i^{(0)} \rho_iu_i^{(1)}+\sum_{i=1}^2\nu_i^{(1)} \rho_iu_i^{(0)}}{\sum_{i=1}^2\nu_i^{(0)} \rho_i}-\frac{ \left(\sum_{i=1}^2\nu_i^{(1)} \rho_i \right)\left(\sum_{i=1}^2\nu_i^{(0)} \rho_iu_i^{(0)}\right)}{\left(\sum_{i=1}^2\nu_i^{(0)} \rho_i\right)^2} \right\}.
\end{split}\end{align} 
\subsection{Zeroth-order expansion}

Substituting the expansion  into the equation \eqref{CE main}, we obtain
\begin{align}\label{e eqn}\begin{split}
&\partial_t \left( f_i^{(0)} + \varepsilon f_i^{(1)} \right) + v \cdot \nabla_x \left( f_i^{(0)} + \varepsilon f_i^{(1)} \right)\cr 
&\quad = \frac{\nu_i^{(0)}}{\varepsilon} \left( \mathcal{M}^{(0)}_i[n_i, u] - f_i^{(0)}  \right)+\left\{\nu_i^{(1)}\left(\mathcal{M}^{(0)}_i[n_i, u] - f_i^{(0)} \right)+\nu_i^{(0)} \left(\mathcal{M}^{(1)}_i[n_i, u]- f_i^{(1)}\right)\right\}\cr 
&\qquad +\ve \nu_i^{(1)}\left(\mathcal{M}^{(1)}_i[n_i, u]- f_i^{(1)}\right).
\end{split}\end{align}
Collecting terms by powers of $\varepsilon$, we find that the leading-order equation at order $O(\varepsilon^{-1})$ gives
\begin{equation}\label{e-1}
f_i^{(0)} = \mathcal{M}^{(0)}_i
\end{equation}
which states that the zeroth-order distribution is simply the truncated Maxwellian associated with $(n_i, u^{(0)})$. This, combined with \eqref{compatibility}, leads to
\begin{align*}
\rho_1u_1+\rho_2u_2&=\sum_{i=1}^2\intr m_ivf_i^{(0)}\,dv=\sum_{i=1}^2 \intr m_iv \mathcal{M}_i[n_i,u^{(0)}]\,dv=(\rho_1+\rho_2) u^{(0)},
\end{align*}
i.e. 
\begin{equation}\label{u^0}
u^{(0)}=\frac{\rho_1u_1+\rho_2u_2}{\rho}={\bf{u}}.
\end{equation} 
To derive macroscopic equations, we multiply \eqref{CE main} by $(1, m_i v)$ and integrate over $v \in \mathbb{R}^n$. This yields the balance laws
\begin{align}\label{balance}
\begin{split}
&\partial_t n_i + \nabla_x \cdot (n_i {\bf{u}}) = -\nabla_x \cdot \left\{ n_i (u_i - {\bf{u}}) \right\}, \qquad i = 1, 2, \\
&\partial_t \left( \rho {\bf{u}} \right) + \nabla_x \cdot \left( \rho {\bf{u}} \otimes {\bf{u}}\right) + \nabla_x \cdot P = 0,
\end{split}
\end{align}
where $P = P_1 + P_2$ is the total pressure tensor, defined by
\[
P_i = m_i \int_{\mathbb{R}^n} (v - {\bf{u}}) \otimes (v - {\bf{u}}) f_i\,dv.
\]
Using the expansion \eqref{decomp}, we now expand $P_i$ accordingly:
\begin{align*}
P_i &= m_i \int_{\mathbb{R}^n} (v - {\bf{u}}) \otimes (v - {\bf{u}}) \left( f_i^{(0)} + \varepsilon f_i^{(1)} \right) \,dv \\
&=: P_i^{(0)} + \varepsilon P_i^{(1)},
\end{align*}
where the zeroth- and first-order terms are respectively given by
\[
P_i^{(0)} = m_i \int_{\mathbb{R}^n} (v - {\bf{u}}) \otimes (v - {\bf{u}}) \mathcal{M}_i[n_i, {\bf{u}}]\,dv = m_i^{\frac{n}{2}(\gamma_i - 1)} n_i^{\gamma_i} \mathbb{I}_{n \times n},
\]
\begin{align*}
P_i^{(1)} &= m_i \int_{\mathbb{R}^n} (v - {\bf{u}}) \otimes (v - {\bf{u}}) f_i^{(1)}\,dv.
\end{align*}
Here we have used the identity derived in Lemma \ref{lem_ident}, and $\mathbb{I}_{n \times n}$ denotes the identity matrix in $\mathbb{R}^{n,n}$.

Similarly, the velocity $u_i$ for species $i$ is expanded as
\begin{align*}
u_i &= \frac{1}{\int_{\mathbb{R}^n} f_i^{(0)}\,dv} \int_{\mathbb{R}^n} v \left( f_i^{(0)} + \varepsilon f_i^{(1)} \right)\,dv   =: u_i^{(0)} + \varepsilon u_i^{(1)},
\end{align*}
with
\begin{equation}\label{u_01}
u_i^{(0)} = \frac{1}{n_i} \int_{\mathbb{R}^n} v \, \mathcal{M}_i[n_i, {\bf{u}}]\,dv = {\bf{u}}, \qquad u_i^{(1)} = \frac{1}{n_i} \int_{\mathbb{R}^n} v f_i^{(1)}\,dv,
\end{equation}
where we again used  the fact that $n_i$ is independent of $\varepsilon$ due to \eqref{compatibility}, and $f_i^{(0)} = \mathcal{M}_i[n_i, {\bf{u}}]$ from \eqref{e-1} and \eqref{u^0}.

Combining the above expressions, the macroscopic system obtained at $O(1)$ is the compressible two-phase isentropic Euler system:
\begin{align}\label{Euler system}
\begin{split}
	&\partial_t n_i + \nabla_x \cdot (n_i {\bf{u}}) = 0, \qquad i=1,2, \\
	&\partial_t \left( \rho {\bf{u}} \right) + \nabla_x \cdot \left( \rho {\bf{u}} \otimes {\bf{u}} \right) + \nabla_x \left( m_1^{\frac{n}{2}(\gamma_1 - 1)} n_1^{\gamma_1} + m_2^{\frac{n}{2}(\gamma_2 - 1)} n_2^{\gamma_2} \right) = 0.
\end{split}
\end{align}


\subsection{First-order expansion: Computation of $f_i^{(1)}$}

We now derive the first-order correction $f_i^{(1)}$ in the Chapman--Enskog expansion. Starting from the kinetic equation \eqref{e eqn}, the term of order $O(1)$ yields the expression
\begin{align}\label{order 0}
f_i^{(1)} =\mathcal{M}_i^{(1)}+ \frac{1}{\nu_i^{(0)}}\left(-\pa_t \mathcal{M}_i^{(0)} - v \cdot \nabla_x \mathcal{M}_i^{(0)}\right).
\end{align}
To compute $f_i^{(1)}$, we begin by evaluating the derivatives of the truncated Maxwellian $\mathcal{M}_i^{(0)} = \mathcal{M}_i[n_i, {\bf{u}}]$. By direct differentiation, we obtain the following formulas:
\[
\pa_{n_i} \mathcal{M}_i^{(0)} = c_i^{\frac{2}{d_i}} d_i m_i^{\frac{n}{2}(\gamma_i - 1)} \gamma_i n_i^{\gamma_i - 2} \mathcal{M}_i^{\frac{d_i - 2}{d_i}}[n_i,{\bf{u}}] \quad \mbox{and} \quad
\nabla_u \mathcal{M}_i^{(0)} = c_i^{\frac{2}{d_i}} m_i d_i (v - {\bf{u}}) \mathcal{M}_i^{\frac{d_i - 2}{d_i}}[n_i,{\bf{u}}].
\]
Substituting the above relations with \eqref{M1} into \eqref{order 0}, we find that $f_i^{(1)}$ can be expressed as
\begin{align}\label{f1}\begin{split}
f_i^{(1)} 
&= d_im_ic_i^{\frac{2}{d_i}}(v-{\bf{u}})\cdot \left( \frac{ \sum_{i=1}^2\nu_i^{(0)} \rho_iu_i^{(1)}}{\sum_{i=1}^2\nu_i^{(0)} \rho_i} \right)\mathcal{M}_i^{\frac{d_i-2}{d_i}}[n_i,{\bf{u}}]\cr 
&-\frac{1}{\nu_i^{(0)}}c_i^{\frac{2}{d_i}}\left\{ m_i^{\frac{n}{2}(\gamma_i - 1)} d_i \gamma_i n_i^{\gamma_i - 2}  (\pa_t n_i + v \cdot \nabla_x n_i) 
+  m_i d_i (v - {\bf{u}})  \cdot (\pa_t {\bf{u}} + v \cdot \nabla_x {\bf{u}})\right\}\mathcal{M}_i^{\frac{d_i - 2}{d_i}}[n_i,{\bf{u}}].
\end{split}\end{align}
To proceed, we collect several useful moment identities related to the truncated Maxwellian. The proofs of these identities are postponed to Appendix \ref{app_vm} for readability.
\begin{lemma}\label{lem_vm} Let $\mathcal{M}_i[n_i, u]$ be the truncated Maxwellian defined by
\[
\mathcal{M}_i[n_i,u](v) = c_i \left( m_i^{\frac{n}{2}(\gamma_i-1)} \frac{2\gamma_i}{\gamma_i-1} n_i^{\gamma_i - 1} - m_i |v-u|^2 \right)_+^{\frac{d_i}{2}},
\]
where $c_i$ is a normalization constant. Then the following identities hold:
\begin{enumerate}
\item[(i)] Zeroth-order moment:
\[
c_i^{\frac{2}{d_i}}\int_{\R^n} \mathcal{M}_i^{\frac{d_i - 2}{d_i}}[n_i, 0]  \,dv = \frac{1}{d_i \gamma_i m_i^{\frac{n}{2}(\gamma_i - 1)}} n_i^{2 - \gamma_i}.
\]
\item[(ii)] Second-order moment:
\[
c_i^{\frac{2}{d_i}}\int_{\R^n} (v  \otimes v) \mathcal{M}_i^{\frac{d_i - 2}{d_i}}[n_i, 0] \,dv = \frac{n_i}{d_i m_i} \mathbb{I}_{n \times n}.
\]
\item[(iii)] Graident-weighted fourth-order moment:
\begin{align*}
&c_i^{\frac{2}{d_i}} \sum_{p,q=1}^n \pa_{x_p} u_q \int_{\R^n} v_p  v_q  (v \otimes v) \mathcal{M}_i^{\frac{d_i - 2}{d_i}}[n_i, 0]  \, dv \cr
&\quad = \frac{n}{n+2}\frac{\gamma_i}{d_i m_i^{2 - \frac n2(\gamma_i-1)}}n_i^{\gamma_i}\left\{(\nabla_x\cdot u) \mathbb{I}_{n\times n} +\nabla_x u+(\nabla_xu)^t \right\}.
\end{align*}
\end{enumerate}
\end{lemma}


\subsubsection*{First-order correction to velocity: $u_i^{(1)}$}
We now derive the first-order correction $u_i^{(1)}$ to the macroscopic velocity field in the Chapman--Enskog expansion. From the macroscopic moment relation \eqref{u_01}, the correction is given by
\[
u_i^{(1)} = \frac{1}{n_i} \int_{\R^n} v f_i^{(1)}\, dv.
\]
Substituting the expression for $f_i^{(1)}$ from \eqref{f1}, we find
\begin{align*}
u_i^{(1)} &=\frac{1}{n_i} \int_{\R^n} v \left\{ d_im_ic_i^{\frac{2}{d_i}}(v-{\bf{u}})\cdot \left( \frac{ \sum_{i=1}^2\nu_i^{(0)} \rho_iu_i^{(1)}}{\sum_{i=1}^2\nu_i^{(0)} \rho_i} \right) \right\}\mathcal{M}_i^{\frac{d_i-2}{d_i}}[n_i,{\bf{u}}]\, dv \cr 
&\quad  -\frac{c_i^{\frac{2}{d_i}}}{n_i\nu_i^{(0)}} \int_{\R^n} v \left\{ 
 m_i^{\frac{n}{2}(\gamma_i - 1)} d_i \gamma_i n_i^{\gamma_i - 2}  (\pa_t n_i + v \cdot \nabla_x n_i) +  m_i d_i (v - {\bf{u}})  \cdot (\pa_t {\bf{u}} + v \cdot \nabla_x {\bf{u}})
\right\}\mathcal{M}_i^{\frac{d_i - 2}{d_i}}[n_i,{\bf{u}}] \,dv \\
&=:I_1+ \frac{1}{\nu_i^{(0)}}\sum_{k=2}^5 I_k.
\end{align*}
By using Lemma \ref{lem_vm}, each term  $I_k$ is computed as follows.
\begin{align*}
I_1&= \left( c_i^{\frac{2}{d_i}}\int_{\R^n} v \otimes  v\mathcal{M}_i^{\frac{d_i-2}{d_i}}[n_i,0]\, dv \right) \frac{d_im_i}{n_i}\left( \frac{ \sum_{i=1}^2\nu_i^{(0)} \rho_iu_i^{(1)}}{\sum_{i=1}^2\nu_i^{(0)} \rho_i} \right)=  \frac{ \sum_{i=1}^2\nu_i^{(0)} \rho_iu_i^{(1)}}{\sum_{i=1}^2\nu_i^{(0)} \rho_i},
\end{align*}
\begin{align*}
I_2&=-   \frac{\pa_tn_i}{n_i}    {\bf{u}},
\end{align*}
\begin{align*}
I_3&=-\frac{1}{n_i}d_i m_i^{\frac n2(\gamma_i - 1)} \gamma_in_i^{\gamma_i-2} \left(c_i^{\frac{2}{d_i}} \intr (v\otimes v)  \mathcal{M}_i^{\frac{d_i-2}{d_i}}[n_i, {\bf{u}}]  \,dv\right) \nabla_x n_i\cr 
&=-\frac{1}{n_i}d_i m_i^{\frac n2(\gamma_i - 1)} \gamma_in_i^{\gamma_i-2} \left(c_i^{\frac{2}{d_i}}\int_{\R^n} (v - {\bf{u}})\otimes (v - {\bf{u}})\mathcal{M}_i^{\frac{d_i - 2}{d_i}}[n_i, {\bf{u}}] \,dv \right.\cr
&\hspace{7cm} \left. + ( {\bf{u}}\otimes {\bf{u}} )c_i^{\frac{2}{d_i}}\int_{\R^n} \mathcal{M}_i^{\frac{d_i - 2}{d_i}}[n_i, {\bf{u}}] \,dv\right) \nabla_x n_i\cr 
&=-\frac{1}{n_i}n_i^{\gamma_i-2}  \left(\gamma_i m_i^{\frac n2(\gamma_i - 1)-1} n_i \mathbb{I}_{n\times n}+ n_i^{2-\gamma_i } ({\bf{u}}\otimes {\bf{u}}) \right)\nabla_x n_i\cr 
&=-\gamma_i m_i^{\frac n2(\gamma_i - 1)-1} n_i^{\gamma_i-2}  \nabla_x n_i-\frac{1}{n_i}{\bf{u}}  ({\bf{u}} \cdot \nabla_x) n_i,
\end{align*}
\begin{align*}
I_4&=-\frac{m_i}{n_i}d_i\pa_t {\bf{u}}\cdot \intr v\otimes (v-{\bf{u}})  c_i^{\frac{2}{d_i}}\mathcal{M}_i^{\frac{d_i-2}{d_i}}[n_i, {\bf{u}}]    \,dv=-\pa_t {\bf{u}},
\end{align*}
and
\begin{align*}
I_5&=-\frac{m_i}{n_i}d_i \intr v   (v-{\bf{u}}) \cdot  \left\{ ( v \cdot \nabla_x) {\bf{u}}  \right\} c_i^{\frac{2}{d_i}}\mathcal{M}_i^{\frac{d_i-2}{d_i}}[n_i, {\bf{u}}]  \, dv\cr 
&=-\frac{m_i}{n_i}d_i  \intr (\{ {\bf{u}}\cdot\nabla_x) {\bf{u}}\} \cdot v)  v   c_i^{\frac{2}{d_i}} \mathcal{M}_i^{\frac{d_i-2}{d_i}}[n_i, 0]  \,dv\cr 
&\quad -\frac{m_i}{n_i}d_i c_i^{\frac 2{d_i}} {\bf{u}} \intr v \cdot \{ (v \cdot \nabla_x) {\bf{u}} \}  \mathcal{M}_i^{\frac{d_i-2}{d_i}}[n_i, 0]  \,dv \cr
&=- ( {\bf{u}}\cdot\nabla_x) {\bf{u}}   -\frac{m_i}{n_i}d_i {\bf{u}}   \frac {(\nabla_x \cdot {\bf{u}})}n\intr |v|^2    c_i^{\frac{2}{d_i}}  \mathcal{M}_i^{\frac{d_i-2}{d_i}}[n_i, 0]  \,dv\cr 
&=- ( {\bf{u}}\cdot\nabla_x) {\bf{u}} - (\nabla_x\cdot {\bf{u}}) {\bf{u}}\cr 
&=-\nabla_x \cdot ({\bf{u}}\otimes {\bf{u}})
\end{align*}
where we used
$$
 \intr (v_\ell)^2    c_i^{\frac{2}{d_i}}  \mathcal{M}_i^{\frac{d_i-2}{d_i}}[n_i, 0]  \,dv =  \frac 1n\intr |v|^2    c_i^{\frac{2}{d_i}}  \mathcal{M}_i^{\frac{d_i-2}{d_i}}[n_i, 0]  \,dv =\frac{n_i}{d_i m_i}
$$
for any $\ell = 1,\dots, n$. 
Summing all five contributions, we get
\begin{align}\label{u_110}
\begin{aligned}
u_i^{(1)}&=\frac{ \sum_{i=1}^2\nu_i^{(0)} \rho_iu_i^{(1)}}{\sum_{i=1}^2\nu_i^{(0)} \rho_i} \cr
&\quad -\frac{1}{n_i\nu_{i}^{(0)}} \left\{     \pa_tn_i  {\bf{u}}+\gamma_im_i^{\frac n2(\gamma_i - 1)-1} n_i^{\gamma_i-1}  \nabla_x n_i + {\bf{u}}  ({\bf{u}} \cdot \nabla_x) n_i +n_i\pa_t {\bf{u}}+n_i\nabla_x\cdot ({\bf{u}}\otimes {\bf{u}}) \right\}\cr 
&=\frac{ \sum_{i=1}^2\nu_i^{(0)} \rho_iu_i^{(1)}}{\sum_{i=1}^2\nu_i^{(0)} \rho_i}-\frac{1}{n_i\nu_{i}^{(0)}}\left\{     \pa_t(n_i {\bf{u}}) +\nabla_x\cdot (n_i {\bf{u}}\otimes {\bf{u}}) + m_i^{\frac n2(\gamma_i - 1)-1}  \nabla_x (n_i^{\gamma_i}) \right\}.
\end{aligned}
\end{align}
From the second relation of \eqref{compatibility} and \eqref{u_01}, we deduce that 
$$
\rho_1u_1^{(1)}+\rho_2u_2^{(1)}=0
$$
which together with \eqref{u_110} gives
the final expression for the first-order velocity correction:
\begin{align}\label{u_11}
\begin{aligned}
u_i^{(1)}&=-\frac{\sum_{k=1}^2\nu_k^{(0)}\rho_k}{\nu_{i}^{(0)}\nu_{j}^{(0)}\rho n_i}\left\{     \pa_t(n_i {\bf{u}}) +\nabla_x\cdot (n_i {\bf{u}}\otimes {\bf{u}}) + m_i^{\frac n2(\gamma_i - 1)-1}  \nabla_x (n_i^{\gamma_i}) \right\},
\end{aligned}
\end{align}
where $i\neq j$.


\subsubsection*{First-order correction to pressure tensor: $P_i^{(1)}$}
We next compute the first-order correction to the pressure tensor $P_i^{(1)}$. This term involves higher-order velocity moments of the kinetic perturbation $f_i^{(1)}$, and plays a crucial role in determining the viscous structure of the resulting macroscopic system. By substituting the expression for $f_i^{(1)}$ into the definition of the pressure tensor, we obtain
\begin{align*}
P_i^{(1)}&=d_im_i^2c_i^{\frac{2}{d_i}}\int_{\mathbb{R}^n} (v - {\bf{u}}) \otimes (v - {\bf{u}}) \left\{ (v-{\bf{u}})\cdot \left( \frac{ \sum_{i=1}^2\nu_i^{(0)} \rho_iu_i^{(1)}}{\sum_{i=1}^2\nu_i^{(0)} \rho_i} \right)\right\}\mathcal{M}_i^{\frac{d_i-2}{d_i}}[n_i,{\bf{u}}]
\,dv\cr 
&\quad -\frac{d_i m_i^{\frac n2(\gamma_i - 1) + 1} \gamma_in_i^{\gamma_i-2}c_i^{\frac{2}{d_i}}}{\nu_i^{(0)}}\intr (v-{\bf{u}})\otimes (v-{\bf{u}})   \mathcal{M}_i^{\frac{d_i-2}{d_i}}[n_i, {\bf{u}}]\left( \pa_tn_i+v\cdot \nabla_x n_i \right)  dv \cr 
&\quad -\frac{d_i m_i^2 c_i^{\frac{2}{d_i}}}{\nu_i^{(0)}}\intr (v-{\bf{u}})\otimes (v-{\bf{u}})      \mathcal{M}_i^{\frac{d_i-2}{d_i}}[n_i, {\bf{u}}] \, \{(v-{\bf{u}})\cdot \left(\pa_t {\bf{u}} + (v \cdot \nabla_x) {\bf{u}}\right)\}   \,dv.
\end{align*}
We compute each of the three terms separately. The first term vanishes due to the oddness: 
\begin{align*}
&c_i^{\frac{2}{d_i}}\int_{\mathbb{R}^n} (v - {\bf{u}}) \otimes (v - {\bf{u}}) \left\{ (v-{\bf{u}})\cdot \left( \frac{ \sum_{i=1}^2\nu_i^{(0)} \rho_iu_i^{(1)}}{\sum_{i=1}^2\nu_i^{(0)} \rho_i} \right)\right\}\mathcal{M}_i^{\frac{d_i-2}{d_i}}[n_i,{\bf{u}}]
\,dv\cr 
&\quad =c_i^{\frac{2}{d_i}}\int_{\mathbb{R}^n} v \otimes v \left\{ v\cdot \left( \frac{ \sum_{i=1}^2\nu_i^{(0)} \rho_iu_i^{(1)}}{\sum_{i=1}^2\nu_i^{(0)} \rho_i} \right)\right\}\mathcal{M}_i^{\frac{d_i-2}{d_i}}[n_i,0]
\,dv\cr 
&\quad =0.
\end{align*}
The second integral can be simplified using the change of variables and the symmetry of the Maxwellian:
\begin{align*}
&c_i^{\frac 2{d_i}} \intr (v-{\bf{u}})\otimes (v-{\bf{u}})   \mathcal{M}_i^{\frac{d_i-2}{d_i}}[n_i, {\bf{u}}]   (\pa_tn_i+v\cdot\nabla_xn_i)\,dv\cr 
&\quad =c_i^{\frac 2{d_i}} \intr (v\otimes v) \mathcal{M}_i^{\frac{d_i-2}{d_i}}[n_i, 0]  \left\{\pa_tn_i+ (v+{\bf{u}})\cdot\nabla_xn_i\right\}\,dv\cr 
&\quad =(\pa_tn_i+{\bf{u}}\cdot\nabla_xn_i) \lt(\frac{c_i^{\frac 2{d_i}}}{n}\intr |v|^2   \mathcal{M}_i^{\frac{d_i-2}{d_i}}[n_i, 0]  \,dv \rt)   \mathbb{I}_{n\times n}\cr 
&\quad =\frac{n_i}{d_i m_i}(\pa_tn_i+{\bf{u}}\cdot\nabla_x n_i) \mathbb{I}_{n\times n}.
\end{align*}
For the third term, we use the moment identity in Lemma \ref{lem_vm} (iii):
\begin{align*}
&c_i^{\frac 2{d_i}}\intr (v-{\bf{u}})\otimes (v-{\bf{u}})\mathcal{M}_i^{\frac{d_i-2}{d_i}}[n_i, {\bf{u}}]    \lt\{ (v-{\bf{u}})\cdot \left(\pa_t {\bf{u}}+ (v \cdot \nabla_x) {\bf{u}}\right)\rt\}\,dv\cr 
&\quad =c_i^{\frac 2{d_i}}\intr v\otimes v \mathcal{M}_i^{\frac{d_i-2}{d_i}}[n_i, 0]  \lt\{ v \cdot \left(\pa_t {\bf{u}}+ ( (v+{\bf{u}}) \cdot \nabla_x) {\bf{u}}\right)\rt\}\,dv\cr 
&\quad =c_i^{\frac 2{d_i}}\intr v\otimes v \mathcal{M}_i^{\frac{d_i-2}{d_i}}[n_i, 0]  \lt\{ v \cdot \left(  ( v \cdot \nabla_x) {\bf{u}}\right)\rt\}\,dv\cr 
&\quad = c_i^{\frac{2}{d_i}} \sum_{p,q=1}^n \pa_{x_p} {\bf{u}}_q \int_{\R^n} v_p  v_q  (v \otimes v) \mathcal{M}_i^{\frac{d_i - 2}{d_i}}[n_i, 0]  \, dv \cr
&\quad =\frac{n}{n+2}\frac{\gamma_i}{d_i m_i^{2 - \frac n2(\gamma_i-1)}}n_i^{\gamma_i}\left\{(\nabla_x\cdot {\bf{u}}) \mathbb{I}_{n\times n} +\nabla_x {\bf{u}}+(\nabla_x{\bf{u}})^t \right\}.
\end{align*}
Combining the three terms, we obtain the explicit expression for $P_i^{(1)}$ as follows.
\begin{align*}
P_i^{(1)}&=- \frac{m_i^{\frac n2(\gamma_i - 1)}  }{\nu_i^{(0)}} \gamma_i(\pa_tn_i+{\bf{u}}\cdot\nabla_xn_i)n_i^{\gamma_i-1} \mathbb{I}_{n\times n} -\frac{m_i^{\frac n2(\gamma_i - 1)}}{\nu_i^{(0)}}\frac{n}{n+2} \gamma_i\left\{(\nabla_x\cdot {\bf{u}}) \mathbb{I}_{n\times n} +\nabla_x {\bf{u}}+(\nabla_x{\bf{u}})^t \right\}n_i^{\gamma_i}.  
\end{align*}
Using the continuity equation and the identities \eqref{u_01} and \eqref{u_11}, we can rewrite this expression as
\begin{align*}
P_i^{(1)}&=- \frac{m_i^{\frac n2(\gamma_i - 1)}  }{\nu_i^{(0)}}  \frac{n}{n+2}\gamma_i\left\{ (\nabla_x\cdot {\bf{u}}) \mathbb{I}_{n\times n} +\nabla_x {\bf{u}}+(\nabla_x{\bf{u}})^t\right\}n_i^{\gamma_i} \cr
&\quad +  \frac{m_i^{\frac n2(\gamma_i - 1)}  }{\nu_i^{(0)}} \gamma_i \left\{ n_i\nabla_x\cdot {\bf{u}} +\nabla_x\cdot (n_i({\bf{u}}_i-{\bf{u}}))\right\}n_i^{\gamma_i-1}\mathbb{I}_{n\times n}\cr 
&=-\frac{m_i^{\frac n2(\gamma_i - 1)}  }{\nu_i^{(0)}}  \frac{n}{n+2}\gamma_i\left\{ (\nabla_x\cdot {\bf{u}}) \mathbb{I}_{n\times n} +\nabla_x {\bf{u}}+(\nabla_x{\bf{u}})^t\right\}n_i^{\gamma_i} \cr 
&\quad +\frac{m_i^{\frac n2(\gamma_i - 1)}  }{\nu_i^{(0)}}  \gamma_i \Bigg[ n_i\nabla_x\cdot {\bf{u}} -\varepsilon\nabla_x\cdot \frac{\sum_{k=1}^2\nu_k^{(0)}\rho_k}{\nu_{i}^{(0)}\nu_{j}^{(0)}\rho } \cr
&\hspace{5.5cm} \times \left\{     \pa_t(n_i {\bf{u}}) +\nabla_x\cdot (n_i {\bf{u}}\otimes {\bf{u}}) + m_i^{\frac n2(\gamma_i - 1)-1}\nabla_x (n_i^{\gamma_i}) \right\}\Bigg]n_i^{\gamma_i-1} \mathbb{I}_{n\times n} .
\end{align*} 
Thus, the system up to $O(\varepsilon)$, \eqref{balance} together with \eqref{u_01} and \eqref{u_11} is given by
\begin{align*}
&\pa_t n_i+\nabla_x\cdot (n_i {\bf{u}})=\varepsilon\nabla_x \cdot  \left\{\frac{\sum_{k=1}^2\nu_k^{(0)}\rho_k}{\nu_{i}^{(0)}\nu_{j}^{(0)}\rho }\left(    \pa_t(n_i {\bf{u}}) +\nabla_x\cdot (n_i {\bf{u}}\otimes {\bf{u}}) + m_i^{\frac n2(\gamma_i - 1)-1}  \nabla_x (n_i^{\gamma_i}) \right) \right\}
\end{align*}
for $i=1,2$, and
\begin{align*}
&\pa_t\{\rho {\bf{u}}\}+ \nabla_x\cdot \left\{\rho  {\bf{u}}\otimes {\bf{u}} \right\}+\nabla_x \left( m_1^{\frac{n}{2}(\gamma_1-1)}n_1^{\gamma_1}+m_2^{\frac{n}{2}(\gamma_2-1)}n_2^{\gamma_2} \right)\cr 
&\quad -\frac{n}{n+2}\nabla_x \cdot \lt[ \left\{ (\nabla_x\cdot {\bf{u}}) \mathbb{I}_{n\times n} +\nabla_x {\bf{u}}+(\nabla_x{\bf{u}})^t\right\}\left(\e_1\gamma_1\frac{m_1^{\frac n2(\gamma_1 - 1)}  }{\nu_1^{(0)}} n_1^{\gamma_1}+\e_2\gamma_2\frac{m_2^{\frac n2(\gamma_2 - 1)}  }{\nu_2^{(0)}} n_2^{\gamma_2}\right)\rt]\cr
&\qquad =-\nabla_x \left\{\left( \e_1\gamma_1 \frac{m_1^{\frac n2(\gamma_1 - 1)}  }{\nu_1^{(0)}} n_1^{\gamma_1}+\e_2 \gamma_2 \frac{m_2^{\frac n2(\gamma_2 - 1)}  }{\nu_2^{(0)}} n_2^{\gamma_2} \right) \nabla_x\cdot {\bf{u}}  \right\}. 
\end{align*}
To express this in a standard viscous form, we define the Lam\'e viscosity coefficients
\[
\mu_i(n_i) = \frac{n }{n+2}\frac{\gamma_i}{\nu_i^{(0)}}  n_i^{\gamma_i} \quad \mbox{and} \quad \lambda_i(n_i) = -\frac{2 }{n+2}\frac{\gamma_i}{\nu_i^{(0)}} n_i^{\gamma_i}, \quad i =1,2,
\]
and the symmetric strain tensor
\[
\T(\nabla_x {\bf{u}}) := \frac12 (\nabla_x {\bf{u}} + (\nabla_x {\bf{u}})^t).
\]
The corresponding stress tensor is given by
\[
\mathbb{S}_i(\nabla_x {\bf{u}}):= 2\mu_i(n_i) \T(\nabla_x {\bf{u}}) + \lambda_i (n_i) (\nabla_x \cdot {\bf{u}}) \mathbb{I}_{n\times n}.
\]
Note that $\mu_i \geq 0$ and $2\mu_i + n \lambda_i = 0$ for $i=1,2$, ensuring consistency with the Stokes hypothesis. The momentum equation now takes the form
\begin{align*}
&\pa_t\{\rho {\bf{u}}\}+ \nabla_x\cdot \left\{\rho {\bf{u}}\otimes {\bf{u}} \right\}+\nabla_x \left( m_1^{\frac{n}{2}(\gamma_1-1)}n_1^{\gamma_1}+m_2^{\frac{n}{2}(\gamma_2-1)}n_2^{\gamma_2} \right)\cr 
&\quad - \nabla_x \cdot \lt(\e_1 m_1^{\frac n2(\gamma_1 - 1)}\mathbb{S}_1(\nabla_x {\bf{u}}) + \e_2  m_2^{\frac n2(\gamma_2 - 1)}  \mathbb{S}_2(\nabla_x {\bf{u}})\rt) = 0.
\end{align*}
In conclusion, we arrive at
\begin{align*}
&\pa_t n_i+\nabla_x\cdot (n_i {\bf{u}})=\varepsilon\nabla_x \cdot  \frac{\sum_{k=1}^2\nu_k^{(0)}\rho_k}{\nu_{i}^{(0)}\nu_{j}^{(0)}\rho }\left\{     \pa_t(n_i {\bf{u}}) +\nabla_x\cdot (n_i {\bf{u}}\otimes {\bf{u}}) +  m_i^{\frac n2(\gamma_i - 1)-1}\nabla_x (n_i^{\gamma_i}) \right\},\qquad i=1,2, \cr
&\pa_t\{\rho {\bf{u}}\}+ \nabla_x\cdot \left\{\rho {\bf{u}}\otimes {\bf{u}} \right\}+\nabla_x \left( m_1^{\frac{n}{2}(\gamma_1-1)}n_1^{\gamma_1}+m_2^{\frac{n}{2}(\gamma_2-1)}n_2^{\gamma_2} \right)\cr 
&\quad - \nabla_x \cdot \lt(\e_1 m_1^{\frac n2(\gamma_1 - 1)}\mathbb{S}_1(\nabla_x {\bf{u}}) + \e_2  m_2^{\frac n2(\gamma_2 - 1)}  \mathbb{S}_2(\nabla_x {\bf{u}})\rt) = 0.
\end{align*}
 Here $\rho_i=m_in_i$ and $\rho=\rho_1+\rho_2$, thus the above system is closed.
 
 \begin{remark} If we assume $\nu^{(0)}=\nu_1^{(0)}=\nu_2^{(0)}$, then we further get
$$
\pa_t n_i+\nabla_x\cdot (n_i {\bf{u}})=\varepsilon\nabla_x \cdot  \frac{1}{\nu^{(0)}}\left\{     \pa_t(n_i {\bf{u}}) +\nabla_x\cdot (n_i {\bf{u}}\otimes {\bf{u}}) +  m_i^{\frac n2(\gamma_i - 1)-1}\nabla_x (n_i^{\gamma_i}) \right\},\quad i=1,2.
$$
\end{remark}

\subsection{Verification of compatibility conditions}

To ensure that the Chapman--Enskog expansion preserves the conservation of mass and momentum at each order, we verify the compatibility conditions:
\[
\intr f_i^{(1)}\,dv=0,\quad \intr v\left(m_1 f_1^{(1)} + m_2f_2^{(1)} \right) dv=0.
\]
We begin by verifying the first condition for mass conservation. Using the expression for $f_i^{(1)}$ obtained earlier, we compute
\begin{align*}
\intr f_i^{(1)}\,dv&= \intr d_im_ic_i^{\frac{2}{d_i}}(v-{\bf{u}})\cdot \left( \frac{ \sum_{i=1}^2\nu_i^{(0)} \rho_iu_i^{(1)}}{\sum_{i=1}^2\nu_i^{(0)} \rho_i} \right)\mathcal{M}_i^{\frac{d_i-2}{d_i}}[n_i,{\bf{u}}] \,dv\cr 
&\quad -\frac{d_i\gamma_i m_i^{\frac{n}{2}(\gamma_i-1)} n_i^{\gamma_i-2}}{\nu_i^{(0)}}c_i^{\frac{2}{d_i}}\intr \left( \pa_tn_i+v\cdot \nabla_x n_i \right) \mathcal{M}_i^{\frac{d_i-2}{d_i}}[n_i, {\bf{u}}]\,dv \cr
&\quad -  \frac{d_i m_i}{\nu_i^{(0)}}c_i^{\frac{2}{d_i}}\intr   (v-{\bf{u}}) \cdot \left(\pa_t {\bf{u}} + (v \cdot \nabla_x) {\bf{u}}\right)\mathcal{M}_i^{\frac{d_i-2}{d_i}}[n_i, {\bf{u}}]  \,dv\cr
&=- \frac{1}{\nu_i^{(0)}} \left( \pa_tn_i+{\bf{u}}\cdot \nabla_x n_i \right) - \frac{d_i m_i}{\nu_i^{(0)}} c_i^{\frac{2}{d_i}}\intr   v \cdot \left(  (v \cdot \nabla_x) {\bf{u}}\right)\mathcal{M}_i^{\frac{d_i-2}{d_i}}[n_i, 0]  \,dv  \cr 
&=-  \frac{1}{\nu_i^{(0)}}\left( \pa_tn_i+{\bf{u}}\cdot \nabla_x n_i \right) - \frac{n_i}{\nu_i^{(0)}}\nabla_x\cdot {\bf{u}}\cr 
&=-\frac{1}{\nu_i^{(0)}}\left\{\pa_t n_i +\nabla_x\cdot(n_i{\bf{u}}) \right\}\cr 
&=\frac{1}{\nu_i^{(0)}}\nabla_x \cdot \{n_i(u_i-{\bf{u}})\},
\end{align*}
and hence it vanishes when evaluated under the zeroth-order relation $u_i = {\bf{u}}$.

Next, we check the condition for momentum conservation:
\begin{align*}
\intr v f_i^{(1)} \,dv&=n_i u_i^{(1)}  =-\frac{\sum_{k=1}^2\nu_k^{(0)}\rho_k}{\nu_{i}^{(0)}\nu_{j}^{(0)}\rho }\left\{     \pa_t(n_i {\bf{u}}) +\nabla_x\cdot (n_i {\bf{u}}\otimes {\bf{u}}) + m_i^{\frac n2(\gamma_i - 1)-1} \nabla_x (n_i^{\gamma_i}) \right\}
\end{align*}
where $j\neq i$.
Summing over $i=1,2$ with weights $m_1$ and $m_2$ gives
\begin{align*}
&\intr v\left(m_1 f_1^{(1)} + m_2f_2^{(1)} \right) dv\cr
&\quad =-\frac{\sum_{k=1}^2\nu_k^{(0)}\rho_k}{\nu_{i}^{(0)}\nu_{j}^{(0)}\rho }\left\{ \pa_t\{\rho {\bf{u}}\} +\nabla_x\cdot \{\rho {\bf{u}}\otimes {\bf{u}}\} + \nabla_x (m_1^{\frac n2(\gamma_1 - 1)} n_1^{\gamma_1}+m_1^{\frac n2(\gamma_2 - 1)}n_2^{\gamma_2}) \right\}.
\end{align*}
Again, this involves the left-hand side of the momentum balance in \eqref{balance}, which vanishes at leading order. Thus, the momentum compatibility condition is satisfied. Therefore, both compatibility conditions \eqref{compatibility} are explicitly verified at the level of first-order corrections, confirming the consistency of the Chapman--Enskog expansion with the underlying conservation laws.

Having derived the compressible two-phase Euler system at leading order and verified that the first-order correction terms preserve both mass and momentum, we now summarize the outcome of the Chapman--Enskog expansion in the following theorem.

 \begin{theorem}[Compressible two-phase fluid limit via Chapman--Enskog expansion]\label{thm:formal}
Consider the kinetic BGK model for a two-species mixture \eqref{CE main}. Suppose that the distribution functions admit the Chapman--Enskog expansion
\[
f_i = \mathcal{M}_i[n_i, {\bf{u}}] + \varepsilon f_i^{(1)},
\]
and the macroscopic quantities $n_i$ and ${\bf{u}}$ are independent of $\varepsilon$. Then the following statements hold:
\begin{enumerate}
    \item[(i)] The zeroth-order macroscopic equations yield the compressible two-phase Euler system:
    \begin{align*}
    &\partial_t n_i + \nabla_x \cdot (n_i {\bf{u}}) = 0, \quad i=1,2, \\
    &\partial_t \left( \rho  {\bf{u}} \right) + \nabla_x \cdot \left( \rho {\bf{u}} \otimes {\bf{u}} \right) + \nabla_x \left( m_1^{\frac{n}{2}(\gamma_1 - 1)} n_1^{\gamma_1} + m_2^{\frac{n}{2}(\gamma_2 - 1)} n_2^{\gamma_2} \right) = 0.
    \end{align*}

    \item[(ii)] At first order in $\varepsilon$, the system closes to a compressible two-phase Navier--Stokes-type system with density-dependent viscosity:
\begin{align*}
    &\partial_t n_i + \nabla_x \cdot (n_i {\bf{u}}) = \varepsilon \nabla_x \cdot \frac{\sum_{k=1}^2\nu_k^{(0)}\rho_k}{\nu_{i}^{(0)}\nu_{j}^{(0)}\rho }\left\{ \partial_t (n_i {\bf{u}}) + \nabla_x \cdot (n_i {\bf{u}} \otimes {\bf{u}}) +m_i^{\frac n2(\gamma_i - 1)-1} \nabla_x (n_i^{\gamma_i}) \right\},  \\
    &\partial_t \left( \rho  {\bf{u}} \right) + \nabla_x \cdot \left( \rho  {\bf{u}} \otimes {\bf{u}} \right) + \nabla_x (p_1 + p_2)   = \nabla_x \cdot \left( \varepsilon_1 m_1^{\frac{n}{2}(\gamma_1 - 1)} \mathbb{S}_1 + \varepsilon_2 m_2^{\frac{n}{2}(\gamma_2 - 1)} \mathbb{S}_2 \right),
    \end{align*}
    where $p_i = m_i^{\frac{n}{2}(\gamma_i - 1)} n_i^{\gamma_i}$ and the viscosity stress tensors $\mathbb{S}_i$ are defined by
    \[
    \mathbb{S}_i := 2 \mu_i(n_i) \mathbb{T}(\nabla_x u) + \lambda_i(n_i) (\nabla_x \cdot u) \mathbb{I}_{n \times n}, \quad
    \mathbb{T}(\nabla_x u) := \frac{1}{2}(\nabla_x u + (\nabla_x u)^T)
    \]
    with Lam\'e viscosity coefficients:
    \[
    \mu_i(n_i) = \frac{n }{n+2}\frac{\gamma_i}{\nu_i^{(0)}} n_i^{\gamma_i}, \qquad
    \lambda_i(n_i) = -\frac{2 }{n+2} \frac{\gamma_i}{\nu_i^{(0)}}n_i^{\gamma_i}.
    \]

    \item[(iii)] The first-order correction $f_i^{(1)}$ satisfies the compatibility conditions:
    \[
    \int_{\mathbb{R}^n} f_i^{(1)}(v)\,dv = 0, \qquad
    \int_{\mathbb{R}^n} v \left( m_1 f_1^{(1)} + m_2 f_2^{(1)} \right)(v) \, dv = 0,
    \]
    ensuring that both mass and total momentum are preserved up to first order.
\end{enumerate}
\end{theorem}


\section{Rigorous derivation: Relative entropy method}\label{sec_re}

In this section, we rigorously derive the isentropic two-phase Euler equations from the BGK-type kinetic model \eqref{CE main}. For clarity and without loss of generality, we set the collision frequencies to unity, i.e., $\nu_1 = \nu_2 = 1$.

We begin by recalling the form of the limiting system. The two-phase compressible Euler system reads
\begin{equation}\label{eq: 2-Euler}
\begin{split}
    &\p_t n_i + \nabla_x \cdot (n_i {\bf{u}}) = 0,\\
    &\p_t (\rho {\bf{u}}) + \nabla_x\cdot (\rho {\bf{u}}\otimes {\bf{u}}) + \nabla_x((\gamma_1-1) s_1(n_1) + (\gamma_2 - 1) s_2(n_2)) = 0,\\
    &s_i(n_i) := \frac{m_i^{\frac{n}{2}(\gamma_i - 1)}}{\gamma_i - 1} n_i^{\gamma_i},\\
    &\rho_i := m_i n_i, \quad  \rho := \rho_1 + \rho_2.
\end{split}
\end{equation}

To express this system compactly, we introduce the macroscopic observables
\[
    U := (n_1 \; n_2 \; w)^t, \quad w := \rho {\bf{u}} ,
\]
so that the system \eqref{eq: 2-Euler} can be written as a system of conservation laws:
\begin{align*}
    \p_t U + \nabla_x \cdot A(U) = 0, \quad A(U) := \begin{pmatrix}
        \frac{n_1}{\rho} w^t \\[6pt]
        \frac{n_2}{\rho}w^t \\[6pt]
        \frac{w\otimes w}{\rho} + ((\gamma_1-1)s_1(n_1) + (\gamma_2-1)s_2(n_2)) \mathbb{I}_{n\times n}
    \end{pmatrix}.
\end{align*}
This system admits the following entropy-entropy flux pair:
\begin{align*}
    \eta(U) := \frac{1}{2} \frac{|w|^2}{\rho} + s_1(n_1) + s_2(n_2),  \quad G(U) := u \left(\frac{1}{2}\frac{|w|^2}{\rho} + \gamma_1 s_1(n_1) + \gamma_2 s_2(n_2) \right).
\end{align*}

To measure the discrepancy between two states, we define the relative entropy associated with a convex functional $\eta$ as
\[
    \eta(V|U) := \eta(V) - \eta(U) - D\eta(U)(V-U).
\]
In the sequel, we denote the relative entropy associated with $\eta$ by $\eta(V|U)$. 

The following theorem provides a quantitative convergence estimate for the hydrodynamic limit of the kinetic model \eqref{CE main}. It shows that, under suitable regularity assumptions and well-prepared initial data, the macroscopic observables of the kinetic solution converge to those of the Euler system with explicit rates.
\begin{theorem}\label{thm: relent}
    Let $s>\frac n2 + 1$, $1 < \gamma_i \leq 2$, and suppose $(n_1,n_2,{\bf{u}})\in L^\infty(0,T^*;H^s(\Omega)\times H^s(\Omega)\times H^s(\Omega))$ is a unique strong solution to \eqref{eq: 2-Euler} for some $T^*>0$, with initial data $({}^{\rm in}n_1, {}^{\rm in}n_2, {}^{\rm in}{\bf{u}})$ satisfying
    \begin{align*}
        \int_\Omega \left(\frac{m_1 {}^{\rm in}n_1 + m_2{}^{\rm in}n_2}{2} |{}^{\rm in}{\bf{u}}|^2 + s_1({}^{\rm in}n_1) + s_2({}^{\rm in}n_2) + {}^{\rm in}n_1 + {}^{\rm in}n_2 \right)\,dx < +\infty.
    \end{align*}
    Assume further that there exist constants $\lambda,\Lambda>0$  such that
\bq\label{bdd_rho}
        \lambda \le n_1(t,x), \,n_2(t,x) \le \Lambda \qquad \forall \, (t,x)\in [0,T^*]\times \Omega.
\eq
    Let $\{f^\ve\} = \{f_1^\ve, f_2^\ve\}_{\ve>0}$ be a family of weak solutions to \eqref{CE main}, with initial data $\{{}^{\rm in}f^{\ve}\}$ satisfying
    \begin{align*}
        &\iint_{\Omega \times \R^n} {}^{\rm in}f^\ve_1 + {}^{\rm in}f_2^\ve\,dxdv < +\infty, \\
        &\iint_{\Omega \times \R^n} h_1({}^{\rm in}f_1^\ve,v) + h_2({}^{\rm in}f_2^\ve,v)\, dxdv < +\infty.
    \end{align*}
Suppose further that the weak solutions satisfy the entropy inequality
    \begin{align*}
        \iint_{\Omega \times \R^n} H(f^\ve,v)\,dv + \frac{1}{\ve} \sum_{i=1}^2 \int_0^t \iint_{\Omega \times \R^n} h_i(f_i^\ve,v) - h_i(\calM_i(f^\ve),v)\,dxdvds \le \iint_{\Omega \times \R^n} H({}^{\rm in}f^\ve,v)\,dxdv.
    \end{align*}
Let the macroscopic observables be defined by    
\[
U^\ve := (n_1^\ve\; n_2^\ve\; w^\ve)^t, \quad w^\ve = \rho^\ve {\bf{u}}^\ve, \quad \rho^\ve:= \rho_1^\ve + \rho_2^\ve.
\]
If the initial data are well-prepared in the sense that
\[
        \iint_{\Omega \times \R^n} H({}^{\rm in}f^\ve) - H(\calM_1({}^{\rm in}f^\ve),\calM_2({}^{\rm in}f^\ve))\,dxdv \lesssim \ve^{1/4},
\]
    then the relative entropy satisfies, for $t\in [0,T^*]$, the estimate
\[
        \int_{\Omega} \eta(U^\ve | U)\,dx \lesssim \ve^{1/4}.
\]
    In particular, we have the following quantitative convergences:
    \begin{equation}\label{eq: hyd}
    \begin{split}
        \|n_i^\ve - n_i\|_{L^\infty(0,T^*;L^{\gamma_i})} \lesssim \ve^{1/8},\\
        \|(\rho_1^\ve + \rho_2^\ve)  {\bf{u}}^\ve - (\rho_1 + \rho_2) {\bf{u}}\|_{L^\infty(0,T^*;L^1)} \lesssim \ve^{1/8}, \\
        \|(\rho_1^\ve + \rho_2^\ve){\bf{u}}^\ve\otimes {\bf{u}}^\ve - (\rho_1 + \rho_2) {\bf{u}}\otimes {\bf{u}}\|_{L^\infty(0,T^*;L^1)} \lesssim \ve^{1/8}.
        \end{split}
    \end{equation}
\end{theorem}

\begin{remark} The local-in-time existence and uniqueness of strong solutions to the system \eqref{eq: 2-Euler} satisfying the uniform bounds in \eqref{bdd_rho} can be established under appropriate assumptions on the initial data, see \cite{D79} and \cite[Proposition 1.3]{MV08}. On the other hand, the global-in-time existence of global solutions to the kinetic model \eqref{CE main} that satisfy the required entropy inequality has not been studied before, to our best knowledge. For completeness, we therefore provide the corresponding existence theory in Appendix \ref{app_gwp}.  
\end{remark}

\subsection{Moment equations and setup for relative entropy analysis}

Our starting step is to write the conservation laws satisfied by the velocity moments of the kinetic equation \eqref{CE main}. Following the balance identities and using Lemma \ref{lem_ident}, we obtain
\begin{align*}
&\p_t n_i^\ve + \nabla_x \cdot (n_i^\ve {\bf{u}}_i^\ve) = 0, \quad i=1,2,\\
 &\p_t (n_i^\ve {\bf{u}}_i^\ve) + \nabla_x\cdot \lt(n_i^\ve {\bf{u}}^\ve \otimes {\bf{u}}^\ve + m_i^{\frac{n}{2}(\gamma_i-1)-1 } n_i^{\gamma_i} \mathbb{I}_{n\times n}\rt) \cr
 &\quad = \nabla_x \cdot \left(\intr v\otimes v \lt( \calM_1(f_1^\ve,f_2^\ve) - f_1^\ve  \rt) dv \right)  + \frac{1}{\ve} n_i^\ve ({\bf{u}}^\ve - {\bf{u}}_i^\ve).
\end{align*}
In particular, multiplying the momentum equations by $m_i$ and summing over $i$, the singular relaxation terms cancel out. This yields a conservation law for $U^\varepsilon$ of the form
\[
    \p_t U^\ve + \nabla_x\cdot A(U^\ve) = \nabla_x\cdot r(U^\ve),
\]
where the remainder term is
\[
    r(U^\ve) = \begin{pmatrix} n_1^\ve({\bf{u}}^\ve - {\bf{u}}_1^\ve) \\  n_2^\ve({\bf{u}}^\ve - {\bf{u}}_2^\ve) \\ \sum_{i=1}^2 \intr v\otimes v (\calM_i - f_i^\ve)\,dv \end{pmatrix} = \begin{pmatrix}
        \intr v( \calM_1 - f_1^\ve)\,dv\\
        \intr v( \calM_2 - f_2^\ve)\,dv\\
        \sum_{i=1}^2 \intr v\otimes v (\calM_i - f_i^\ve)\,dv
    \end{pmatrix}.
\]

To proceed with the relative entropy framework, we now compute key quantities. The gradient and Hessian of the entropy functional $\eta$ are given by
\begin{align*}
    D\eta(U) &= \begin{pmatrix} -\frac{1}{2}\frac{m_1|w|^2}{\rho^2} + s_1'(n_1) \\[6pt]
    -\frac{1}{2}\frac{m_2|w|^2}{\rho^2} + s_2'(n_2) \\[6pt]
    \frac{w}{\rho}
    \end{pmatrix}= \begin{pmatrix} -\frac{m_1}{2}|{\bf{u}}|^2 + s_1'(n_1) \\[6pt]
     -\frac{m_2}{2}|{\bf{u}}|^2 + s_2'(n_2) \\[6pt]
   {\bf{u}}
    \end{pmatrix}
    \end{align*}
    and
    \begin{align*}
    D^2\eta(U) &=  \begin{pmatrix}
        \frac{m_1^2|w|^2}{\rho^3} + s_1''(n_1) & \frac{m_1m_2|w|^2}{\rho^3}  & -\frac{m_1 w}{\rho^2}  \\[6pt]
        \frac{m_1 m_2|w|^2}{\rho^3}& \frac{m_2^2|w|^2}{\rho^3} + s_2''(n_2) & - \frac{m_2 w}{\rho^2} \\[6pt]
        - \frac{m_1 w}{\rho^2} & - \frac{m_2 w}{\rho^2} & \frac{1}{\rho} \mathbb{I}_{n\times n}
    \end{pmatrix}  =  \begin{pmatrix}
        \frac{m_1^2|{\bf{u}}|^2}{\rho} + s_1''(n_1) & \frac{m_1m_2|{\bf{u}}|^2}{\rho}  & -\frac{m_1 {\bf{u}}}{\rho}  \\[6pt]
        \frac{m_1 m_2|{\bf{u}}|^2}{\rho}& \frac{m_2^2|{\bf{u}}|^2}{\rho} + s_2''(n_2) & - \frac{m_2 {\bf{u}}}{\rho} \\[6pt]
        - \frac{m_1 {\bf{u}}}{\rho} & - \frac{m_2 {\bf{u}}}{\rho} & \frac{1}{\rho} \mathbb{I}_{n\times n}
    \end{pmatrix}.
\end{align*}
It follows that the relative entropy takes the explicit form
\[
    \eta(U^\ve|U) = \frac{\rho^\ve}{2}|{\bf{u}}^\ve - {\bf{u}}|^2 + s_1(n_1^\ve|n_1) + s_2(n_2^\ve|n_2).
\]
We next turn to the flux $A(U)$. The relative flux is defined component-wise by
\[
    [A(U^\ve|U)]_i := A(U^\ve) - A(U) - DA_i(U)(U^\ve - U).
\]

To prepare for the stability analysis, we now compute the relative flux in a form that allows precise control of the nonlinear error terms. This quantity will play a central role in the subsequent relative entropy estimate.
\begin{lemma}
    Let $e_i = (0,\ldots, 1,\ldots, 0)\in \R^n$ be the standard basis vector with 1 at the $i$th entry and zeros elsewhere. Then for each $1 \le i \le n$, the $i$-th component of the relative flux satisfies
        \begin{align*}
         [A(U^\ve|U)]_i  &= \begin{pmatrix}
            \left[\frac{n_1}{\rho} - \frac{n_1^\ve}{\rho^\ve } \right] \rho^\ve  ({\bf{u}}_i -{\bf{u}}_i^\ve) \\[6pt]
            \left[\frac{n_2}{\rho} - \frac{n_2^\ve}{\rho^\ve} \right] \rho^\ve  ({\bf{u}}_i-{\bf{u}}_i^\ve) \\[6pt]
            \rho^\ve({\bf{u}}^\ve-{\bf{u}})({\bf{u}}_i^\ve - {\bf{u}}_i) + (\gamma_1-1)[s_1(n_1^\ve|n_1)]e_i + (\gamma_2-1)[s_2(n_2^\ve|n_2)]e_i
        \end{pmatrix} \cr
        &= \begin{pmatrix}
            \left[\frac{n_1}{\rho} - \frac{n_1^\ve}{\rho^\e} \right]  \rho^\e  ({\bf{u}}_i -{\bf{u}}_i^\ve) \\[6pt]
            \left[\frac{n_2}{\rho} - \frac{n_2^\ve}{\rho^\ve} \right]  \rho^\ve  ({\bf{u}}_i-{\bf{u}}_i^\ve) \\[6pt]
             \rho^\ve ({\bf{u}}^\ve-{\bf{u}})({\bf{u}}_i^\ve - {\bf{u}}_i) + (\gamma_1-1)[s_1(n_1^\ve|n_1)]e_i + (\gamma_2-1)[s_2(n_2^\ve|n_2)]e_i
        \end{pmatrix}.
    \end{align*}
\end{lemma}
\begin{proof}
Since the result follows from direct computation, we only outline the relevant quantities. First, compute the Jacobian matrix $DA_i(U)$ of the $i$th flux component:
    \begin{align*}
        DA_i(U)&= \begin{pmatrix}
            \frac{\rho_2}{\rho^2} w_i & - \frac{m_2  n_1}{\rho^2}w_i & \frac{n_1}{\rho} e_i^t \\
            -\frac{m_1   n_2}{\rho^2} w_i & \frac{\rho_1}{\rho^2}w_i & \frac{n_2}{\rho} e_i^t \\
            -\frac{m_1   w_iw}{\rho^2} +  (\gamma_1-1)s_1'(n_1) e_i & -\frac{m_2  w_iw}{\rho^2} +  (\gamma_2-1)s_2'(n_2) e_i  & \frac{1}{\rho} ( w_i \mathbb{I}_{n\times n} + R_i(w))
        \end{pmatrix} \cr
           & = \begin{pmatrix}
            \frac{\rho_2}{\rho^2} w_i & - \frac{m_2  n_1}{\rho^2}w_i & \frac{n_1}{\rho} e_i^t \\
            -\frac{m_1   n_2}{\rho^2} w_i & \frac{\rho_1}{\rho^2}w_i & \frac{n_2}{\rho} e_i^t \\
            -\frac{m_1   w_iw}{\rho^2} +  (\gamma_1-1)s_1'(n_1) e_i & -\frac{m_2  w_iw}{\rho^2} +  (\gamma_2-1)s_2'(n_2) e_i  & \frac{1}{\rho} ( w_i \mathbb{I}_{n\times n} + R_i(w))
        \end{pmatrix},
    \end{align*}
    where $R_i(w)$ is the matrix with the $i$th column equal to $w$ and zeros elsewhere:
    \begin{align*}
        R_i(w) := \begin{pmatrix}
            0 & 0 & \cdots & \underbrace{w}_{\textnormal{$i$th column}} & \cdots & 0 
        \end{pmatrix} \in \R^{n\times n}.
    \end{align*}
    Then, applying this Jacobian to $U^\varepsilon - U$ and simplifying gives  
    \begin{align*}
        &DA_i(U)\cdot (U^\ve - U) \\
        &\quad = \begin{pmatrix}
            \frac{\rho_2}{\rho^2} w_i (n_1^\ve - n_1) - \frac{m_2  n_1}{\rho^2} w_i(n_2^\ve - n_2) + \frac{n_1}{ \rho} (w_i^\ve - w_i) \\
            \vtop{\hsize=12cm \hrulefill}\\
            -\frac{m_1  n_2}{\rho^2} w_i (n_1^\ve-n_1) + \frac{\rho_1}{\rho^2} w_i (n_2^\ve - n_2) + \frac{n_2}{\rho}(w_i^\ve - w_i) \\
            \vtop{\hsize=12cm \hrulefill}\\
            -\frac{m_1  w_i w}{\rho^2} (n_1^\ve - n_1) - \frac{m_2 w_i w}{\rho^2}  (n_2^\ve - n_2) \\[6pt]
            + (\gamma_1-1)s_1'(n_1) (n_1^\ve - n_1) e_i + (\gamma_2-1)s_2'(n_2) (n_2^\ve - n_2)e_i \\[6pt]
            + \frac{w_i}{\rho}(w^\ve - w) + \frac{w}{\rho} (w_i^\ve - w_i)
        \end{pmatrix} \cr
         &\quad = \begin{pmatrix}
            \frac{\rho_2}{\rho} {\bf{u}}_i (n_1^\ve - n_1) - \frac{m_2  n_1}{\rho} {\bf{u}}_i(n_2^\ve - n_2) + \frac{n_1}{\rho} (w_i^\ve - w_i) \\
            \vtop{\hsize=12cm \hrulefill}\\
            -\frac{m_1  n_2}{\rho} {\bf{u}}_i (n_1^\ve-n_1) + \frac{\rho_1}{\rho} {\bf{u}}_i (n_2^\ve - n_2) + \frac{n_2}{\rho}(w_i^\ve - w_i) \\
            \vtop{\hsize=12cm \hrulefill}\\
            - m_1  {\bf{u}}_i {\bf{u}}  (n_1^\ve - n_1) - m_2 {\bf{u}}_i {\bf{u}}  (n_2^\ve - n_2) \\[6pt]
            + (\gamma_1-1)s_1'(n_1) (n_1^\ve - n_1) e_i + (\gamma_2-1)s_2'(n_2) (n_2^\ve - n_2)e_i \\[6pt]
            + {\bf{u}}_i(w^\ve - w) + {\bf{u}} (w_i^\ve - w_i)
        \end{pmatrix}.
    \end{align*}
    This completes the proof.
\end{proof}
%
%
%
%
%
%
%
%

\subsection{Estimate of the relative entropy}
We now quantify the evolution of the relative entropy in time. Let us denote
\[
    \calE(U^\ve|U) := \int_{\Omega} \eta(U^\ve|U)\,dx.
\]

We begin with a standard identity, which expresses the evolution of $\calE(U^\ve|U)$ in terms of the relative flux and remainder terms.
\begin{lemma}
The following identity holds:
    \begin{equation} \label{eq: relent}
        \begin{split}
        \calE(U^\ve|U) &= \calE({}^{\rm in}U^\ve|{}^{\rm in}U) + \int_\Omega \eta(U^\ve) - \eta({}^{\rm in}U^\ve)\,dx \\
        &\quad - \int_0^t \int_\Omega \nabla_x D\eta(U): A(U^\ve|U)\,dx ds + \int_0^t \int_\Omega \nabla D\eta(U) \cdot r(U^\ve)\,dx ds.
    \end{split}
    \end{equation}
\end{lemma}
\begin{proof}
 A direct computation yields  
    \begin{align*}
        \p_t[\eta(U^\ve|U)] &= \p_t(\eta(U^\ve)) - D\eta(U) \p_t U - D^2\eta(U) \p_t U \cdot (U^\ve - U) - D\eta(U)\cdot (\p_t U^\ve - \p_t U) \\
        &= \p_t (\eta(U^\ve)) + D^2\eta(U) (\nabla_x\cdot A(U)) (U^\ve - U) - D\eta(U) \cdot (-\nabla_x\cdot A(U^\ve) + R(U^\ve)).
    \end{align*}
Integrating over $\mathbb{R}^n$, we obtain the evolution identity for the relative entropy
    \begin{align*}
        \frac{d}{dt}\calE(U^\ve|U) &= \int_\Omega \p_t(\eta(U^\ve)) \,dx   + \int_\Omega D^2\eta(U)(\nabla_x\cdot A(U))(U^\ve-U)\,dx\\
        &\quad + \int_\Omega D\eta(U)\cdot \nabla_x A(U^\ve)\,dx - \int_\Omega D\eta(U)\cdot [\nabla_x\cdot r(U^\ve)]\,dx \\
        &=: \sum_{i=1}^4 I_i.
    \end{align*}
    
    We now analyze each of the four terms on the right-hand side. $I_1$ is left unchanged. For $I_2$, we use integration by parts and the chain rule (see, e.g., \cite[Lemma 4.4]{KMT15}) to write        
\[
 I_2 = \int_\Omega \nabla_x D\eta(U) : (DA(U)(U^\ve - U)),
\]
    where
\[
        [DA(U)(U^\ve - U)]_i := D A_i(U) \cdot (U^\ve - U).
\]
    Similarly, we get
    \[
    I_3 = -\int_\Omega \nabla_x D\eta(U) : A(U^\ve)\,dx.
    \]
    Thus the sum of $I_2$ and $I_3$ becomes 
\[
        I_2 + I_3 = -\int_\Omega \nabla_x D\eta(U) : A(U^\ve | U) \,dx,
\]
    where $A(U^\ve | U)$ denotes the relative flux. Here, we have used the classical entropy-entropy flux identity
    \[
    \int_\Omega \nabla_x D\eta(U):A(U)\,dx = 0,
    \] 
which holds for any smooth entropy-entropy flux pair associated to the limiting system (see \cite[Eq. (4.14)]{KMT15}).
    
    Finally, for $I_4$, another integration by parts yields
\[
        I_4 = \int_\Omega \nabla D\eta(U) \cdot r(U^\ve) \,dx.
\]
Combining all the contributions, we conclude the desired identity. 
\end{proof}

As can be seen in the previous lemma, a crucial question is whether we can estimate the relative flux $A(U^\ve|U)$ which arises on the right-hand side of the relative entropy identity. We now provide an estimate that shows this term is controlled by the relative entropy.

\begin{lemma}\label{lem: flux}
Let $1< \gamma_i \leq 2$.    Assume that $(n_1,n_2,u)$ is the strong solution to \eqref{eq: 2-Euler} given in Theorem \ref{thm: relent}. Then for any $t \in [0,T^*]$, the following estimate holds:
\[
        \int_\Omega |A(U^\ve|U)|\,dx \le C_{\lambda,m_1,m_2,\gamma_1,\gamma_2,\Lambda} \calE(U^\ve|U).
\]
\end{lemma}
\begin{proof}
The argument follows the strategy of \cite{MV08}, and we provide details for completeness. Note that the first two components of $A(U^\ve|U)$ are symmetric, and the third can be estimated directly using $\mathcal{E}(U^\ve|U)$. We thus focus on estimating the first component.
 
 From H\"older's inequality, we find
\[
 \int_\Omega |[A(U^\ve|U)]^1_i|\,dx   \le \left(\int_\Omega \left[\frac{n_1}{\rho} - \frac{n_1^\ve}{\rho^\ve} \right]^2 \rho^\ve\,dx \right)^{1/2} \left(\int_\Omega \rho^\ve |{\bf{u}}_i - {\bf{u}}_i^\ve|^2\,dx\right)^{1/2}.
\]
    Clearly, the second integral on the right-hand side can be controlled by $\calE(U^\ve|U)$. Thus, we concentrate on estimating the first term, denoted by $I$.
    
    Let us observe that
\[
        \frac{n_1}{\rho} - \frac{n_1^\ve}{\rho^\ve}  
         = \frac{m_{2} n_1 (n_{2}^\ve - n_{2}) }{\rho \rho^\ve } + \frac{m_{2} n_{2} (n_1 - n_1^\ve) }{\rho \rho^\ve}  = \frac{m_2n_2^\ve (n_1 - n_1^\ve)}{\rho \rho^\ve} + \frac{m_2 n_1^\ve (n_2^\ve - n_2) }{\rho \rho^\ve}. 
\]
This leads to two possible upper bounds for $I$:
    \begin{align}
        I &\le \int_\Omega \frac{m_2^2n_1^2(n_2^\ve-n_2)^2}{\rho^2 \rho^\ve} + \frac{m_2^2 n_2^2 (n_1-n_1^\ve)^2}{\rho^2\rho^\ve}\,dx, \label{eq: diff1} \\
        I &\le \int_\Omega \frac{m_2^2 (n_2^\ve)^2 (n_1-n_1^\ve)^2}{\rho^2 \rho^\ve} + \frac{m_2^2(n_1^\ve)^2 (n_2^\ve-n_2)^2}{\rho^2 \rho^\ve}\,dx. \label{eq: diff2}
    \end{align}

 We now distinguish two cases. First, suppose $\rho^\ve < \Lambda$. Then using \eqref{eq: diff2} and the bounds $m_in_i^\ve < \Lambda$ and $m_in_i^\ve / \rho^\ve \le 1$, we obtain
    \begin{align*}
        I &\le \Lambda \int_\Omega \frac{(n_1-n_1^\ve)^2}{\rho^2} \,dx + \frac{m_2^2 \Lambda}{m_1^2} \int_\Omega \frac{(n_2^\ve - n_2)^2}{\rho^2}\,dx \\
        &\le C_{m_1,m_2,\Lambda,\lambda} \int_\Omega (n_1-n_1^\ve)^2 + (n_2^\ve-n_2)^2\,dx \\
        &\le C_{m_1,m_2,\gamma_1,\gamma_2,\Lambda,\lambda} \int_\Omega s_1(n_1^\ve|n_1) + s_2(n_2^\ve|n_2)\,dx,
    \end{align*}
where the final line uses Taylor's theorem and the condition $\gamma_i \le 2$ to deduce
    \begin{align*}
         s_i(n_i^\ve|n_i) &\ge \min\left\{\frac{1}{(n_i^\ve)^{2-\gamma_i}}, \frac{1}{(n_i)^{2-\gamma_i}} \right\} (n_i^\ve - n_i)^2 \ge C_{m_i,\gamma_i,\Lambda} (n_i^\ve - n_i)^2.
    \end{align*}

Next, suppose $\rho^\ve \ge \Lambda$. In this case we apply \eqref{eq: diff1} and use the fact that $m_in_i / \rho \le 1$ to first estimate
\[
        I \le C_{m_1,m_2,\Lambda} \int_\Omega \frac{(n_1^\ve-n_1)^2}{\rho^\ve} + \frac{(n_2^\ve-n_2)^2}{\rho^\ve}\,dx.
\]
In the subcase $n_i^\ve < \Lambda$, we have
\[
        \min\left\{\frac{1}{(n_i^\ve)^{2-\gamma_i}}, \frac{1}{(n_i)^{2-\gamma_i}}\right\} \ge \Lambda^{\gamma_i-2},
\]
and thus
\[
        \int_\Omega \frac{(n_i^\ve-n_i)^2}{\rho^\ve}\,dx \le \frac{1}{\Lambda}\int_\Omega (n_i^\ve - n_i)^2\,dx \le C_{m_i,\gamma_i,\Lambda}\int_\Omega s_i(n_i^\ve|n_i)\,dx.
\]
If instead $n_i^\ve \ge \Lambda$, then $n_i^\ve \ge n_i$, and
\[
        \min\left\{\frac{1}{(n_i^\ve)^{2-\gamma_i}}, \frac{1}{(n_i)^{2-\gamma_i}}\right\} = \frac{1}{(n_i^\ve)^{2-\gamma_i}},
\]
so using $\gamma_i > 1$ and $\rho_i^\ve \ge \Lambda$,
\[
        \int_\Omega \frac{(n_i^\ve-n_i)^2}{\rho^\ve}\,dx \le \int_\Omega \frac{(n_i^\ve-n_i)^2}{m_in_i^\ve}\,dx \le C_{\Lambda,\gamma_i} \int_\Omega \frac{(n_i^\ve-n_i)^2}{m_i (n_i^\ve)^{2-\gamma_i}}\,dx \le C_{\Lambda,\gamma_i,m_i} \int_\Omega s_i(n_i^\ve|n_i)\,dx.
\]

Combining all cases, we conclude that
\[
    I \le C_{m_1, m_2, \gamma_1, \gamma_2, \Lambda} \mathcal{E}(U^\ve|U),
\]
and hence the desired bound follows.
\end{proof}

%
%
%
%
%
%
%
%

\subsection{Proof of Theorem \ref{thm: relent}}
We now combine the estimates established in the previous subsections to complete the proof of Theorem \ref{thm: relent}. Our approach is based on controlling the relative entropy functional via the identity \eqref{eq: relent}, and carefully estimating each term appearing on its right-hand side.
 
Let 
\[
\calE(U^\ve|U) =: J_1 + J_2 + J_3 + J_4, 
\]
where the $J_i$ denote the four terms on the right-hand side of \eqref{eq: relent}.

We begin by estimating $J_1$, which is bounded by $\ve^{1/4}$ due to the assumption on the well-preparedness of the initial data.

For $J_2$, we observe that
\begin{align*}
    \int_\Omega \eta(U^\ve) - \eta({}^{\rm in}U^\ve)\,dx &= \iint_{\Omega \times \R^n} H(\calM_1(f^\ve),\calM_2(f^\ve)) - H(\calM_1({}^{\rm in}f^\ve),\calM_2({}^{\rm in}f^\ve))\,dxdv \\
    &\le \iint_{\Omega \times \R^n} H(f^\ve) - H({}^{\rm in}f^\ve) \,dxdv \\
    &\quad + \iint_{\Omega \times \R^n} H({}^{\rm in}f^\ve) - H(\calM_1({}^{\rm in}f^\ve),\calM_2({}^{\rm in}f^\ve))\,dxdv \\
    &\lesssim \ve^{1/4},
\end{align*}
due to the entropy dissipation property and the well-preparedness of the initial data.

To estimate $J_3$, we use the expression
    \begin{align*}
        D\eta(U) = \begin{pmatrix}
            -\frac{m_1}{2}|{\bf{u}}|^2 + \frac{\gamma_1 m_1^{\frac{n}{2}(\gamma_1 - 1)}}{\gamma_1 - 1}n_1^{\gamma_1-1} \\[2mm]
            -\frac{m_2}{2}|{\bf{u}}|^2 + \frac{\gamma_2 m_2^{\frac{n}{2}(\gamma_2 - 1)}}{\gamma_2 - 1} n_2^{\gamma_2 - 1}\\[2mm]
            {\bf{u}}
        \end{pmatrix}
    \end{align*}
and the bound $\gamma_i \le 2$ to obtain 
    \begin{align*}
        \|\nabla D\eta(U)\|_{L^\infty} \le C(m_1+m_2 + 1)\|{\bf{u}}\|_{H^s}^2 + \frac{\gamma_1 m_1^{\frac{n}{2}(\gamma_1 - 1)}}{\lambda^{2-\gamma_1}} \|\nabla n_1\|_{L^\infty} + \frac{\gamma_2 m_2^{\frac{n}{2}(\gamma_2 - 1)}}{\lambda^{2-\gamma_2}} \|\nabla n_2\|_{L^\infty} \le C.
    \end{align*}
Then, by Lemma \ref{lem: flux}, we deduce
    \begin{align*}
        J_3 \le C \int_0^t \int_\Omega |A(U^\ve|U)|\,dx ds \le C\int_0^t \calE(U^\ve|U)\, ds.
    \end{align*}
    
For $J_4$, we estimate
    \begin{equation} \label{eq: J4}
    \begin{split}
    J_4 &\le \|\nabla D\eta(U)\|_{L^\infty} \int_0^t \int_\Omega \left|\intr r(U^\ve)\,dv\right|\,dxds \\
    &\le C\sum_{i=1}^2 \Bigg[ \int_0^t \iint_{\Omega \times \R^n} |v| |f_i^\ve - \calM_i|\,dvdx ds  + \int_0^t \int_\Omega \left|\intr |v|^2 (\calM_i - f_i^\ve)\,dv\right|dx ds\Bigg] \\
    &\le C \sum_{i=1}^2  \sqrt{2}\|f_i^\ve\|_{L^1((0,t)\times\Omega \times \R^n)}^{1/2} \left(\int_0^t \iint_{\Omega \times \R^n} |v|^2 |f_i^\ve - \calM_i|\,dv dx ds\right)^{1/2}  \\
    &\quad + C \sum_{i=1}^2 \int_0^t \iint_{\Omega \times \R^n} |v|^2 |f_i^\ve - \calM_i|\,dvdxds.
    \end{split}
    \end{equation}
    In the above, the last inequality follows by applying H\"older's inequality and then the triangle inequality, recalling that $\intr f_i^\ve\,dv = \intr \calM_i(f^\ve)\,dv$. Since the kinetic equation conserves mass, we find 
    \[
    \|f_i^\ve\|_{L^1((0,t)\times\Omega \times \R^n)} = t\|{}^{\rm in}f_i^\ve\|_{L^1(\Omega \times \R^n)}. 
    \]
The second moment terms can be estimated by (see \cite[Proposition 4.1]{BV05}, \cite[Lemma 1.4]{KS24})
    \begin{align*}
    \int_0^t \iint_{\Omega \times \R^n} |v|^2 |f_i^\ve - \calM_i|\,dvdxds \le \int_0^t \int_\Omega  ((n_i^{\ve})^{\gamma_i/2 } \sqrt{D_i} + D_i) \,dx ds, 
    \end{align*}
    where 
    \[
    D_i := \intr h_i(f_i^\ve) - h_i(\calM_i(f_1^\ve,f_2^\ve))\,dv \ge 0
    \]
     by Remark \ref{rem: each h}.     Using the uniform bound of $n_i^\ve$ in $L^\infty(0,T;L^{\gamma_i})$, we obtain
    \begin{equation} \label{eq: secmoments}
    \begin{split}
        \int_0^t \iint_{\Omega \times \R^n} |v|^2 |f_i^\ve - \calM_i|\,dvdxds &\le C \int_0^t \left(\int_\Omega D_i\,dx\right)^{1/2}ds + \int_0^t \int_\Omega D_i\,dxds \\
        &\le C_T \left(\int_0^t \int_\Omega D_i\,dxds \right)^{1/2} + \int_0^t \int_\Omega D_i\,dxds.
    \end{split}
    \end{equation}
    
    Combining \eqref{eq: J4} and \eqref{eq: secmoments}, and applying arithmetic geometric inequality $\sqrt{x}+\sqrt{y} \le \sqrt{2(x+y)}$ twice, we deduce
    \begin{align*}
        J_4 &\le C_T \left( 2 \sum_{i=1}^2 \int_0^t \iint_{\Omega \times \R^n} |v|^2 |f_i^\ve - \calM_i|\,dvdxds \right)^{1/2}   + C\sum_{i=1}^2 \int_0^t \iint_{\Omega \times \R^n} |v|^2 |f_i^\ve - \calM_i|\,dvdxds \\
        &\le C_T \left( \sum_{i=1}^2 \left[\left(\int_0^t \int_\Omega D_i\,dxds \right)^{1/2} + \int_0^t \int_\Omega D_i\,dxds \right] \right)^{1/2}  \\
        &\quad + C_T \sum_{i=1}^2 \left[\left(\int_0^t \int_\Omega D_i\,dxds\right)^{1/2} + \int_0^t \int_\Omega D_i\,dxds \right] \\
        &\le C_T \left(\sum_{i=1}^2 \int_0^t \int_\Omega D_i\,dxds \right)^{1/4} + C_T \left(\sum_{i=1}^2 \int_0^t \int_\Omega D_i\,dxds\right)^{1/2} + C_T \int_0^t \int_\Omega D_i\,dxds \\
        &\lesssim_T \ve^{1/4}.
    \end{align*}
 Summing the estimates for $J_1$ through $J_4$, we arrive at
    \begin{align*}
        \calE(U^\ve|U) \le C_T \ve^{1/4} + C\int_0^t \calE(U^\ve|U) \, ds,
    \end{align*}
  from which the desired estimate follows via Gr\"onwall inequality. The strong convergence statements in \eqref{eq: hyd} are then deduced from the decay of the relative entropy, and we refer to \cite{BV05, CCJ21} for the standard arguments.

%
%
%
%
%
%
%
%
%
%
%
%
%

%
%
%
%
%

\section{Numerical scheme and asymptotic-preserving property}\label{sec_numer}
In this section, we describe a class of finite difference numerical methods for the BGK model \eqref{CE main}. In the construction of a numerical method, the main issue is whether we can capture the correct dynamics in the limit $\varepsilon\to 0$. More precisely, in view of the efficiency of numerical methods, we need to handle the time step restriction that arises from the small Knudsen number. One popular approach to address the stiff problem is to use an implicit-explicit time discretization. In detail, the non-stiff convection term is handled explicitly, and hence CFL-type time step restriction remains. However, the stiff relaxation term is treated implicitly, which allows us to avoid time step restriction from small Knudsen numbers. It is notable that this strategy has been widely adopted in numerous works \cite{FJ,ZH, PP}. 

In this work, we adopt implicit-explicit Runge--Kutta (IMEX-RK) methods for time discretization. Butcher's table that corresponds to $s$-stage IMEX-RK time discretization is represented by
\begin{align*}
	\begin{array}{c|c}
		\tilde{c} & \tilde{A}\\
		\hline \\[-3mm]
		& \tilde{b}^\top
	\end{array},\quad \begin{array}{c|c}
		c & A\\
		\hline \\[-3mm]
		& b^\top
	\end{array},
\end{align*}
where $\tilde{A},\,A$ are $s\times s$ real matrices and $\tilde{b}^\top,\,b^\top,\,\tilde{c},\,c$ are real vectors of length $s$. Note that the left table $(\tilde{A},\tilde{b},\tilde{c})$ is for the explicit part and the right one $(A,b,c)$ is for the implicit part. 

Among numerous IMEX-RK methods, when choosing IMEX method, we take into account two aspects: in the limit $\varepsilon\to0$ (I) whether the schemes degenerate into an approximation of the isentropic Euler system and (II) whether the scheme ensures the asymptotic accuracy. 
These two properties are referred to as asymptotic preserving (AP) and asymptotic accurate (AA) properties in the literature \cite{BPR}. In our framework, AP property means that a consistent time discretization of BGK model \eqref{CE main} becomes a consistent time discretization method for the isentropic Euler system \eqref{Euler system} in the limit $\varepsilon\to 0$, and AA property implies that the order of time discretization is maintained in the limit $\varepsilon\to 0$. Note that these properties enable us to ensure the robustness and high accuracy of the numerical methods in the limit $\varepsilon \to 0$. 

With reference to \cite{BPR}, to attain the AP and AA properties, we adopt a $L$-stable second order globally stiffly accurate (GSA) IMEX-RK method called ARS(2,3,2) \cite{ARS}:
	\begin{align}\label{Butcher}
		\tilde{A}=\begin{array}{c|ccc}
			0& 0 &0 & 0\\
			\alpha& \alpha &0 & 0\\
			1 &\delta& 1-\delta & 0\\
			\hline \\[-3mm]
			& \delta& 1-\delta & 0
		\end{array},\quad A=\begin{array}{c|ccc}
			0& 0 &0 & 0\\
			\alpha& 0&\alpha &0 \\
			1 &0& 1-\alpha &\alpha \\
			\hline \\[-3mm]
			& 0&1-\alpha & \alpha
		\end{array},\quad \alpha = 1-\frac{1}{\sqrt{2}},\quad \delta = 1-\frac{1}{2\alpha},
	\end{align}
where the left table applies to the convection term, while the right one applies to the relaxation terms. Note that the adopted IMEX-RK method \eqref{Butcher} is of type CK and ARS. In addition, the method also satisfies the GSA condition, which means that the last row of $\tilde{A}$ and $A$ coincides with the weights $\tilde{b}^\top$ and $b^\top$. This is an important feature to attain AP property. For more details, we refer to a book \cite{BPR}.

Now, we describe general $s$-stage IMEX-RK methods for \eqref{CE main}. For the description of numerical methods, we consider the case when $\nu_i$ is dependent on $n_i$, i.e., there is no dependency on $v$ variable. We begin by discretizing \eqref{CE main} in time with $s$-stage IMEX-RK scheme, i.e., 
\begin{equation}\label{IMEXcomp}
	f_i^{(k)} = f_i^{m} - \Delta t\sum_{\ell=1}^{k-1}\tilde{a}_{k\ell} v \cdot \nabla_x f_i^{(\ell)} + \frac{\Delta t}{\varepsilon}\sum_{\ell=1}^{k}a_{k\ell}\nu_i^{(\ell)}\left(\mathcal{M}_i^{(\ell)} - f_i^{(\ell)}\right), \quad k = 1,\dots,s, \quad i=1,2,
\end{equation}
\begin{equation}\label{solnum}
	f_i^{m+1} = f_i^{m} - \Delta t\sum_{k=1}^{s}\tilde{b}_{k,} v \cdot \nabla_x f_i^{(k)} + \frac{\Delta t}{\varepsilon}\sum_{k=1}^{s}{b}_{k}\nu_i^{(k)}\left(\mathcal{M}_i^{(k)} - f^{(k)}\right),
\end{equation}
where $f_i^{(k)}$ and $f_i^{m+1}$ denote the $k$-th stage value and numerical solution at time $t_{m+1}$, respectively. 
To ensure the stability of numerical solutions and prevent spurious oscillations near discontinuities, we approximate the convection terms $v\cdot \nabla f^{(k)}$  using upwind flux with WENO method \cite{CWS}. Each $k$-th stage value will be computed from $k=1$ to $s$ in order using the following arguments.

By integrating \eqref{IMEXcomp} 
over velocity variable, we obtain the relation for discrete number densities $n_i^{(k)}$ of species $i$ as follows:
\begin{align}\label{rho k}
	\begin{split}
		n_i^{(k)} &= \int_{\mathbb{R}^{n}}   \left(f_i^{n} - \Delta t\sum_{\ell=1}^{k-1}\tilde{a}_{k\ell} v \cdot \nabla_x f_i^{(\ell)} \right)  dv,\quad i=1,2.
	\end{split}
\end{align}
Note that we can use $n_i^{(k)}$ to obtain the collision frequencies $\nu_i^{(k)}$ explicitly. Next, to compute the discrete bulk velocity $u_i^{(k)}$ of species $i$, we multiply $v$ to \eqref{IMEXcomp} and integrate over $v$:
\begin{align}\label{stage uk}
	\begin{split}
		n_i^{(k)}u_i^{(k)}
		 &= \underbrace{\int_{\mathbb{R}^{n}}   \left(f_i^{n} - \Delta t\sum_{\ell=1}^{k-1}\tilde{a}_{k\ell} v \cdot \nabla_x f_i^{(\ell)}  + \frac{\Delta t}{\varepsilon}\sum_{\ell=1}^{k-1}a_{k\ell}\nu_i^{(\ell)} (\mathcal{M}_i^{(\ell)}-f_i^{(\ell)}) \right) v\,  dv}_{A_i^{(k)}} \cr
		 &\quad + \frac{\Delta t}{\varepsilon}a_{kk}\nu_i^{(k)} (n_i^{(k)}u^{(k)}-n_i^{(k)}u_i^{(k)}),
	\end{split}
\end{align}
where we introduce $A_i^{(k)}$ ($i=1,2$) for brevity. Then, we can transform \eqref{stage uk} into
\begin{align*}
	\begin{split}
		u_1^{(k)} 
		&= \frac{A_1^{(k)}}{n_1^{(k)}}+ \frac{\Delta t}{\varepsilon}a_{kk}\nu_1^{(k)} \left(\frac{\nu_1^{(k)}\rho_1^{(k)}u_1^{(k)} + \nu_2^{(k)}\rho_2^{(k)}u_2^{(k)}}{\nu_1^{(k)}\rho_1^{(k)} + \nu_2^{(k)}\rho_2^{(k)}}-u_1^{(k)}\right),\cr
		u_2^{(k)} 
		&= \frac{A_2^{(k)}}{n_2^{(k)}}+ \frac{\Delta t}{\varepsilon}a_{kk}\nu_2^{(k)} \left(\frac{\nu_1^{(k)}\rho_1^{(k)}u_1^{(k)} + \nu_2^{(k)}\rho_2^{(k)}u_2^{(k)}}{\nu_1^{(k)}\rho_1^{(k)} + \nu_2^{(k)}\rho_2^{(k)}}-u_2^{(k)}\right),
	\end{split}
\end{align*}
which can be rearranged as follows:
\begin{equation}\label{u sys}
	\begin{bmatrix}
		1+ \frac{\Delta t}{\varepsilon}a_{kk}\nu_1^{(k)} \frac{ \nu_2^{(k)}\rho_2^{(k)}}{\nu_1^{(k)}\rho_1^{(k)} + \nu_2^{(k)}\rho_2^{(k)}} & -\frac{\Delta t}{\varepsilon}a_{kk}\nu_1^{(k)} \frac{ \nu_2^{(k)}\rho_2^{(k)}}{\nu_1^{(k)}\rho_1^{(k)} + \nu_2^{(k)}\rho_2^{(k)}}\\
		-\frac{\Delta t}{\varepsilon}a_{kk}\nu_2^{(k)} \frac{ \nu_1^{(k)}\rho_1^{(k)}}{\nu_1^{(k)}\rho_1^{(k)} + \nu_2^{(k)}\rho_2^{(k)}} & 1+\frac{\Delta t}{\varepsilon}a_{kk}\nu_2^{(k)} \frac{ \nu_1^{(k)}\rho_1^{(k)}}{\nu_1^{(k)}\rho_1^{(k)} + \nu_2^{(k)}\rho_2^{(k)}}
	\end{bmatrix}
	\begin{bmatrix}
		u_1^{(k)} \\
		u_2^{(k)}  
	\end{bmatrix}=	\begin{bmatrix}
		\frac{A_1}{n_1^{(k)}} \\
		\frac{A_2}{n_2^{(k)}}  
	\end{bmatrix}.
\end{equation}
This linear system is solvable if $a_{kk}\neq 0$ and $n_i^{(k)}>0$ for $i=1,2$. 
To sum up, the relations \eqref{rho k} and the solvability of the system \eqref{u sys} allow us to compute the discrete common bulk velocity $u^{(k)}$ explicitly:
\begin{align*}
	u^{(k)}= \frac{\nu_1^{(k)}\rho_1^{(k)}u_1^{(k)} + \nu_2^{(k)}\rho_2^{(k)}u_2^{(k)}}{\nu_1^{(k)}\rho_1^{(k)} + \nu_2^{(k)}\rho_2^{(k)}}.
\end{align*}
This further implies that $\mathcal{M}_i^{(k)}:=\mathcal{M}_i[n_i^{(k)},u^{(k)}](v)$ can be obtained without resorting to any Newton solver. Thus, we can update stage values as follows:
\begin{equation*}
	f_i^{(k)} = \frac{\varepsilon}{\varepsilon + \Delta t a_{kk} \nu_i^{(k)}} f_{i,*}^{(k)} + \frac{\Delta t  a_{kk}\nu_i^{(k)}}{\varepsilon + \Delta t a_{kk}\nu_i^{(k)}}\mathcal{M}_i^{(k)}, 
\end{equation*}
where
$$
f_{i,*}^{(k)} = f_{i}^{m} - \Delta t\sum_{\ell=1}^{k-1}\tilde{a}_{k\ell} v \cdot \nabla_x f_{i}^{(\ell)} + \frac{\Delta t}{\varepsilon}\sum_{\ell=1}^{k-1}a_{k\ell} \nu_i^{(\ell)}\left(\mathcal{M}_i^{(\ell)} - f_i^{(\ell)}\right).
$$
Finally, using the stage values, we can update the numerical solution at time $t_{n+1}$ reads
\begin{align*}
\begin{split}
	f_i^{m+1} &= f_i^{m} - \Delta t\sum_{k=1}^{s}\tilde{b}_{k} v \cdot \nabla_x f_i^{(k)} + \frac{\Delta t}{\varepsilon}\sum_{k=1}^{s}{b}_{k}\nu_i^{(k)}\left(\mathcal{M}_i^{(k)} - f_i^{(k)}\right).
\end{split}
\end{align*}
If the method is GSA, one can obtain a numerical solution immediately, i.e., $f_i^{m+1}=f_i^{(s)}$.

We next investigate the behavior of numerical solutions in the limit $\varepsilon\to 0$. The following proposition shows that the proposed IMEX-RK method degenerates into the limiting isentropic Euler system.
\begin{proposition}\label{prop}
	Assume that a $s$-stage IMEX-RK method of type CK and GSA is considered in the proposed method \eqref{IMEXcomp} and \eqref{solnum} for BGK model \eqref{CE main}.
	Then, for a well-prepared initial
	data, i.e., $\lim\limits_{\varepsilon \to 0} f_i^m=\mathcal{M}_i[n_i^m,\mathbf{u}^m]$, and nonzero relaxation term $\nu_i^{(k)}\neq 0$ $(i=1,2,\, k=1,\dots,s)$, in the limit $\varepsilon \to 0$ the numerical method becomes the explicit RK scheme characterized by the pair $(\tilde{A},\tilde{b})$ applied to the limit isentropic Euler system (\ref{Euler system}). 
\end{proposition}
\begin{proof}
 For brevity of proof, we begin by introducing vector notations:
\begin{align*}
	&{\bf{\nu}}_i=diag(\nu_i^{(1)},\hat{\bf{\nu}}_i),\quad {\bf{\mathcal{M}}}_i=\left(\mathcal{M}_i^{(1)},\hat{\bf{\mathcal{M}}}_i\right)^\top,\quad{\bf{f}}_i=(f_i^{(1)},\hat{\bf{f}}_i)^\top,\cr
	&\hat{\bf{\nu}}_i=diag(\nu_i^{(2)},\dots,\nu_i^{(s)}),\quad \hat{\bf{\mathcal{M}}}_i=\left(\mathcal{M}_i^{(2)},\dots,\mathcal{M}_i^{(s)}\right)^\top,\quad\hat{\bf{f}}_i=(f_i^{(2)},\dots,f_i^{(s)})^\top,\quad i=1,2.
\end{align*}
Using this notation, the stage values \eqref{IMEXcomp} can be rewritten as follows:
\begin{align}\label{num form}
	\begin{split}
		{\bf{f}}_i &= f_i^{m}{\textbf{e}_{s}} - \Delta t \tilde{A} v \cdot \nabla_x {\bf{f}}_i + \frac{\Delta t}{\varepsilon}A{{\bf {\nu}}}_i\left({\bf{\mathcal{M}}}_i - {\bf{f}}_i\right), \quad i=1,2,
	\end{split}
\end{align}
where $\tilde{A}$ and $A$ are the explicit and implicit parts of an IMEX-RK method of type CK and ${\bf{e}}_{s}=(1,\dots,1)^\top \in \mathbb{R}^{s}$. 
Taking the zeroth moment $\langle \cdot \rangle$ on both sides of \eqref{num form}, we derive the equation for number densities:
\begin{align}\label{num n form}
	\begin{split}
		\langle{\bf{f}}_i \rangle &= 	\langle f_i^{m}\rangle {\textbf{e}_{s}} - \Delta t \tilde{A} 	\langle v \cdot \nabla_x {\bf{f}}_i \rangle.
	\end{split}
\end{align}
Similarly, we take the first moment $\langle \cdot m_i v \rangle$ on both sides of \eqref{num form} and sum up the equations for $i=1,2$ to derive the equation for bulk velocity of the mixture:
\begin{align}\label{num u form}
	\begin{split}
		\sum_i \langle{m_i v\bf{f}}_i \rangle &= \sum_i \langle m_i  vf_i^{m}\rangle{\textbf{e}_{s}} - \Delta t \tilde{A} \langle m_i vv \cdot \nabla_x {\bf{f}}_i \rangle + \frac{\Delta t}{\varepsilon}A{{\bf {\nu}}}_i\left\langle {m_i v\bf{\mathcal{M}}}_i -m_iv {\bf{f}}_i\right\rangle,\cr
		&= \sum_i \langle m_i  vf_i^{m}\rangle{\textbf{e}_{s}} - \Delta t \tilde{A} \langle m_i vv \cdot \nabla_x {\bf{f}}_i \rangle,
	\end{split}
\end{align}
where the second equality holds due to the definition of common velocity.

Next, we pass the limit $\varepsilon\to 0$, then, by the assumption that the initial data is well-prepared, we have 
\begin{align}\label{compact form 1}
	&f_i^{(1)}=\mathcal{M}_i[n_i^m,\mathbf{u}^m],\quad i=1,2.
\end{align}
This further gives
\begin{align*}
	\begin{split}
		n_i^{(1)} &=n_i^m,\quad n_i^{(1)}u_i^{(1)}=n_i^m \mathbf{u}^m,\quad i=1,2,
	\end{split}
\end{align*}
and
\[
\rho^{(1)}\mathbf{u}^{(1)}=m_1n_1^{(1)}u_1^{(1)}+m_2n_2^{(1)}u_2^{(1)}=\rho^m \mathbf{u}^m.
\]
In a similar manner, in the limit $\varepsilon\to 0$, we derive from the second to last rows of \eqref{num form} that
\begin{align*}
	\hat{A}\hat{\bf{\nu}}_i(\hat{\bf{\mathcal{M}}}_i-\hat{\bf{f}}_i)={\bf{0}}_{s-1},\quad i=1,2,
\end{align*}
where $\hat{A}\in \mathbb{R}^{s-1 \times s-1}$ is a submatrix of $A$, i.e., $A=\begin{bmatrix}
	0 &\bf{0}_{s-1}^\top\\
\bf{0}_{s-1}&\hat{A}
\end{bmatrix}$ and ${\bf{0}}_{s-1}=(0,\dots,0)^\top \in \mathbb{R}^{s-1}$.
Since the submatrix $\hat{A}$ in the IMEX-RK methods of type CK is non-singular and $\nu_i^{(k)}\neq 0$ is assumed for $k=2,\dots,s$, we further have
\begin{align}\label{compact form 2}
	{\hat{\bf{f}}_i=\hat{\bf{\mathcal{M}}}_i} \quad i=1,2.
\end{align}
Taking the moment with respect to $v$ in this equation, we have $u_i^{(k)}=u^{(k)}$ for $k=2,\dots,s$. This further implies that 
$$\textbf{u}^{(k)}=\frac{\rho_1^{(k)} u_1^{(k)} + \rho_2^{(k)} u_2^{(k)}}{\rho^{(k)}}=u^{(k)},\quad k=2,\dots,s.$$
Now, we insert \eqref{compact form 1} and \eqref{compact form 2} into \eqref{num n form} and \eqref{num u form} to derive
\begin{align*}
	\begin{split}
		n_i^{(1)} &=n_i^m,\cr
		n_i^{(k)} &= n_i^{m} - \Delta t\sum_{\ell=1}^{k-1}\tilde{a}_{k\ell}  \nabla_x \cdot (n_i^{(\ell)} \mathbf{u}^{(\ell)}),\quad k=2,\dots,s, \quad i=1,2,\cr
			\rho^{(1)}\mathbf{u}^{(1)} &=\rho^m \mathbf{u}^m,\cr
		\rho^{(k)} \mathbf{u}^{(k)} &= \rho^{m} \mathbf{u}^{m} - \Delta t\sum_{\ell=1}^{k-1}\tilde{a}_{k\ell}  \left[\nabla_x \cdot \left(\rho^{(\ell)} \mathbf{u}^{(\ell)} \otimes \mathbf{u}^{(\ell)} \right)+ \nabla_x\left( m_1^{\frac{n}{2}(\gamma_1 - 1)} n_1^{\gamma_1}+ m_2^{\frac{n}{2}(\gamma_2 - 1)} n_2^{\gamma_2}\right)\right],\quad k=2,\dots,s.
	\end{split}
\end{align*}
This is the explicit RK scheme characterized by $(\tilde{A}, \tilde{b})$ applied to the isentropic Euler equations \eqref{Euler system}. Note that the scheme is GSA, and hence we have $f_i^{m+1}=\mathcal{M}_i^{m+1}$ for each $i=1,2$. This again becomes the well-prepared initial data for the next time step.
\end{proof}
%
%
%
%
%
%
%
%
%
%
%
%
\section{Numerical tests: Comparison with two-phase Euler flow}\label{sec_numer2}
In this section, our goal is to show that our proposed method is asymptotic preserving in the limit $\varepsilon\to 0$, i.e.,  we will show that numerical solutions to the BGK model and its isentropic Euler system are consistent for small values of $\varepsilon$. To verify this, we focus on the comparison of two equations: (I) the BGK model with one dimension in space and velocity \eqref{CE main}, (II) the isentropic Euler system in one dimension \eqref{Euler system}. In the BGK model, since our interest lies in the scale of the Knudsen number, we simply set $\nu_i=1$ for $i=1,2$. Moreover, for both models, we set $m_i=1$ for $i=1,2$ to clarify the effect of values of $\gamma_i$. Under this assumption, we have $\rho_i=n_i$ for $i=1,2$. For numerical tests, we use the following uniform grids:
\begin{align*}
	t_m&=m\Delta t,\quad m=0,1,2,\dots,\cr
	x_i&=x_L + (i+1/2)\Delta x,\quad i=0,\dots,N_x-1,\cr 
	v_j&=v_{min}+j\Delta v,\quad j=0,\dots,N_v,
\end{align*}
where the spatial and velocity domains are $[x_L,x_R]$ and $[v_{min},v_{max}]$, respectively. The size of grids for $x$ and $v$ are denoted by $\Delta x$ and $\Delta v$, and the time step $\Delta t$ is determined according to CFL number defined by
$$\text{CFL}=\max_j |v_j|\frac{\Delta t}{\Delta x}.$$ 
We compute numerical solutions to the BGK model with the proposed second-order IMEX-RK method \eqref{Butcher} with classical WENO23. Note that we use CFL$=0.4$ and $N_x=200$ on the spatial domain $x\in [-1,1]$. For solving the isentropic Euler system, we use the explicit part of the IMEX-RK method in \eqref{Butcher} for time discretization, while we adopt Lax-Friedrichs flux and WENO23 to treat numerical fluxes.

%
%
%
%
%
%
%
%
%
%
%
%
\subsection{Riemann problem I}
In this example, we show that the proposed scheme is asymptotically preserving in the limit $\varepsilon\to 0$. For this aim, we consider a Riemann problem with initial macroscopic variables:
\begin{align*}
	\rho_{10}(x)=\begin{cases}
		1,\quad x\leq 0\\
		0.5\quad x>0
	\end{cases},\quad 	\rho_{20}(x)=\begin{cases}
		0.8,\quad x\leq 0\\
		0.25\quad x>0
	\end{cases}.
\end{align*}
The equilibrium attractors are taken as initial distribution functions. For boundaries on the spatial domain, we assume the free-flow boundary condition. For each value of $\varepsilon\in \{10^{-2},\,10^{-4},\,10^{-6}\}$, we compute numerical solutions up to $t=0.25$. For the comparison, we set $\gamma_1=2,3$ and $\gamma_2=7/5$ 
and take $N_v=1000$ to resolve the shape of the distribution function on the velocity domain $[-3,3]$.
\begin{figure}[htbp]
	\centering
	\begin{subfigure}[h]{0.49\linewidth}
		\includegraphics[width=1\linewidth]{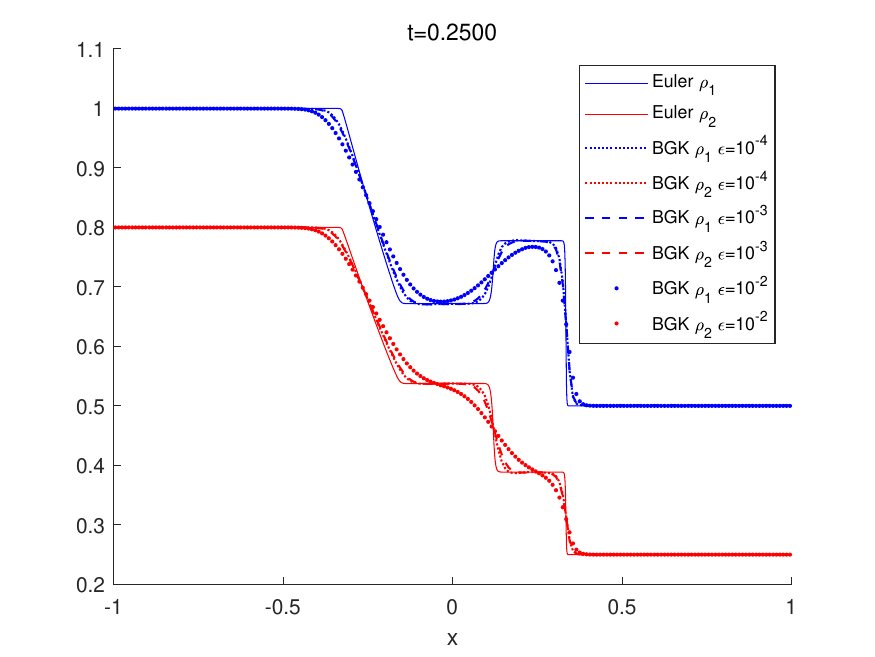}
	\end{subfigure}	
	\begin{subfigure}[h]{0.49\linewidth}
		\includegraphics[width=1\linewidth]{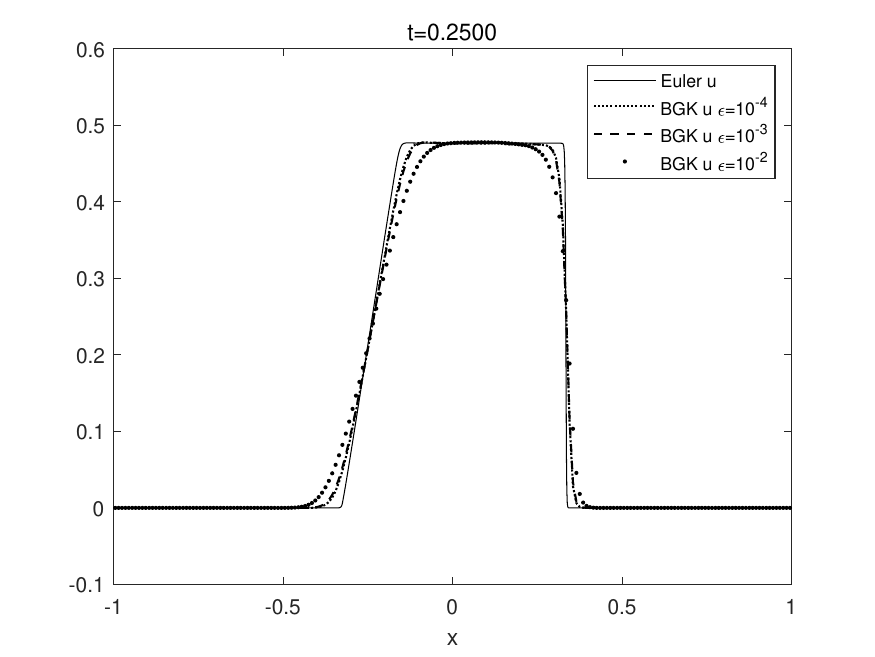}
	\end{subfigure}	
	\caption*{Case 1: $\gamma_1=2$ and $\gamma_2=7/5$.}
	\begin{subfigure}[h]{0.49\linewidth}
		\includegraphics[width=1\linewidth]{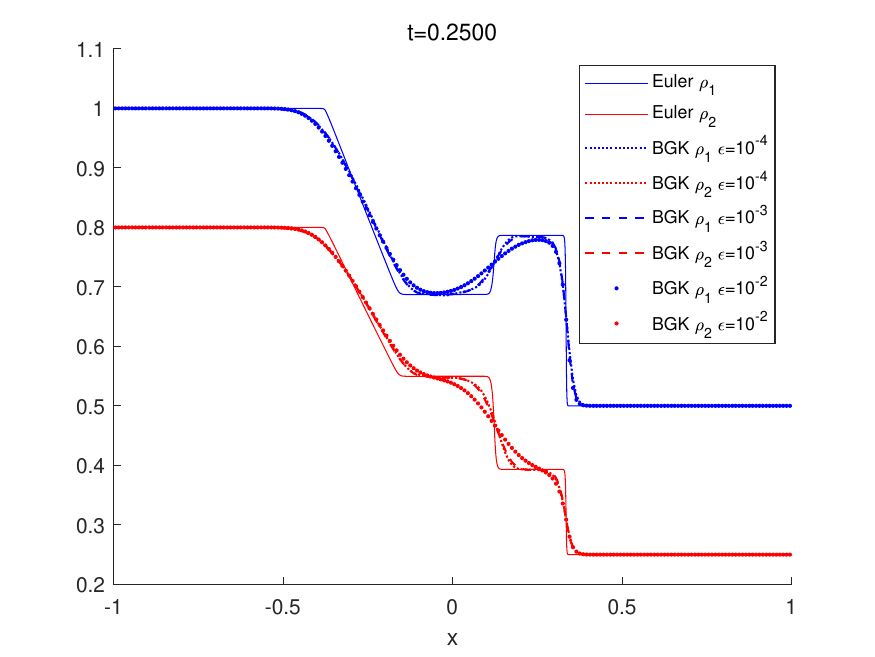}
	\end{subfigure}	
	\begin{subfigure}[h]{0.49\linewidth}
		\includegraphics[width=1\linewidth]{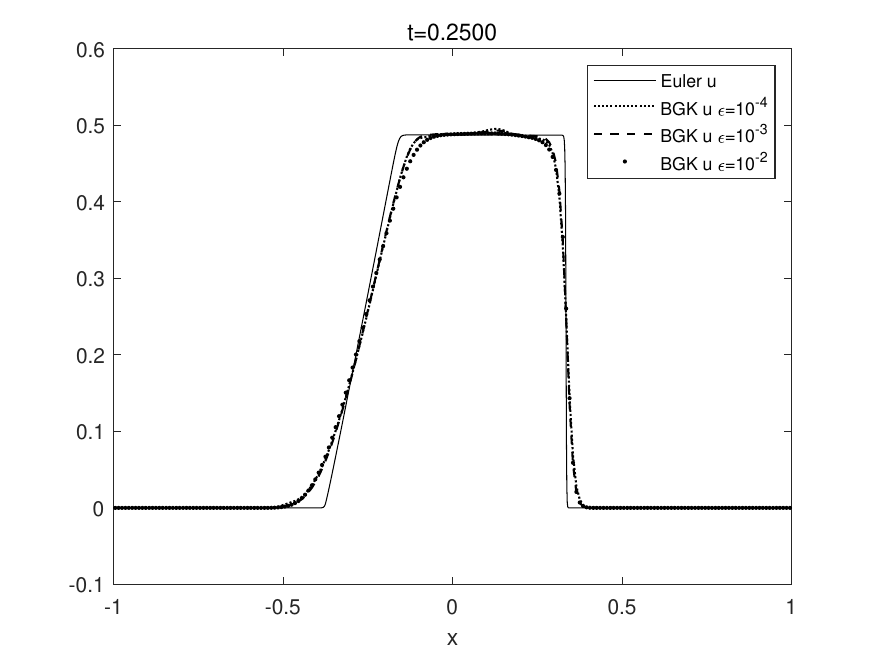}
	\end{subfigure}	
	\caption*{Case 2: $\gamma_1=3$ and $\gamma_2=5/3$.}
	\caption{Riemann Problem I: Comparison of species densities (left panels) and bulk velocity of mixtures (right panels).}\label{fig1}
\end{figure}
In Figure \ref{fig1}, we plot the species densities $\rho_1(=n_1)$, $\rho_2(=n_2)$ and the bulk velocity $u$ at time $t=0.25$. We observe that the macroscopic quantities associated with numerical solutions to the BGK model converge to the solution to the isentropic Euler system as $\varepsilon$ goes to zero. These numerical evidences support the theoretical results in Theorem \ref{thm:formal} and Proposition \ref{prop}.

%
%
%
%
%
%
%
%
%
%
%
%
\subsection{Riemann problem II}
To observe the effect of choices of $\gamma_1$ and $\gamma_2$, we consider the following initial macroscopic variables:
\begin{align*}
	\rho_{10}(x)=\begin{cases}
		\rho_L,\quad x\leq 0\\
		0.5.\quad x>0
	\end{cases},\quad 	\rho_{20}(x)=\begin{cases}
		1,\quad x\leq 0\\
		0.25\quad x>0
	\end{cases},
\end{align*}
where $\rho_L \in \{1,2,3\}$. As in the previous test, we take the equilibrium attractors as initial data and compare the numerical solutions at $t=0.25$, imposing the free-flow boundary condition on the spatial domain. Here, we set $\gamma_1=2,3$ and $\gamma_2=7/5$ with a fixed Knudsen number $\varepsilon=10^{-6}$. To resolve the distribution function correctly, we enlarge the velocity domain depending on $\rho_L$ and use sufficiently many velocity grids. We report the associated velocity domain and the size of $\Delta v$ in Figures \ref{fig2 2}-\ref{fig2 3}. 
\begin{figure}[htbp]
	\centering
	\begin{subfigure}[h]{0.49\linewidth}
		\includegraphics[width=1\linewidth]{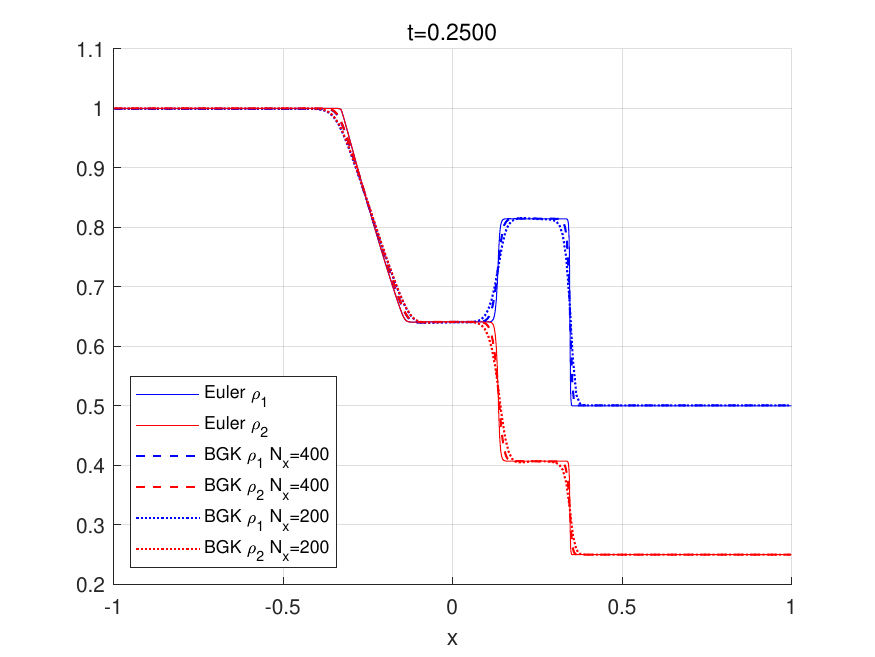}
	\end{subfigure}	
	\begin{subfigure}[h]{0.49\linewidth}
		\includegraphics[width=1\linewidth]{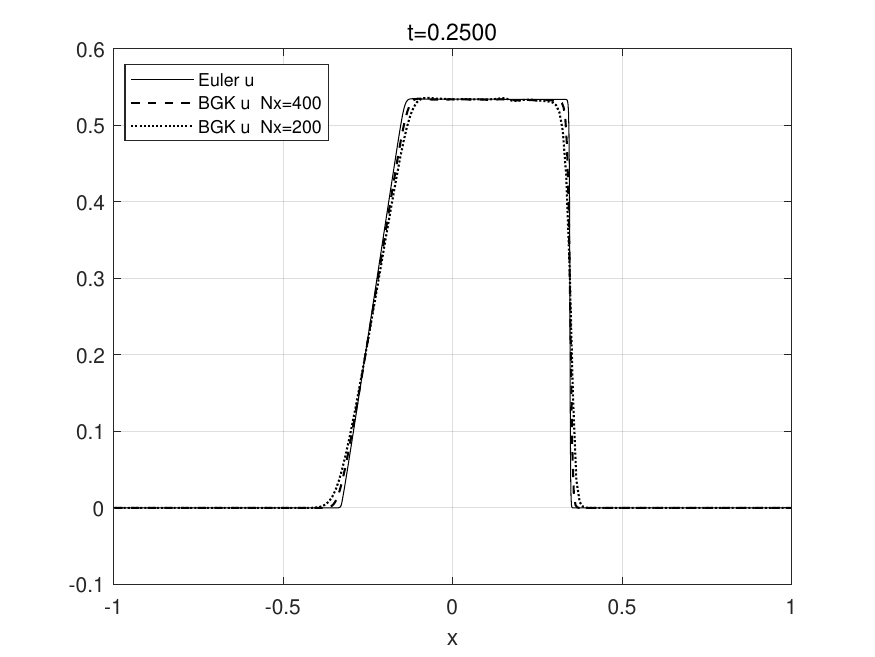}
	\end{subfigure}	
	\caption*{Case 1: $\rho_L=1$, $v\in[-3,3]$, $N_v=3000$.}
		\begin{subfigure}[h]{0.49\linewidth}
		\includegraphics[width=1\linewidth]{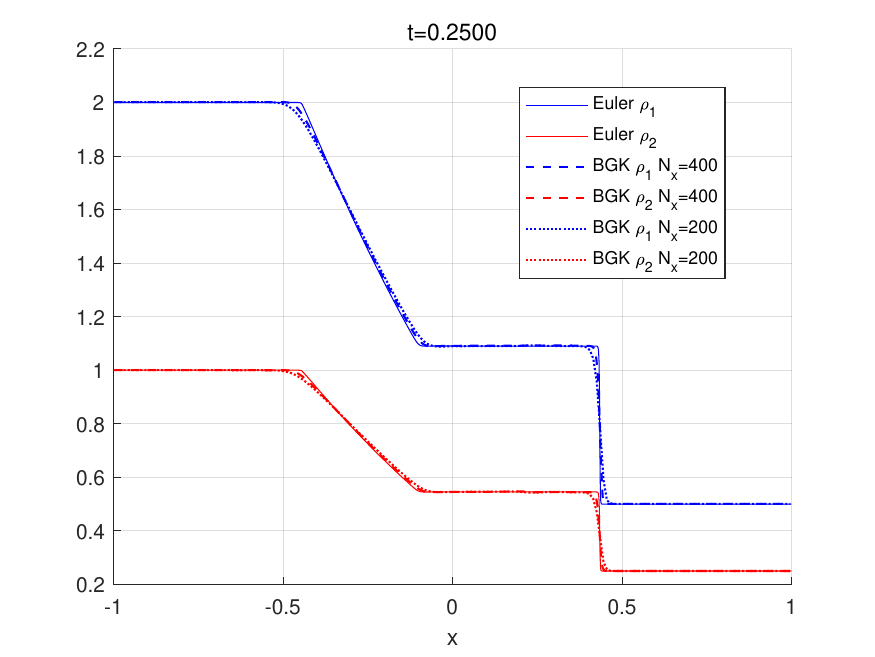}
	\end{subfigure}	
	\begin{subfigure}[h]{0.49\linewidth}
		\includegraphics[width=1\linewidth]{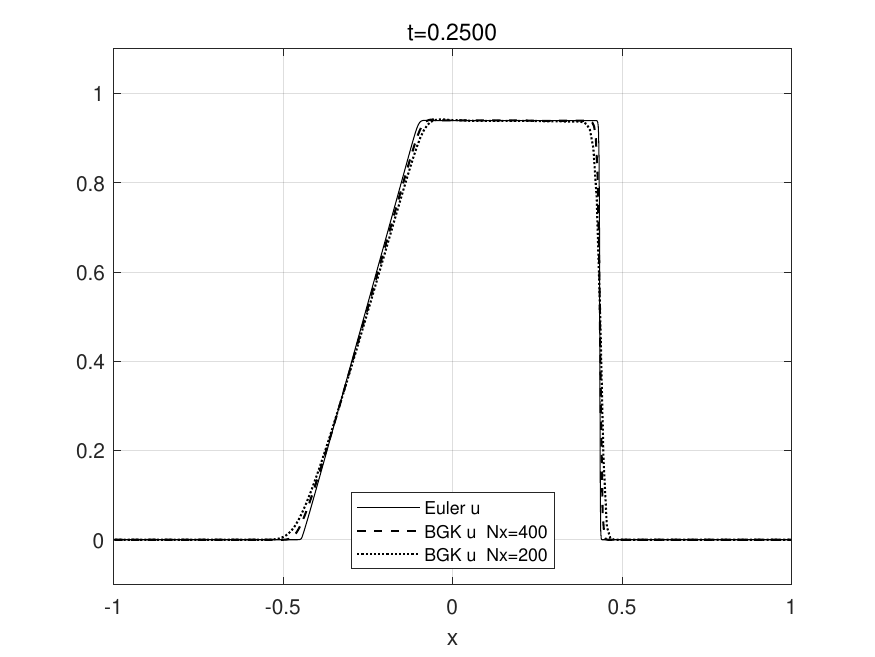}
	\end{subfigure}	
	\caption*{Case 2: $\rho_L=2$, $v\in[-5,5]$, $N_v=5000$.}
		\begin{subfigure}[h]{0.49\linewidth}
		\includegraphics[width=1\linewidth]{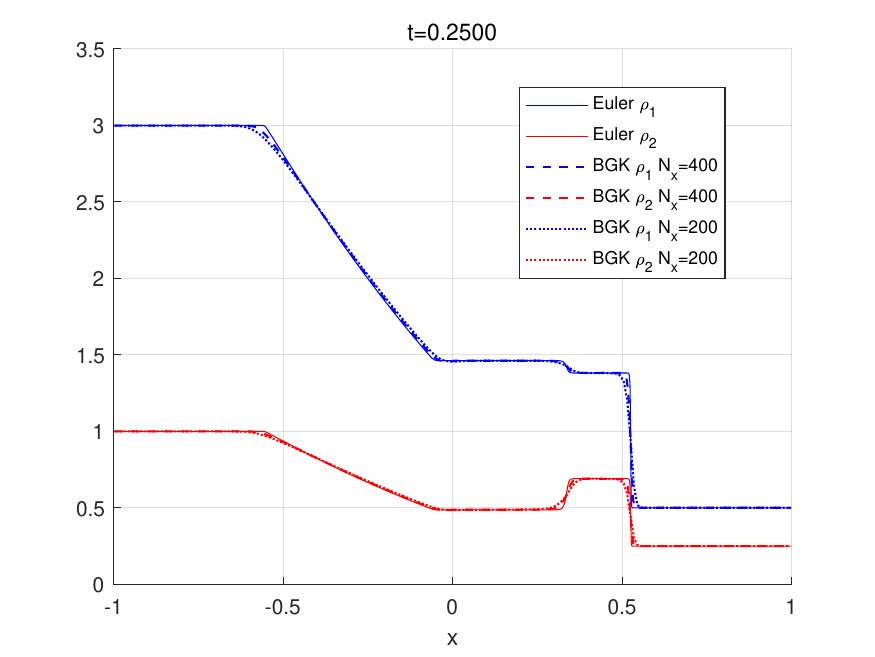}
	\end{subfigure}	
	\begin{subfigure}[h]{0.49\linewidth}
		\includegraphics[width=1\linewidth]{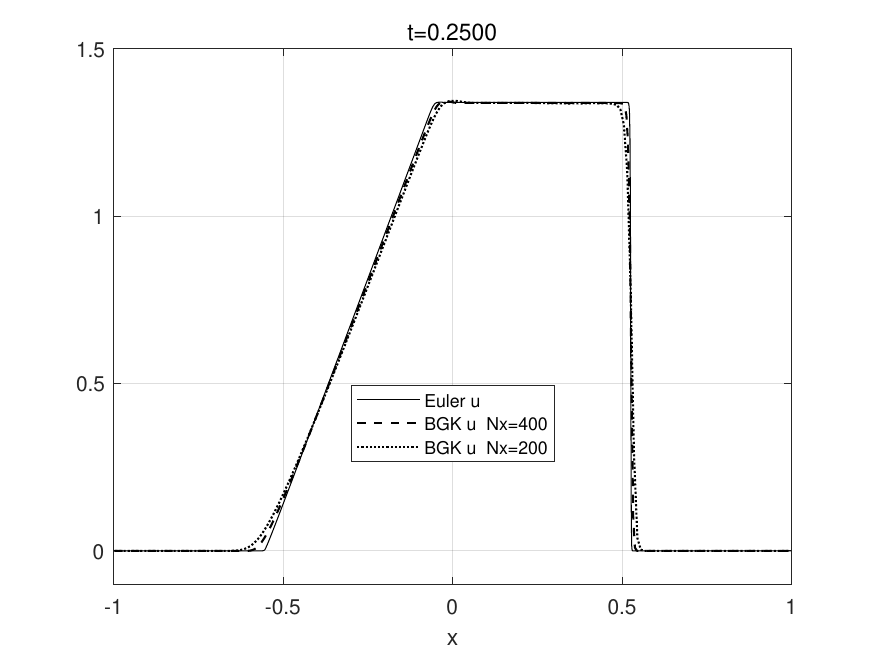}
	\end{subfigure}	
	\caption*{Case 3: $\rho_L=3$, $v\in[-8,8]$, $N_v=8000$.}
	\caption{Riemman problem II: Comparison of species densities (left panels) and bulk velocity of mixtures (right panels) when $\gamma_1=2$ and $\gamma_2=7/5$.}\label{fig2 2}
\end{figure}		
\begin{figure}[htbp]
	\centering
	\begin{subfigure}[h]{0.49\linewidth}
		\includegraphics[width=1\linewidth]{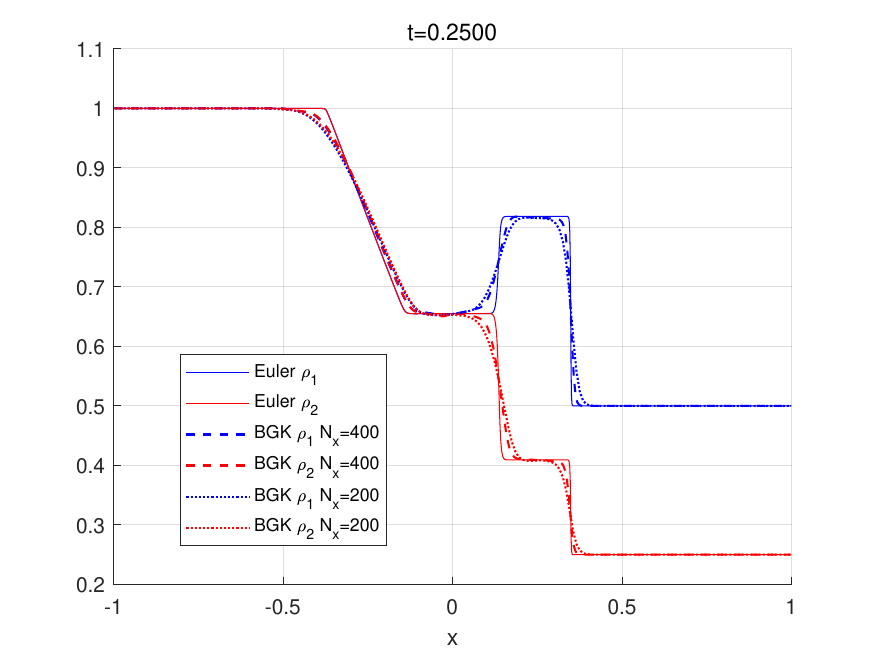}
	\end{subfigure}	
	\begin{subfigure}[h]{0.49\linewidth}
		\includegraphics[width=1\linewidth]{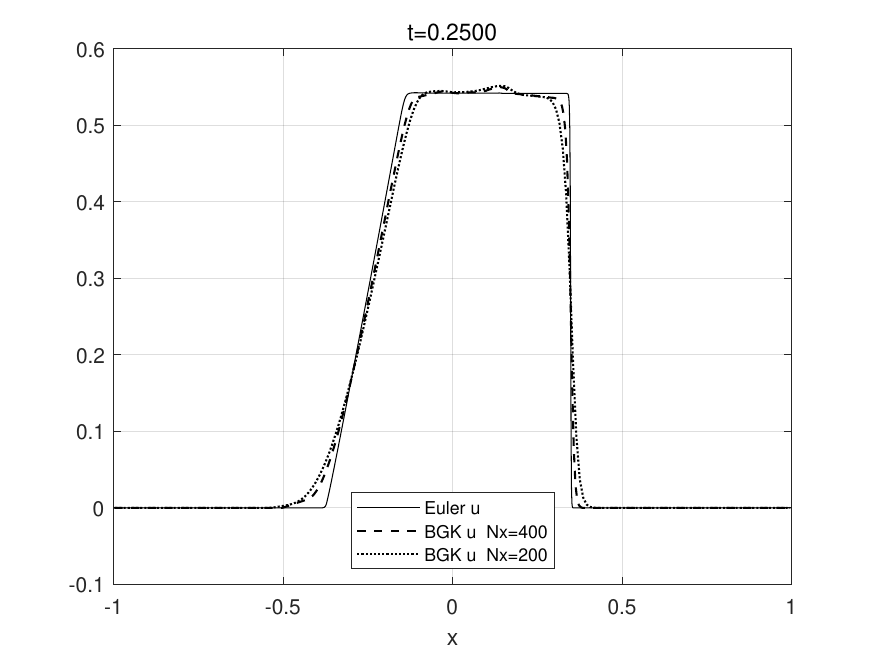}
	\end{subfigure}	
	\caption*{Case 1: $\rho_L=1$, $v\in[-3,3]$, $N_v=24000$.}
	\begin{subfigure}[h]{0.49\linewidth}
		\includegraphics[width=1\linewidth]{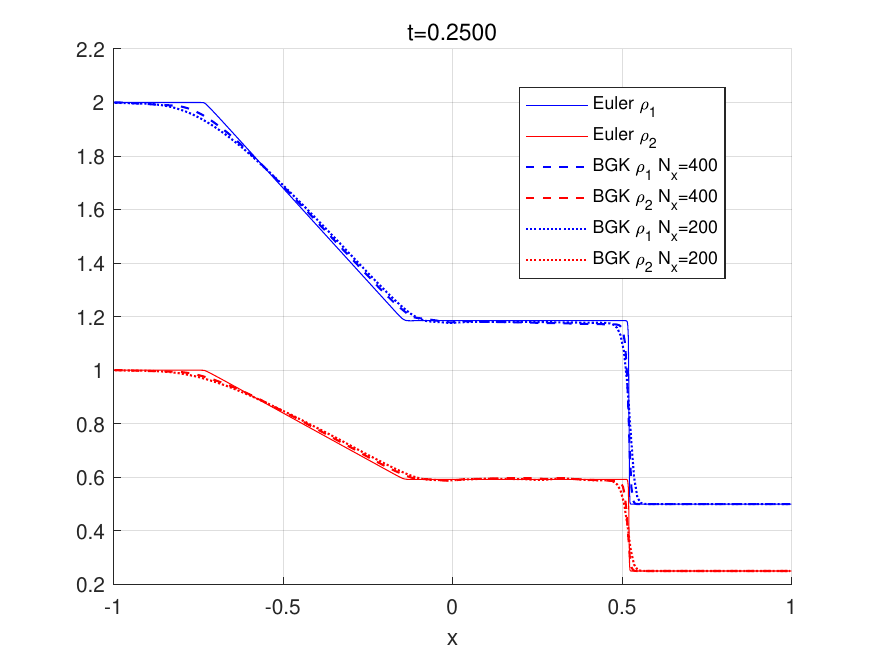}
	\end{subfigure}	
	\begin{subfigure}[h]{0.49\linewidth}
		\includegraphics[width=1\linewidth]{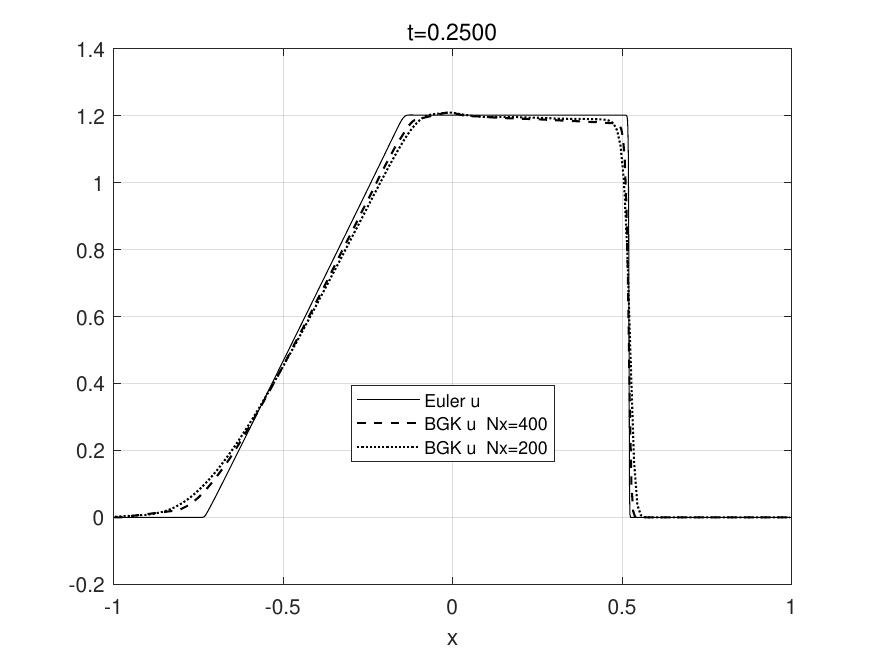}
	\end{subfigure}	
	\caption*{Case 2: $\rho_L=2$, $v\in[-5,5]$, $N_v=40000$.}
	\begin{subfigure}[h]{0.49\linewidth}
		\includegraphics[width=1\linewidth]{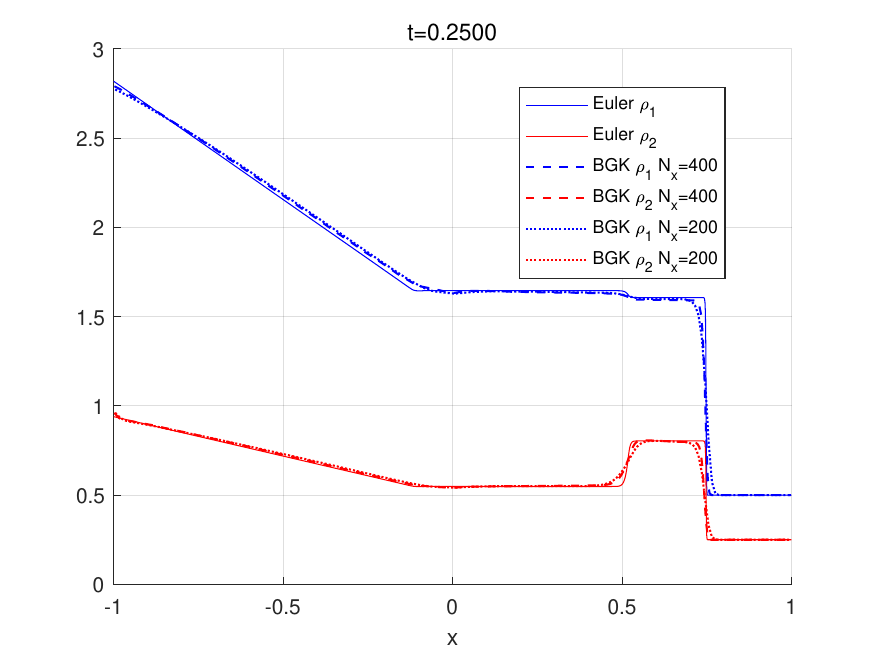}
	\end{subfigure}	
	\begin{subfigure}[h]{0.49\linewidth}
		\includegraphics[width=1\linewidth]{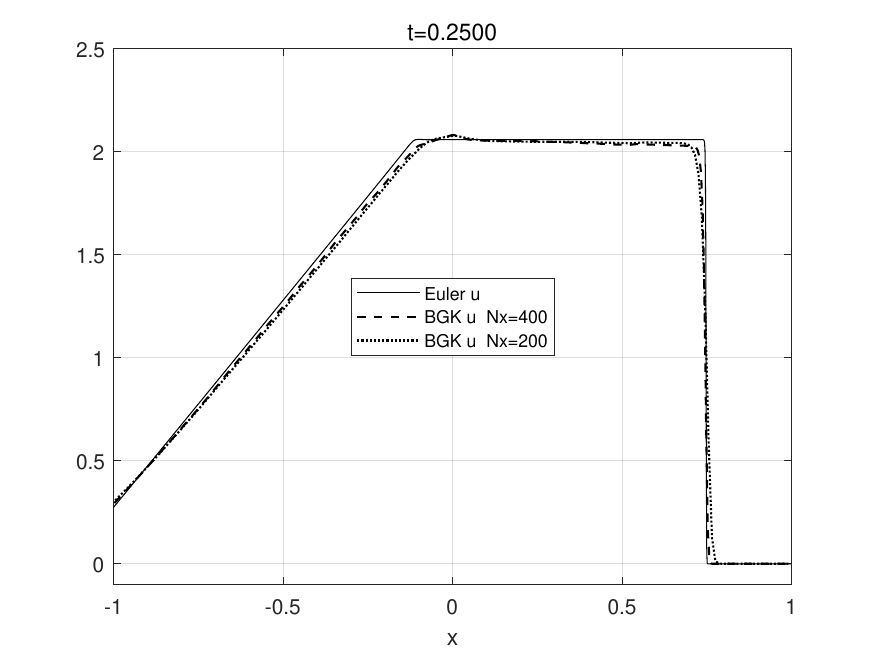}
	\end{subfigure}	
	\caption*{Case 3: $\rho_L=3$, $v\in[-8,8]$, $N_v=64000$.}
	\caption{Riemman problem II: Comparison of species densities (left panels) and velocity (right panels) when $\gamma_1=3$ and $\gamma_2=7/5$.}\label{fig2 3}
\end{figure}		
In the figures, it is observed that numerical solutions associated with the BGK model with finer grid $N_x=400$ lead to sharper and better approximation of the reference solution. We also note that larger values of bulk velocity $\textbf{u}$ of the mixture are observable as the initial density $\rho_L$ and $\gamma_1$ are taken bigger as expected..

%
%
%
%
%
 \section*{Acknowledgments}

 S. Y. Cho was supported by Learning \& Academic research institution for Master's$\cdot$PhD students, and Postdocs (LAMP) Program of the National Research Foundation of Korea (NRF) grant funded by the Ministry of Education (No. RS-2023-00301974).  The research of YPC and SS was supported by the NRF grant no. 2022R1A2C1002820 and RS-202400406821. B.-H. Hwang was supported by Basic Science Research Programs through the National Research Foundation of Korea (NRF) funded by the Ministry of Education (NRF-2019R1A6A1A10073079 and RS-2024-00462755).

%
%
%
%
%

\appendix
\section{On truncated Maxwellians}\label{app_deri}
\subsection{Derivation via minimization principle}
This appendix is devoted to the derivation of the equilibrium distribution $\mathcal{M}_i$ via an entropy minimization principle, following the approach pioneered in \cite{B99}. As a preliminary step, we present an auxiliary lemma that provides several key integral identities. These moment formulas for truncated functions will be instrumental in both the asymptotic expansions and the hydrodynamic limit analysis.

\begin{lemma}\label{lem_vm0}
Let $\alpha \geq 0$ and $v = (v_1, \dots, v_n) \in \R^n$. Then the following integral identities hold:
\begin{enumerate}[label=(\roman*)]
\item $\displaystyle \int_{\mathrm{B}_1} (1 - |v|^2)^\alpha \, dv = \frac{1}{2} |\partial \mathrm{B}_1| \, B\left(\frac{n}{2}, \alpha + 1\right)$,
\item $\displaystyle \int_{\mathrm{B}1} v \otimes v (1 - |v|^2)^\alpha\, dv = \frac{1}{2n} |\partial \mathrm{B}_1|\, B\left(\frac{n}{2}+1, \alpha + 1\right) \mathbb{I}_{n \times n}$,
\item $\displaystyle \int_{\mathrm{B}_1} |v|^2 v_a^2 (1 - |v|^2)^\alpha \, dv = \frac{1}{2n} |\partial \mathrm{B}_1| \, B\left(\frac{n}{2}+2, \alpha + 1\right)$,
\item $\displaystyle \int_{\mathrm{B}_1} v_a^2 v_b^2 (1 - |v|^2)^\alpha \, dv = \frac{1}{2n(n+2)} |\partial \mathrm{B}_1| \, B\left(\frac{n}{2}+2, \alpha + 1\right)$ for $a \neq b$,
\item $\displaystyle \int_{\mathrm{B}_1} v_a^4 (1 - |v|^2)^\alpha \, dv = \frac{3}{2n(n+2)} |\partial \mathrm{B}_1| \, B\left(\frac{n}{2}+2, \alpha + 1\right)$.
\end{enumerate}
Here $B(\cdot, \cdot)$ is the beta function defined as
$$
B(p,q) = \int_0^1 t^{p-1}(1 - t)^{q-1} dt = \frac{\Gamma(p)\Gamma(q)}{\Gamma(p+q)}.
$$
\end{lemma}

\begin{proof}
We evaluate these integrals using spherical coordinates. Set $v = r\omega$, with $r \in [0,1]$ and $\omega \in \partial \mathrm{B}_1$, then $dv = r^{n-1} dr d\omega$ and $|v| = r$. A homogeneous polynomial $g(v)$ of degree $k$ decomposes as $g(r\omega) = r^k h(\omega)$, with $h(\omega)$ homogeneous on the sphere.

Thus,
$$
\int_{{\rm B}_1} g(v)(1 - |v|^2)^\alpha dv = \lt(\int_{\pa {\rm B}_1} h(\omega) \,d\omega\rt) \lt(\int_0^1 r^{n - 1 + k} (1 - r^2)^\alpha \,dr\rt).
$$
The angular part is evaluated using standard identities:
\begin{align*}
\int_{\pa {\rm B}_1} \omega_a^2 d\omega &= \frac{1}{n} |\pa {\rm B}_1|,  \\
\int_{\pa {\rm B}_1} \omega_a^4 d\omega &= \frac{3}{n(n+2)} |\pa {\rm B}_1|, \\
\int_{\pa {\rm B}_1} \omega_a^2 \omega_b^2 d\omega &= \frac{1}{n(n+2)} |\pa {\rm B}_1| \quad (a \ne b).
\end{align*}
These identities result from rotational symmetry and the identity $\sum_{a=1}^n \omega_a^2 = 1$ on the sphere. Specifically,
\[
\int_{\pa {\rm B}_1} \omega_a^2 \, d\omega = \frac{1}{n} \int_{\pa {\rm B}_1} \sum_{a=1}^n \omega_a^2 \, d\omega = \frac{1}{n} |\pa {\rm B}_1|.
\]

We next consider the rotationally symmetric fourth-order tensor.  Note that 
\[
\int_{\pa {\rm B}_1} \left( \sum_{a=1}^n \omega_a^2 \right)^2 d\omega = \int_{\pa {\rm B}_1} 1^2\, d\omega = |\pa {\rm B}_1|.
\]
On the other hand, expanding the square gives
\[
 \left( \sum_{a=1}^n \omega_a^2 \right)^2 = \sum_{a=1}^n \omega_a^4 + 2 \sum_{1 \le a < b \le n} \omega_a^2 \omega_b^2.
\]
Note that for $a \neq b$
\[
   \int_{\pa {\rm B}_1}\omega_a^2 \omega_b^2\,d\omega =   \frac{ \lt( \int_0^\pi \sin^n\varphi_1\cos^2\varphi_1\,d\varphi_1\rt) \lt(\int_0^\pi \sin^{n-3}\varphi_2\cos^2\varphi_2\,d\varphi_2\rt)}{ \lt(\int_0^\pi \sin^{n-2}\varphi_1\,d\varphi_1\rt) \lt(\int_0^\pi \sin^{n-3}\varphi_2\,d\varphi_2\rt)}|\pa {\rm B}_1| = \frac{1}{n(n+2)}|\pa {\rm B}_1|,
\]
where we used
\[
\int_0^\pi \sin^n x\,dx=\frac{n-1}{n}\int_0^\pi\sin^{n-2}x\,dx
\]
so that
\[
\int_0^\pi \sin^n\varphi_1\cos^2\varphi_1\,d\varphi_1=\frac{n-1}{n(n+2)}\int_0^\pi \sin^{n-2}\varphi_1\,d\varphi_1
\]
and
\[ 
\int_0^\pi \sin^{n-3}\varphi_2\cos^2\varphi_2\,d\varphi_2=\frac{1}{n-1}\int_0^\pi \sin^{n-3}\varphi_2\,d\varphi_2.
\]
This together with symmetry implies
\[
|\pa {\rm B}_1| = n \int_{\pa {\rm B}_1} \omega_1^4 \, d\omega  + n(n - 1) \int_{\pa {\rm B}_1} \omega_1^2 \omega_2^2 \, d\omega = n \int_{\pa {\rm B}_1} \omega_1^4 \, d\omega + \frac{n-1}{n+2}|\pa {\rm B}_1|.
\]
Thus,
\[
\int_{\pa {\rm B}_1} \omega_a^4 \, d\omega = \frac{3}{n(n+2)} |\pa {\rm B}_1|.
\]
The radial integral is computed via the substitution $s = r^2$, yielding
$$
\int_0^1 r^{m} (1 - r^2)^\alpha dr = \frac{1}{2} B\left(\frac{m+1}{2}, \alpha + 1\right).
$$
This produces the stated identities.
\end{proof}

To characterize the equilibrium distributions $\mathcal{M}_i$, we now turn to an entropy minimization framework under suitable macroscopic constraints. For $\gamma_i \in \left(1,\frac{n+2}{n}\right)$ and $m_i>0$, we define the entropy density function 
\begin{align} \label{eq: def: h}
    h_{i}(f_i, v) := \frac{m_i|v|^2}{2}f_i + \frac{1}{2c_i^{2/d_i}} \frac{f_i^{1+2/d_i}}{1+2/d_i}
\end{align}
and set
\begin{align*}
    H(p_1,p_2,v) := \sum_{i=1}^2 \nu_i  h_i(p_i, v).
\end{align*}
Given $\gamma_1, \gamma_2 \in \left(1,\frac{n+2}{n}\right)$ and nonnegative functions $f_1, f_2$, our goal is to minimize the total entropy 
\[
\begin{split}
    \argmin \int_{\R^n} H(p_1,p_2,v)\,dv &= \argmin \int_{\R^n} \sum_{i=1}^2 \nu_i h_{i}(p_i, v)\, dv  \\
    &= \argmin \left\{\int_{\R^n} \sum_{i=1}^2 \nu_i \Big(\frac{m_i|v|^2}{2}p_i + \frac{1}{2c_i^{2/d_i}} \frac{p_i^{1+2/d_i}}{1+2/d_i} \Big) \, dv\right\},
    \end{split}
\]
subject to the following constraints:
\begin{align}
\label{eq: constraints}
    \begin{cases}
        p_i(v) \ge 0 \quad i=1,2,\\[2mm]
       \displaystyle  \int_{\R^n} (p_i - f_i) \,dv= 0 \quad i=1,2,\\[2.5mm]
       \displaystyle   \sum_{i=1}^2  \nu_i \int_{\R^n} v  m_i(f_i - p_i) \,dv = 0.
    \end{cases}
\end{align}
Let $\mathcal{X}_i$ denote the admissible class for  $p_i$:
\[
    \calX_i := \lt\{p:\R^n \to \R, \, p\ge 0, \, p\in L^1(\R^n,(1+|v|^2)dv) \cap L^{1+2/d_i}(\R^n) \rt\}.
\]
We introduce the Lagrangian functional $\calL:\calX_1 \times \calX_2 \times\R^2\times\R^n \to \R$ defined by
\begin{align*}
    \calL(p_1,p_2,\eta,\lambda) = \int_{\R^n} H(p_1,p_2,v) \,dv - \begin{pmatrix} \eta \\ \lambda \end{pmatrix} \cdot \int_{\R^n} \begin{pmatrix} p_1 - f_1 \\ p_2 - f_2 \\ v\Big( \nu_1 m_1(p_1-f_1) + \nu_2 m_2(p_2 - f_2) \Big) \end{pmatrix} dv,
\end{align*}
where $(\eta,\lambda) = (\eta_1,\eta_2,\lambda)\in \R \times \R \times \R^n$ are the Lagrange multipliers associated with the mass and momentum constraints. 

Let $(\mathcal{M}_1,\mathcal{M}_2,\eta^*,\lambda^*)$ be a critical point of this functional. Then, for any $p \in \mathcal{X}_1$ satisfying the constraints \eqref{eq: constraints}, the directional derivative of $\mathcal{L}$ along $(p, 0)$ must vanish: 
\begin{align*}
    \delta\calL((\mathcal{M}_1, \mathcal M_2, \eta^*, \lambda^*); (p,0,0,0) ) := \frac{d}{d\varepsilon} \calL(\mathcal M_1 + \varepsilon p, \mathcal M_2, \eta^*, \lambda^*) \Big|_{\varepsilon = 0} = 0,
\end{align*}
whenever it is true that $\mathcal{M}_1 + \varepsilon p \ge 0$ (for this is when the right-hand side is well-defined). Note in particular that this implies the above holds for each $p\in \calX_1$ that is supported only on $V_1 := \{v\in \R^n: \mathcal M_1 > 0\}$. By computing the variation, we obtain
\begin{align*}
    \delta \calL((\mathcal M_1, \mathcal{M}_2, \eta^*, \lambda^*);(p,0,0,0)) &= \int_{V_1} p(v)\cdot \Big( \frac{\nu_1 m_1}{2}|v|^2 + \frac{\nu_1}{2c_1^{2/d_1}}\mathcal M_1^{2/d_1} - \eta^*_1 -  \nu_1 m_1 \lambda^* \cdot v \Big) \,dv = 0,
\end{align*}
which implies the Euler--Lagrange condition:
\begin{align*}
    \frac{\nu_1}{2c_1^{2/d_1}}\mathcal M_1^{2/d_1} = -\frac{\nu_1 m_1}{2}|v|^2 + \eta_1^* + \nu_1 m_1\lambda^*\cdot v \quad \text{in}\quad V_1.
\end{align*}
By positivity of $\mathcal{M}_1$ in $V_1$, the right-hand side must be strictly positive there. At the boundary $\partial V_1 := \{ v : \mathcal{M}_1(v) = 0 \}$, convex analysis (see \cite[Section 3.1]{B99}) yields the inequality
\begin{align*}
    \frac{m_1}{2}|v|^2 \le \frac{\eta^*_1}{\nu_1} +  m_1\lambda^*\cdot v \quad \text{on}\quad \pa V_1.
\end{align*}
Hence, we conclude that the equilibrium distribution takes the truncated polynomial form
\begin{align*}
    \mathcal M_1 = c_1 \Big(-m_1|v|^2 + 2(\frac{\eta_1^*}{\nu_1} +  \nu_1 m_1\lambda^*\cdot v) \Big)_+^{d_1/2}.
\end{align*}
A symmetric derivation yields the corresponding formula for $\mathcal{M}_2$:
\begin{align*}
    \mathcal M_2 = c_1 \Big(-m_2|v|^2 + 2(\frac{\eta_2^*}{\nu_2} + \nu_2 m_2\lambda^*\cdot v) \Big)_+^{d_2/2}.
\end{align*}

We now identify the explicit expressions for the Lagrange multipliers $\eta^*$ and $\lambda^*$ by utilizing the constraints \eqref{eq: constraints}, and in doing so, we obtain the explicit form of the truncated Maxwellians.

We begin with the mass and momentum identities for $\mathcal{M}_1$. By Lemma \ref{lem_vm0} (i) and the change of variables $v \mapsto v -  \lambda^*$, we find
\[
\begin{split}
    n_1 = \int_{\R^n} \mathcal M_1(v) \,dv     &=c_1 \int_{\R^n}  \Big(- m_1 |v|^2 + 2\frac{\eta_1^*}{\nu_1} + m_1  | \lambda^*|^2 \Big)^{d_1/2}_+ \, dv \\
     &= c_1 \frac{( 2\frac{\eta_1^*}{\nu_1} + m_1 | \lambda^*|^2)^{\frac{d_1 + n}2}}{m_1^{\frac n2}} \int_{B_1}  (1 - |v|^2)^{d_1/2}_+ \, dv \\
     &= c_1 \frac{( 2\frac{\eta_1^*}{\nu_1} + m_1 | \lambda^*|^2)^{\frac1{\gamma_1 - 1}}}{m_1^{\frac n2}} \frac12 |\pa {\rm B}_1| B\lt( \frac n2, \frac{d_1}2 +1 \rt). 
    \end{split}
\]
Using the following 
\[
|\pa {\rm B}_1| = \frac{2\pi^{\frac{n}{2}}}{\Gamma\left(\frac{n}{2}\right)}, \quad
c_i = \left( \frac{2\gamma_i}{\gamma_i - 1} \right)^{-\frac{1}{\gamma_i - 1}} \frac{\Gamma\left( \frac{\gamma_i}{\gamma_i - 1} \right)}{\pi^{\frac{n}{2}} \Gamma\left( \frac{d_i}{2} + 1 \right)}, \quad B\left( \frac{n}{2}, \frac{d_i}{2} \right) = \frac{\Gamma\left(\frac{n}{2}\right)\Gamma\left(\frac{d_i}{2}\right)}{\Gamma\left(\frac{n + d_i}{2}\right)},
\]
we get
\[
\frac{c_1}2 |\pa {\rm B}_1| B\lt( \frac n2, \frac{d_1}2 +1 \rt) = \left( \frac{2\gamma_1}{\gamma_1 - 1} \right)^{-\frac{1}{\gamma_1 - 1}}
\]
and thus
\[
    n_1 = \frac{( 2\frac{\eta_1^*}{\nu_1} + m_1 | \lambda^*|^2)^{\frac1{\gamma_1 - 1}}}{m_1^{\frac n2}} \left( \frac{2\gamma_1}{\gamma_1 - 1} \right)^{-\frac{1}{\gamma_1 - 1}}.
\]
This yields
\[
2\frac{\eta_1^*}{\nu_1} + m_1 | \lambda^*|^2 = \left( \frac{2\gamma_1}{\gamma_1 - 1} \right) m_1^{\frac n2 (\gamma_1 - 1)} n_1^{\gamma_1 - 1}.
\]

Next, using the symmetry of the integrand, we compute the momentum as
\bq\label{eq: calc vM1}
\int_{\mathbb{R}^n} v \mathcal{M}_1(v)\, dv =  \lambda^* n_1.
\eq
A symmetric computation yields
\[
\int_{\mathbb{R}^n} v \mathcal{M}_2(v)\, dv =   \lambda^* n_2.
\]
Applying the momentum constraint \eqref{eq: constraints}$_3$, we deduce
\bq\label{eq: lambda*}
\lambda^* = \frac{ \nu_1 \rho_1 u_1 +  \nu_2 \rho_2 u_2}{\nu_1 \rho_1 +  \nu_2 \rho_2} =: u.
\eq
Plugging this into the expression for $\eta_1^*$, we obtain
\[
2\frac{\eta_1^*}{\nu_1} + m_1 |u|^2 =  \left( \frac{2\gamma_1}{\gamma_1 - 1} \right) m_1^{\frac n2 (\gamma_1 - 1)} n_1^{\gamma_1 - 1}.
\]

Therefore, we recover the explicit formula for $\mathcal{M}_1 = \mathcal{M}_1[n_1, u](v)$:
\[
\mathcal{M}_1[n_1, u](v) = c_1 \left( -m_1 |v - u|^2 + m_1^{\frac{n}{2}(\gamma_1 - 1)}  \frac{2\gamma_1}{\gamma_1 - 1} n_1^{\gamma_1 - 1} \right)_+^{d_1/2}.
\]
Similarly, we obtain the expression for $\mathcal{M}_2$:
\[
\mathcal{M}_2[n_2, u](v) = c_2 \left( -m_2 |v - u|^2 + m_2^{\frac{n}{2}(\gamma_2 - 1)}  \frac{2\gamma_2}{\gamma_2 - 1} n_2^{\gamma_2 - 1} \right)_+^{d_2/2}.
\]
These explicit forms confirm that the equilibrium distribution functions are truncated polynomials supported on velocity balls centered at the bulk velocity $u$, with radii determined by the local densities and parameters $m_i$, $\gamma_i$. The case $\gamma_i = 1 + \frac2n$ can be treated similarly by taking the appropriate limit.

%
%
%
%
%
\subsection{Macroscopic identities for Maxwellians}
In this subsection, we derive several macroscopic identities satisfied by the truncated Maxwellian distributions. These identities correspond to the conserved moments and energy functional evaluated at equilibrium. 
\begin{lemma}\label{lem_ident}
Then the following identities hold:
    \begin{align}
        &\int_{\R^n} \calM_i[n_i, u](v)\,dv = n_i, \nonumber \\
        &\int_{\R^n} v\calM_i[n_i, u](v)\,dv = n_i u, \label{eq: vM}\\
                &\int_{\R^n} (v\otimes v) \calM_i[n_i, u](v)\,dv = n_i u \otimes u + m_i^{\frac{n(\gamma_i-1)}{2}-1} n_i^{\gamma_i} \label{eq: tens prod M1}\mathbb{I}_{n \times n},\\
        &\int_{\R^n} h_i(\calM_i[n_i, u](v),v)\,dv = \frac{m_i}{2}n_i|u|^2 + \frac{m_i^{\frac{n(\gamma_i-1)}{2}}}{\gamma_i - 1} n_i^{\gamma_i}\nonumber ,
    \end{align}
    where $h_i$ is given by \eqref{eq: def: h}. 
\end{lemma}
\begin{proof} For brevity, we only compute the quantities associated with $\mathcal{M}_1$; the case for $\mathcal{M}_2$ follows analogously by symmetry. We also consider the case $\gamma_i \in (1, 1 + \frac2n)$ since the case $\gamma_i = 1 + \frac2n$ can be treated similarly.

    The first assertion is just by construction. The momentum identity \eqref{eq: vM} follows from  \eqref{eq: calc vM1} and \eqref{eq: lambda*}.  For the identity \eqref{eq: tens prod M1}, we observe that
    \[
    \int_{\R^n} (v\otimes v) \calM_1[n_1, u](v)\,dv = n_1 u \otimes u + \int_{\R^n} (v-u)\otimes (v-u) \calM_1[n_1, u](v)\,dv.
    \]
    On the other hand,
    \begin{align*}
    \int_{\R^n} (v-u)\otimes (v-u) \calM_1[n_1, u](v)\,dv &= \int_{\R^n} (v\otimes v) \calM_1[n_1, 0](v)\,dv\cr
    &= c_1 \int_{\mathbb{R}^n}  (v\otimes v) \left( m_1^{\frac{n}{2}(\gamma_1 - 1)} \frac{2\gamma_1}{\gamma_1 - 1} n_1^{\gamma_1 - 1} - m_1 |v|^2 \right)_+^{\frac{d_1}{2}}\, dv \\
&= \frac{c_1}{m_1^{\frac{n + 2}{2}}} \left( m_1^{\frac{n}{2}(\gamma_1 - 1)} \frac{2\gamma_1}{\gamma_1 - 1} n_1^{\gamma_1 - 1} \right)^{\frac{n + d_1 + 2}{2}} \int_{{\rm B}_1}(v\otimes v)  (1 - |v|^2)^{\frac{d_1}{2}}\, dv \\
&=c_1 m_1^{\frac{n}{2}(\gamma_1 - 1)-1} \lt(\frac{2\gamma_1}{\gamma_1 - 1}\rt)^{\frac{\gamma_1}{\gamma_1-1}} n_1^{\gamma_1} \frac{1}{2n} |\partial \mathrm{B}_1|\, B\left(\frac{n}{2}+1, \frac{d_1}2 + 1\right) \mathbb{I}_{n \times n}
    \end{align*}
    due to Lemma \ref{lem_vm0} (ii). Since 
    \[
   c_1  \lt(\frac{2\gamma_1}{\gamma_1 - 1}\rt)^{\frac{\gamma_1}{\gamma_1-1}}   \frac{1}{2n} |\partial \mathrm{B}_1|\, B\left(\frac{n}{2}+1, \frac{d_1}2 + 1\right) = 1,
    \]
    we get
    \[
     \int_{\R^n} (v-u)\otimes (v-u) \calM_1[n_1, u](v)\,dv = m_1^{\frac{n(\gamma_1-1)}{2}-1} n_1^{\gamma_1}\mathbb{I}_{n \times n},
    \]
    and thus the desired identity is obtained. 
 
 For the kinetic entropy estimate, as a direct consequence of the above, we deduce
 \begin{align*}
        \int_{\R^n} \frac{m_1|v|^2}{2}\calM_1[n_1, u](v)\,dv = \frac{m_1}{2}n_1|u|^2 + \frac{n}{2}m_1^{\frac{n(\gamma_1-1)}{2}} n_1^{\gamma_1} .
    \end{align*}
    Next, we calculate
    \begin{align*}
        \frac{1}{2c_1^{\frac{2}{d_1}}} \frac{1}{1+\frac{2}{d_1}} \int_{\R^n}  (\calM_1[n_1, u](v))^{1+\frac{2}{d_1}} \,dv 
                & = \int_{\R^n} \frac{1}{2c_1^{\frac{2}{d_1}}} \frac{1}{1+\frac{2}{d_1}} (\calM_1[n_1, 0](v))^{1+\frac{2}{d_1}} \,dv \\
                & = \frac{c_1}{2(1 + \frac2{d_1})} \frac{\lt(m_1^{\frac{n}{2}(\gamma_1 - 1)} \frac{2\gamma_1}{\gamma_1-1}n_1^{\gamma_1 - 1}\rt)^{\frac1{\gamma_1 - 1} + 1}}{m_1^{\frac n2}} \int_{{\rm B}_1} (1 - |v|^2)^{\frac{d_1+2}{2}}\,dv\cr
                &  = \frac{c_1}{2(1 + \frac2{d_1})} m_1^{\frac{n}{2}(\gamma_1 - 1)}\lt(\frac{2\gamma_1}{\gamma_1 - 1}\rt)^{\frac{\gamma_1}{\gamma_1-1}} n_1^{\gamma_1} \frac{1}{2} |\partial \mathrm{B}_1| \, B\left(\frac{n}{2}, \frac{d_1}{2} + 2\right)\cr
        &  = \frac{d_1}{2} m_1^{\frac{n(\gamma_1-1)}{2}} n_1^{\gamma_1}.
    \end{align*}
Combining those results concludes the desired result.
\end{proof}

\begin{remark}\label{rem: each h} 
The Maxwellians $\calM_i$ were obtained by minimizing the total entropy functional
\[
\sum_{i=1}^2 \intr  \nu_i  h_i(p_i, v)\,dv
\]
under the constraints of given macroscopic quantities. A noteworthy consequence of this construction is that each $\calM_i$ also minimizes the individual entropy $\intr h_i(f_i,v)\,dv$ among all functions with the same local moments. That is, for any  $f_1, f_2 \in \calX_1 \times \calX_2$, we have
\[
\intr h_i(\calM_i(f_1, f_2), v)\,dv \leq \intr h_i(f_i, v)\,dv, \quad i=1,2.
\]
This decoupled entropy-minimizing property will be crucial in the relative entropy method when establishing the rigorous hydrodynamic limit. 
\end{remark}

%
%
%
%
%
%
%
%
%
%

%
%
%
%
%

%
%
%
%
%
%
%
%
%
%
%
%
%

\section{Proof of Lemma \ref{lem_vm}: Velocity moments of truncated Maxwellians}\label{app_vm}

In this appendix, we provide a detailed proof of Lemma \ref{lem_vm}, which collects several key velocity moment identities associated with the truncated Maxwellian distributions $\mathcal{M}_i[n_i, u]$. These identities are essential in computing first-order corrections in the Chapman--Enskog expansion and appear repeatedly in the derivation of macroscopic quantities such as the pressure tensor and momentum flux.

For the reader's convenience, we recall below Lemma \ref{lem_vm}.

\begin{lemma} Let $\mathcal{M}_i[n_i, u]$ be the truncated Maxwellian defined by
\[
\mathcal{M}_i[n_i,u](v) = c_i \left( m_i^{\frac{n}{2}(\gamma_i-1)} \frac{2\gamma_i}{\gamma_i-1} n_i^{\gamma_i - 1} - m_i |v-u|^2 \right)_+^{\frac{d_i}{2}},
\]
where $c_i$ is a normalization constant. Then the following identities hold:
\begin{enumerate}[label=(\roman*)]
\item[(i)] Zeroth-order moment:
\[
c_i^{\frac{2}{d_i}}\int_{\R^n} \mathcal{M}_i^{\frac{d_i - 2}{d_i}}[n_i, 0](v) \,dv = \frac{1}{d_i \gamma_i m_i^{\frac{n}{2}(\gamma_i - 1)}} n_i^{2 - \gamma_i}.
\]
\item[(ii)] Second-order moment:
\[
c_i^{\frac{2}{d_i}}\int_{\R^n} (v  \otimes v) \mathcal{M}_i^{\frac{d_i - 2}{d_i}}[n_i, 0](v)\,dv = \frac{n_i}{d_i m_i} \mathbb{I}_{n \times n}.
\]
\item[(iii)] Graident-weighted fourth-order moment:
\begin{align*}
&c_i^{\frac{2}{d_i}} \sum_{p,q=1}^n \pa_{x_p} u_q \int_{\R^n} v_p  v_q  (v \otimes v) \mathcal{M}_i^{\frac{d_i - 2}{d_i}}[n_i, 0](v) \, dv \cr
&\quad = \frac{n}{n+2}\frac{\gamma_i}{d_i m_i^{2 - \frac n2(\gamma_i-1)}}n_i^{\gamma_i}\left\{(\nabla_x\cdot u) \mathbb{I}_{n\times n} +\nabla_x u+(\nabla_xu)^t \right\}.
\end{align*}
\end{enumerate}
\end{lemma}
 \begin{proof}
Let us write $\mathcal{M}_i := \mathcal{M}_i[n_i, 0]$ for simplicity. Each identity below is derived using the integral formulas collected in Lemma \ref{lem_vm0}.

Using Lemma \ref{lem_vm0} (i), we first compute the zeroth-order moment as
\begin{align*}
c_i^{\frac{2}{d_i}} \int_{\mathbb{R}^n} \mathcal{M}_i^{\frac{d_i - 2}{d_i}}(v)\, dv 
&= c_i \int_{\mathbb{R}^n} \left( m_i^{\frac{n}{2}(\gamma_i - 1)} \frac{2\gamma_i}{\gamma_i - 1} n_i^{\gamma_i - 1} - m_i |v|^2 \right)_+^{\frac{d_i - 2}{2}} dv \\
&= \frac{c_i}{m_i^{\frac{n}{2}}} \left( m_i^{\frac{n}{2}(\gamma_i - 1)} \frac{2\gamma_i}{\gamma_i - 1} n_i^{\gamma_i - 1} \right)^{\frac{n + d_i - 2}{2}} \int_{{\rm B}_1} (1 - |v|^2)^{\frac{d_i - 2}{2}}\, dv \\
&= \frac{c_i}{m_i^{\frac{n}{2}(\gamma_i - 1)}} \left( \frac{2\gamma_i}{\gamma_i - 1} n_i^{\gamma_i - 1} \right)^{\frac{2 - \gamma_i}{\gamma_i - 1}} \cdot \frac{1}{2} |\pa {\rm B}_1| B\left( \frac{n}{2}, \frac{d_i}{2} \right).
\end{align*}
Using the known expressions:
\[
|\pa {\rm B}_1| = \frac{2\pi^{\frac{n}{2}}}{\Gamma\left(\frac{n}{2}\right)}, \quad
c_i = \left( \frac{2\gamma_i}{\gamma_i - 1} \right)^{-\frac{1}{\gamma_i - 1}} \frac{\Gamma\left( \frac{\gamma_i}{\gamma_i - 1} \right)}{\pi^{\frac{n}{2}} \Gamma\left( \frac{d_i}{2} + 1 \right)}, \quad B\left( \frac{n}{2}, \frac{d_i}{2} \right) = \frac{\Gamma\left(\frac{n}{2}\right)\Gamma\left(\frac{d_i}{2}\right)}{\Gamma\left(\frac{n + d_i}{2}\right)},
\]
we deduce
\[
c_i^{\frac{2}{d_i}} \int_{\mathbb{R}^n} \mathcal{M}_i^{\frac{d_i - 2}{d_i}}(v)\, dv = \frac{1}{d_i \gamma_i m_i^{\frac{n}{2}(\gamma_i - 1)}} n_i^{2 - \gamma_i}.
\]

Using Lemma \ref{lem_vm0} (ii), the isotropy of the integrand gives
\begin{align*}
c_i^{\frac{2}{d_i}} \int_{\mathbb{R}^n} (v \otimes v) \, \mathcal{M}_i^{\frac{d_i - 2}{d_i}}(v)\, dv 
&= \frac{c_i}{m_i^{\frac{n}{2} + 1}} \left( m_i^{\frac{n}{2}(\gamma_i - 1)} \frac{2\gamma_i}{\gamma_i - 1} n_i^{\gamma_i - 1} \right)^{\frac{n + d_i}{2}} \int_{{\rm B}_1} (v \otimes v) (1 - |v|^2)^{\frac{d_i - 2}{2}}\, dv \\
&= \frac{c_i}{m_i} \left( \frac{2\gamma_i}{\gamma_i - 1} n_i^{\gamma_i - 1} \right)^{\frac{1}{\gamma_i - 1}} \cdot \frac{1}{2n} |\pa {\rm B}_1| B\left( \frac{n}{2} + 1, \frac{d_i}{2} \right) \mathbb{I}_{n \times n}.
\end{align*}
Note that 
\[
B\left(\frac n2+1,\frac {d_i}{2}\right)=B\left(\frac n2,\frac {d_i}{2}\right)\frac{\frac n2}{\frac n2+\frac {d_i}{2}}=B\left(\frac n2,\frac {d_i}{2}+1\right)\frac{n}{d_i}=\frac{\Gamma(\frac n2)\Gamma(\frac {d_i}{2}+1)}{\Gamma(\frac{\gamma_i}{\gamma_i-1})}\frac{n}{d_i}.
\]
Substituting the constants, this simplifies to
\[
c_i^{\frac{2}{d_i}} \int_{\mathbb{R}^n} (v - u) \otimes (v - u) \mathcal{M}_i^{\frac{d_i - 2}{d_i}}[n_i, u](v)\, dv = \frac{n_i}{d_i m_i} \mathbb{I}_{n \times n}.
\]

 Finally, we compute the fourth-order moment. For this, we first consider
\[
\sum_{p,q=1}^n \int_{\mathbb{R}^n} v_p v_q (v \otimes v) (1 - |v|^2)^{\frac{d_i - 2}{2}}\, dv =: (A_{ab})_{n \times n}.
\]
For the diagonal entries $A_{aa}$, we calculate
\begin{align*}
A_{aa}&=\intr |v|^2 v_a^2   \left(1-|v|^2\right)^{\frac{d_i-2}{2}}_+\,dv   =\frac {1}{2n}|\pa {\rm B}_1|B\left(\frac{n}{2}+2,\frac {d_i}2 \right) =\frac{n}{2d_i(n+d_i)}|\pa {\rm B}_1|\frac{\Gamma(\frac n2)\Gamma(\frac {d_i}{2}+1)}{\Gamma(\frac{\gamma_i}{\gamma_i-1})}.
\end{align*}
For the off-diagonal entries $A_{ab}$ with $a \ne b$, we obtain
\begin{align}\label{ab}\begin{split}
A_{ab}&=2\intr v_a^2 v_b^2   \left(1-|v|^2\right)^{\frac{d_i-2}{2}}_+\,dv\cr 
&=2  \frac{1}{n(n+2)} |\partial \mathrm{B}_1| \, B\left(\frac{n}{2}+2, \frac{d_i}2 \right)\cr 
&=\frac{1}{n(n+2)}|\pa {\rm B}_1|\frac{\Gamma(\frac n2)\Gamma(\frac {d_i}{2}+1)}{\Gamma(\frac{\gamma_i}{\gamma_i-1})}\frac{n^2}{d_i(n+d_i)}\cr 
&=\frac{n}{d_i(n+2)(n+d_i)}|\pa {\rm B}_1|\frac{\Gamma(\frac n2)\Gamma(\frac {d_i}{2}+1)}{\Gamma(\frac{\gamma_i}{\gamma_i-1})}.
\end{split}\end{align}
We now multiply the fourth-order tensor by the prefactor
\[
\frac{c_i}{m_i^{\frac{n}{2} + 2}} \left( m_i^{\frac{n}{2}(\gamma_i - 1)} \frac{2\gamma_i}{\gamma_i - 1} n_i^{\gamma_i - 1} \right)^{\frac{n + d_i + 2}{2}},
\]
and simplify to obtain the final result:
\[
c_i^{\frac{2}{d_i}} \sum_{p,q=1}^n \int_{\mathbb{R}^n} (v_p - u_p)(v_q - u_q)(v - u) \otimes (v - u) \mathcal{M}_i^{\frac{d_i - 2}{d_i}}[n_i, u](v)\, dv = \frac{n \gamma_i}{d_i m_i^{2 - \frac{n}{2}(\gamma_i - 1)}} n_i^{\gamma_i} (B_{ab})_{n \times n},
\]
where the matrix $(B_{ab})$ is given by
\[
B_{ab} = \begin{cases}
1 & \text{if } a = b, \\
\frac{2}{n + 2} & \text{if } a \ne b.
\end{cases}
\]
Now define
$$
(C_{ab})_{n \times n}:=\sum_{p,q=1}^n\pa_{x_p}u_q\intr v_p v_q (v\otimes v) \left(1-|v |^2\right)^{\frac{d_i-2}{2}}_+  dv.
$$
For the diagonal entries $C_{aa}$, we write
\begin{align*}
C_{aa}&=\sum_{p,q=1}^n\pa_{x_p}u_q\intr v_p v_q v_a^2 \left(1-|v |^2\right)^{\frac{d_i-2}{2}}_+  dv\cr 
&=\sum_{p=1}^n\pa_{x_p}u_p\intr v_p^2 v_a^2 \left(1-|v |^2\right)^{\frac{d_i-2}{2}}_+  dv\cr 
&=\sum_{p\neq a}^n\pa_{x_p}u_p\intr v_p^2 v_a^2 \left(1-|v |^2\right)^{\frac{d_i-2}{2}}_+  dv+\pa_{x_a}u_a\intr  v_a^4 \left(1-|v |^2\right)^{\frac{d_i-2}{2}}_+  dv\cr 
&=\frac{1}{2n(n+2)} |\partial \mathrm{B}_1| \, B\left(\frac{n}{2}+2, \frac{d_i}{2}\right) \sum_{p\neq a}^n\pa_{x_p}u_p + \frac{3}{2n(n+2)} |\partial \mathrm{B}_1| \, B\left(\frac{n}{2}+2, \frac{d_i}{2}\right)\pa_{x_a}u_a\cr 
&=\frac{1}{2n(n+2)} |\partial \mathrm{B}_1| \, B\left(\frac{n}{2}+2, \frac{d_i}{2}\right) \lt( \nabla_x \cdot u + 2 \pa_{x_a} u_a \rt)\cr 
&=\frac 12\frac{n}{d_i(n+2)(n+d_i)}|\pa {\rm B}_1|\frac{\Gamma(\frac n2)\Gamma(\frac {d_i}{2}+1)}{\Gamma(\frac{\gamma_i}{\gamma_i-1})} \lt( \nabla_x \cdot u + 2 \pa_{x_a} u_a \rt),
\end{align*}
where we used \eqref{ab} and Lemma \ref{lem_vm0} (iv) and (v). For the off-diagonal entries $C_{ab}$, $a \ne b$, we obtain
\begin{align*}
C_{ab}&=\sum_{p,q=1}^n\pa_{x_p}u_q\intr v_p v_q v_av_b \left(1-|v |^2\right)^{\frac{d_i-2}{2}}_+  dv\cr 
&=(\pa_{x_a}u_b)\intr v_a^2 v_b^2 \left(1-|v |^2\right)^{\frac{d_i-2}{2}}_+ dv+(\pa_{x_b}u_a)\intr v_a^2 v_b^2 \left(1-|v |^2\right)^{\frac{d_i-2}{2}}_+  dv\cr 
&=\frac 12\frac{n}{d_i(n+2)(n+d_i)}|\pa {\rm B}_1|\frac{\Gamma(\frac n2)\Gamma(\frac {d_i}{2}+1)}{\Gamma(\frac{\gamma_i}{\gamma_i-1})}(\pa_{a}u_b+\pa_b u_a ).
\end{align*}
This completes the proof.
\end{proof}

%
%
%
%
%
%
%
%
%
%
%
%
%

\section{Global existence of mild solutions}\label{app_gwp}

In this appendix, we prove the global-in-time existence of mild solutions to the kinetic BGK-type system \eqref{main}, under suitable assumptions on the initial data. The constructed solutions also satisfy a kinetic entropy inequality. For simplicity, we fix the collision frequencies to unity: $\nu_1 = \nu_2 = 1$. Note that the parameter $\e > 0$ used in this section refers to a regularization parameter introduced later, and should not be confused with the Knudsen number used in the Chapman--Enskog expansion. For simplicity, we write $f := (f_1, f_2)$ and denote by $\calM_i[f] := \calM_i[n_i, \bf u]$ the local equilibrium associated with species $i$, where $n_i = \int_{\R^n} f_i\,dv$ and the bulk velocity $\bf u$ is defined as in \eqref{total}. For clarity of presentation, we restrict to the case $\Omega = \mathbb{R}^n$. The case $\Omega = \mathbb{T}^n$ is in fact simpler due to the compactness of the spatial domain.

\begin{theorem} \label{thm: exist}
    Let $T>0$ be given, and let the initial data ${}^{\rm in}f_i:\R^n \times \R^n\to [0,\infty)$, for $i=1,2$, satisfy the finite entropy and mass condition
    \begin{equation}\label{eq: hyp in}
        \iint_{\R^n \times \R^n} H({}^{\rm in}f,v) + {}^{\rm in}f_1 + {}^{\rm in}f_2 \, dxdv < +\infty.
    \end{equation}
    Then there exists a global-in-time weak solution $f = (f_1, f_2)$ to the BGK equation \eqref{main} which satisfies the Duhamel-type mild formulation:
    \begin{align*}
        f_i(t,x,v) = e^{-t} \cdot {}^{\rm in}f_i(x-vt,v) + \int_0^t e^{s-t} \calM_i[f](s,x-v(t-s),v)\,ds.
    \end{align*}
    for almost every $(t,x,v)\in (0,\infty)\times\R^n \times \R^n$ and for $i=1,2$.
    
    Moreover, the solution satisfies  
    \begin{align}\label{eq: prop f}
        f = (f_1, f_2) \in L^\infty((0,\infty);L^1_2 \cap L^{1+2/d_1} (\R^n \times \R^n)) \times L^\infty((0,\infty);L^1_2 \cap L^{1+2/d_2} (\R^n \times \R^n)),
    \end{align}
    and the kinetic entropy inequality
    \begin{align*}
        \iint_{\R^n \times \R^n} h_i(f_i,v)\,dxdv + \int_0^t \iint_{\R^n \times \R^n} h_i(\calM_i[f],v)\,dxdvds \le \iint_{\R^n \times \R^n} h_i({}^{\rm in}f_i,v)\,dxdv.
    \end{align*}
\end{theorem}
\begin{remark}
Once the mild formulation holds, one can verify that $f_i$ satisfies the BGK system \eqref{main} in the sense of distributions. See \cite[Appendix C]{KS24} for details.
\end{remark}

\begin{remark}
The proof presented below addresses the case $\gamma_1, \gamma_2 \in (1, \frac{n+2}{n})$, where the truncated Maxwellians are unbounded but integrable. The endpoint case $\gamma_i = \frac{n+2}{n}$ can be handled by similar, and in fact simpler, arguments due to the boundedness of $\mathcal{M}_i$ in that regime.
\end{remark}
%
%
%
%
%
%
%
%
%
%
%
%
%
\subsection{Regularization and linearization}
To establish global existence, we first construct a regularized version of the BGK model \eqref{main}. Specifically, we define the regularized macroscopic quantities:
\begin{align*}
    n^{(\ve)}_{f_1} := \frac{n_{f_1}}{1+\ve n_{f_1}}, \quad {\bf{u}}^{(\ve)}_{f_1,f_2} := \frac{\rho_{f_1}{\bf{u}}_{f_1} + \rho_{f_2}{\bf{u}}_{f_2}}{\rho_{f_1} + \rho_{f_2} + \ve(1+ |\rho_{f_1}{\bf{u}}_{f_1}| + |\rho_{f_2}{\bf{u}}_{f_2}| )},
\end{align*}
and consider the following regularized BGK-type system:
\begin{align}\label{eq: rbgk}
    \begin{cases}
        \p_t f^{\ve}_1 + v\cdot \nabla_x f^{\ve}_1 = \calM_1^{(\ve)}[f^\ve] - f_1^\ve,\\
        \p_t f^{\ve}_2 + v\cdot \nabla_x f^{\ve}_2 = \calM_2^{(\ve)}[f^\ve] - f_2^\ve,\\
        f^\ve_1(0,\cdot,\cdot) = {}^{\rm in}f_1(\cdot,\cdot), \\
        f^\ve_2(0,\cdot,\cdot) = {}^{\rm in}f_2(\cdot,\cdot),\\
        \calM_i^{(\ve)}[f^\ve] := c_i\left(\frac{2\gamma_i}{\gamma_i - 1} m_i^{\frac{n}{2}(\gamma_i-1)} (n_{f_i^\ve}^{(\ve)})^{\gamma_i-1} - m_i\left|v- {\bf{u}}^{(\ve)}_{f_1^\ve,f_2^\ve} \right|^2 \right)_+^{d_i/2}.
    \end{cases}
\end{align}
Notice that the regularized Maxwellian can be rewritten as
\begin{align*}
    \calM_i^{(\ve)}[(f_1,f_2)] = c_i m_i^{d_i/2} \left(\frac{2\gamma_i}{\gamma_i - 1} m_i^{\frac{n}{2}(\gamma_i - 1) - 1} (n_{f_i}^{(\ve)})^{\gamma_i-1} - \left|v- {\bf{u}}^{(\ve)}_{f_1,f_2}\right|^2 \right)_+^{d_i/2}.
\end{align*}

To justify the existence of the regularized system, we rely on the stability properties of the truncated Maxwellian. In particular, the following estimate from \cite[Lemma 1.9]{KS24} plays a central role. Let us denote
\[
L^1_2(\R^n \times \R^n) = L^1(\R^n \times \R^n,(1+|v|^2)dxdv)
\]
Then for any pair of functions $(f_1, f_2), (g_1, g_2) \in (L^1_2(\R^n \times \R^n))^2$, the regularized Maxwellian satisfies the following stability estimate:
\begin{equation}\label{eq: stab1}
\begin{split}
    &\int_{\R^n} (1+|v|^2)\lt(\calM^{(\ve)}_i[(f_1,f_2)] - \calM^{(\ve)}_i[(g_1,g_2)] \rt) dv  \le C_{\ve,m_i,\gamma_i} \left(|n_{f_i}^{(\ve)} - n_{g_i}^{(\ve)}| + |{\bf{u}}_{f_1,f_2}^{(\ve)} - {\bf{u}}_{g_1,g_2}^{(\ve)}| \right).
    \end{split}
\end{equation}
To make effective use of this estimate, we analyze the difference of the regularized velocities ${\bf{u}}_{f_1,f_2}^{(\ve)}$ and ${\bf{u}}_{g_1,g_2}^{(\ve)}$. A direct computation of the difference yields:
\begin{align*}
   {\bf{u}}_{f_1,f_2}^{(\ve)} - {\bf{u}}_{g_1,g_2}^{(\ve)} 
    &= \frac{(\rho_{f_1}{\bf{u}}_{f_1} + \rho_{f_2}{\bf{u}}_{f_2})\Big(\rho_{g_1} + \rho_{g_2}+\ve(1+|\rho_{g_1}{\bf{u}}_{g_1}| + |\rho_{g_2}{\bf{u}}_{g_2}|)\Big)}{(\rho_{f_1}+\rho_{f_2} + \ve(1+|\rho_{f_1}{\bf{u}}_{f_1}| + |\rho_{f_2}{\bf{u}}_{f_2}|) (\rho_{g_1}+\rho_{g_2}+\ve(1+|\rho_{g_1}{\bf{u}}_{g_1}| + |\rho_{g_2}{\bf{u}}_{g_2}|))} \\
    &\quad - \frac{(\rho_{g_1}{\bf{u}}_{g_1} + \rho_{g_2}{\bf{u}}_{g_2})(\rho_{f_1}+\rho_{f_2} + \ve(1+|\rho_{f_1}{\bf{u}}_{f_1}| + |\rho_{f_2}{\bf{u}}_{f_2}|)}{(\rho_{f_1}+\rho_{f_2} + \ve(1+|\rho_{f_1}{\bf{u}}_{f_1}| + |\rho_{f_2}{\bf{u}}_{f_2}|) (\rho_{g_1}+\rho_{g_2}+\ve(1+|\rho_{g_1}{\bf{u}}_{g_1}| + |\rho_{g_2}{\bf{u}}_{g_2}|))}.
\end{align*}
Since both expressions share a common denominator, their difference reduces to the difference of the numerators. We group the resulting terms as follows:
\begin{align*}
    &J_1 := (\rho_{f_1} {\bf{u}}_{f_1}) \rho_{g_1} - (\rho_{g_1} {\bf{u}}_{g_1}) \rho_{f_1}, \quad J_2 := (\rho_{f_1} {\bf{u}}_{f_1}) \rho_{g_2} - (\rho_{g_1} {\bf{u}}_{g_1}) \rho_{f_2},\\
    &J_3 := (\rho_{f_2}{\bf{u}}_{f_2}) \rho_{g_1} - (\rho_{g_2}{\bf{u}}_{g_2} )\rho_{f_1}, \quad J_4 := (\rho_{f_2}{\bf{u}}_{f_2})\rho_{g_2} - (\rho_{g_2}{\bf{u}}_{g_2})\rho_{f_2},\\
    &J_5 := (\rho_{f_1}{\bf{u}}_{f_1})\ve(1+|\rho_{g_1}{\bf{u}}_{g_1}|+|\rho_{g_2}{\bf{u}}_{g_2}|)  - (\rho_{g_1}{\bf{u}}_{g_1})\ve(1+|\rho_{f_1}{\bf{u}}_{f_1}|+|\rho_{f_2}{\bf{u}}_{f_2}|) ,\\
    &J_6 := (\rho_{f_2}{\bf{u}}_{f_2})\ve(1+|\rho_{g_1}{\bf{u}}_{g_1}|+|\rho_{g_2}{\bf{u}}_{g_2}|) - (\rho_{g_2}{\bf{u}}_{g_2})\ve(1+|\rho_{f_1}{\bf{u}}_{f_1}|+|\rho_{f_2}{\bf{u}}_{f_2}|) .
\end{align*}
Then, we make the following quick observations:\\
$\bullet$ Estimate of $J_1$: Subtracting and adding $ (\rho_{g_1})^2 {\bf{u}}_{g_1}$ gives
\begin{align*}
    J_1 = (\rho_{f_1}{\bf{u}}_{f_1} - \rho_{g_1}{\bf{u}}_{g_1})\rho_{g_1} + (\rho_{g_1}{\bf{u}}_{g_1}) (\rho_{g_1} - \rho_{f_1}).
\end{align*}
$\bullet$ Estimate of $J_2$: Similarly, we find
\begin{align*}
    J_2 &= (\rho_{f_1} {\bf{u}}_{f_1} - \rho_{g_1} {\bf{u}}_{g_1}) \rho_{g_2} +  (\rho_{g_1} {\bf{u}}_{g_1})(\rho_{g_2} - \rho_{f_2})  .
\end{align*}
$\bullet$ Estimates of $J_3,J_4$: In the same way,
\begin{align*}
    &J_3 =(\rho_{f_2} {\bf{u}}_{f_2} - \rho_{g_2} {\bf{u}}_{g_2}) \rho_{g_1} +  (\rho_{g_2} {\bf{u}}_{g_2}) (\rho_{g_1} - \rho_{f_1}),\\
    &J_4 = (\rho_{f_2} {\bf{u}}_{f_2} - \rho_{g_2} {\bf{u}}_{g_2}) \rho_{g_2} + (\rho_{g_2} {\bf{u}}_{g_2})(\rho_{g_2} - \rho_{f_2}).
\end{align*}
$\bullet$ Estimate of $J_5$: Subtracting and adding $\ve  (\rho_{g_1}{\bf{u}}_{g_1})(1+|\rho_{g_1}{\bf{u}}_{g_1}| + |\rho_{g_2}{\bf{u}}_{g_2}|)$, we obtain
\begin{align*}
    J_5 &= \ve (\rho_{f_1}{\bf{u}}_{f_1} - \rho_{g_1}{\bf{u}}_{g_1}) (1+|\rho_{g_1}{\bf{u}}_{g_1}|+|\rho_{g_2}{\bf{u}}_{g_2}|)  + \ve  (\rho_{g_1}{\bf{u}}_{g_1}) (|\rho_{g_1}{\bf{u}}_{g_1}| - |\rho_{f_1}{\bf{u}}_{f_1}| + |\rho_{g_2}{\bf{u}}_{g_2}| - |\rho_{f_2}{\bf{u}}_{f_2}|) \\
    &\le \ve  |\rho_{f_1} {\bf{u}}_{f_1} - \rho_{g_1}{\bf{u}}_{g_1}| (1+|\rho_{g_1}{\bf{u}}_{g_1}|+|\rho_{g_2}{\bf{u}}_{g_2}|)  + \ve  |\rho_{g_1} {\bf{u}}_{g_1}| |\rho_{g_1}{\bf{u}}_{g_1} - \rho_{f_1}{\bf{u}}_{f_1}| + \ve  |\rho_{g_1} {\bf{u}}_{g_1}| |\rho_{g_2}{\bf{u}}_{g_2} - \rho_{f_2}{\bf{u}}_{f_2}| .
\end{align*}
$\bullet$ Estimate of $J_6$: In the same way,
\begin{align*}
    J_6 &= \ve   (\rho_{f_2}{\bf{u}}_{f_2} - \rho_{g_2}{\bf{u}}_{g_2}) (1+|\rho_{g_1}{\bf{u}}_{g_1}| + |\rho_{g_2} {\bf{u}}_{g_2}|)   + \ve  (\rho_{g_2} {\bf{u}}_{g_2}) (|\rho_{g_1}{\bf{u}}_{g_1}| - |\rho_{f_1}{\bf{u}}_{f_1}| + |\rho_{g_2}{\bf{u}}_{g_2}| - |\rho_{f_2} {\bf{u}}_{f_2}|) \\
    &\le \ve  |\rho_{f_2}{\bf{u}}_{f_2} - \rho_{g_2}{\bf{u}}_{g_2}| (1 + |\rho_{g_1}{\bf{u}}_{g_1}| + |\rho_{g_2}{\bf{u}}_{g_2}|)  + \ve  |\rho_{g_2} {\bf{u}}_{g_2}| |\rho_{g_1}{\bf{u}}_{g_1} - \rho_{f_1}{\bf{u}}_{f_1}| + \ve  |\rho_{g_2}{\bf{u}}_{g_2}| |\rho_{g_2}{\bf{u}}_{g_2} - \rho_{f_2}{\bf{u}}_{f_2}|.
\end{align*}
Now taking the denominator into effect, we deduce that
\begin{equation*}
\begin{split}
    |{\bf{u}}_{f_1,f_2}^{(\ve)} - {\bf{u}}_{g_1,g_2}^{(\ve)}| &\lesssim_{m_1,m_2} \frac{|n_{f_1}{\bf{u}}_{f_1} - n_{g_1}{\bf{u}}_{g_1}| + |n_{f_2}{\bf{u}}_{f_2} - n_{g_2}{\bf{u}}_{g_2}|}{\ve} + \frac{|n_{f_1}-n_{g_1}| + |n_{f_2}-n_{g_2}|}{\ve^2}\\
    &\lesssim \left(\frac{1}{\ve}+\frac{1}{\ve^2}\right) \int_{\R^n} (1+|v|^2)\Big(|f_1 - g_1| + |f_2 - g_2| \Big) dv.
\end{split}
\end{equation*}
Finally, we also control the difference in the regularized densities:
\begin{align*}
    |n_{f_i}^{(\ve)} - n_{g_i}^{(\ve)}| &= \left|\frac{n_{f_i} - n_{g_i}}{(1+\ve n_{f_i})(1+\ve n_{g_i})} \right| \le |n_{f_i} - n_{g_i}| \le \int_{\R^n} (1+|v|^2) |f_i-g_i| \,dv.
\end{align*}
Combining the estimates above, we conclude from \eqref{eq: stab1} that
\begin{equation}\label{eq: stabimp}
\begin{split}
 &\iint_{\R^n \times \R^n} (1+|v|^2)(\calM^{(\ve)}_i[f] - \calM^{(\ve)}_i[g]) \,dvdx  \cr
 &\quad \le C_{\ve,m_1,m_2,\gamma_i} \iint_{\R^n \times \R^n} (1+|v|^2) (|f_1 - g_1| + |f_2 - g_2|) \, dvdx.
    \end{split}
\end{equation}
By using this stability estimate for the regularized Maxwellian, we now construct mild solutions to the regularized system \eqref{eq: rbgk} in the following lemma.
\begin{lemma} \label{lem: regbgk}
    Assume \eqref{eq: hyp in}. Then for each $\ve>0$, the Cauchy problem for the regularized system \eqref{eq: rbgk} admits a mild solution. That is, for almost every $(t,x,v)\in [0,\infty)\times\R^n \times \R^n$, we have
    \begin{align}\label{eq: reg mild}
        f_i^\ve(t,x,v) = e^{-t}\cdot {}^{\rm in}f_i(x-vt,v) + \int_0^t e^{s-t}\calM_i^{(\ve)}[f^\ve](s,x-v(t-s),v)\,ds.
    \end{align}
    Moreover
    \begin{align*}
        (f_1^\ve,f_2^\ve) \in \prod_{i=1}^2 L^\infty([0,\infty);L^1_2 \cap L^{1+2/d_i} (\R^n \times \R^n)),
    \end{align*}
    and the following uniform estimates hold:
    \begin{align*}
        \sup_{t\ge 0}\iint_{\R^n \times \R^n} f_i^\ve + h_i(f_i^\ve,v) \, dvdx \le C({}^{\rm in}f_i).
    \end{align*}
\end{lemma}
\begin{proof}
We construct a sequence of approximations via the standard iterative scheme. For $k \in \mathbb{N}$, define $(f_1^{k+1},f_2^{k+1})$ by solving the linearized problem:
\[
    \begin{cases}
        \p_t f_1^{k+1} + v\cdot \nabla_x f_1^{k+1} = \calM^{(\ve)}_1[f^k] - f_1^{k+1},\\
        \p_t f_2^{k+1} + v \cdot \nabla_x f_2^{k+1} = \calM^{(\ve)}_2[f^k] - f_2^{k+1},\\
        \calM_i[f^k] := c_i \left(\frac{2\gamma_i}{\gamma_i - 1} m_i^{\frac{n}{2}(\gamma_i - 1)} \left(n_{f_i^k}^{(\ve)}\right)^{\gamma_i - 1} - m_i\left|v - {\bf{u}}_{f_1^k,f_2^k}^{(\ve)}\right|^2 \right)_+^{d_i/2},\\
        f_i^0(t,x,v) = {}^{\rm in}f_i(x,v)\quad \forall (t,x,v)\in [0,\infty)\times\R^n \times \R^n,\\
        f_i^k(0,x,v) = {}^{\rm in}f_i(x,v) \quad \forall k\in \mathbb{N}\text{ and }(x,v)\in \R^n \times \R^n.
    \end{cases}
\]
For each $k \in \N$, this is a linear inhomogeneous transport equation with an explicitly given source term, and classical theory ensures the existence of a unique mild solution.  In particular, each $f_i^k$ satisfies the mild formula
\begin{align*}
    f_i^{k+1}(t,x,v) = e^{-t}\cdot {}^{\rm in}f_i(x-vt,v) + \int_0^t e^{s-t}\calM_i[f^k](s,x-v(t-s),v)\,ds.
\end{align*}
Although the $f_i^{k+1}$ also depends on the parameter $\ve>0$, we have omitted this dependence for the sake of brevity. From conservation and entropy balance laws (see \cite{CH24,KS24}), we obtain for each $i=1,2$:
\begin{align*}
    &\frac{d}{dt}\iint_{\R^n \times \R^n} (f_1^{k+1} + f_2^{k+1}) \,dxdv + \iint_{\R^n \times \R^n} (f_1^{k+1} + f_2^{k+1}) \,dxdv \le \iint_{\R^n \times \R^n} (f_1^k + f_2^k) \,dxdv ,\\
    &\frac{d}{dt}\iint_{\R^n \times \R^n} h_i(f_i^{k+1},v)\, dx dv + \iint_{\R^n \times \R^n} h_i(f_i^{k+1},v) dxdv \le \iint_{\R^n \times \R^n} h_i(f_i^{k},v) \,dxdv.
\end{align*}
By an inductive argument (see \cite[Lemma 3.3]{KS24}), this shows
\begin{equation}\label{eq: unif}
\begin{split}
    &\sup_{t\ge 0} \iint_{\R^n \times \R^n} (f_1^{k+1} + f_2^{k+1})\, dxdv \le \iint_{\R^n \times \R^n} ({}^{\rm in}f_1(\cdot,\cdot)+{}^{\rm in}f_2(\cdot,\cdot))\, dxdv,\\
    &\sup_{t\ge 0} \iint_{\R^n \times \R^n} h_i(f_i^{k+1},v)\, dxdv \le \iint_{\R^n \times \R^n} h_i({}^{\rm in}f_i(\cdot,\cdot),v)\, dxdv.
    \end{split}
\end{equation}
To establish convergence of the iteration, we consider the $L^1_2$-distance between successive iterates:
\begin{equation}\label{eq: ddt xk}
\begin{split}
    &\frac{d}{dt}\iint_{\R^n \times \R^n} (1+|v|^2)|f^{k+1}_i - f^k_i|\,dxdv + \iint_{\R^n \times \R^n} (1+|v|^2)|f^{k+1}_i - f^k_i|\,dxdv \\
    &\quad = \iint_{\R^n \times \R^n} (1+|v|^2)\textnormal{sgn}\lt(f^{k+1}_i - f^k_i) (\calM^{(\ve)}_i [f^k] - \calM^{(\ve)}_i [f^{k-1}]\rt) dxdv,\quad i=1,2.
\end{split}
\end{equation}
Utilizing \eqref{eq: stabimp}, we find that the right-hand side can be estimated as
\begin{equation} \label{eq: stab}
\begin{split}
    &\iint_{\R^n \times \R^n} (1+|v|^2)\textnormal{sgn}\lt(f^{k+1}_i - f^k_i) (\calM^{(\ve)}_i [f^k] - \calM^{(\ve)}_i [f^{k-1}] \rt) dxdv\\
    &\quad \le C_{\ve,m_1,m_2,\gamma_i} \int_{\R^n} \left|n_{f_i^{k}}^{(\ve)} - n_{f_i^{k-1}}^{(\ve)} \right| + \left|{\bf{u}}^{(\ve)}_{f_1^{k},f_2^k} - {\bf{u}}^{(\ve)}_{f_1^{k-1}, f_2^{k-1}}\right| dx\\
    &\quad \le C_{\ve,m_1,m_2,\gamma_1,\gamma_2} \iint_{\R^n \times \R^n} (1+|v|^2)\Big(|f_1^k - f_1^{k-1}| + |f_2^k - f_2^{k-1}| \Big)\, dxdv.
\end{split}
\end{equation}
Thus, if we let $\calX^{k+1}(t) = \sum_{i=1}^2 \|f_i^{k+1} - f_i^k\|_{L^1_2}$, then \eqref{eq: ddt xk} and \eqref{eq: stab} imply
\begin{align*}
    \frac{d}{dt}\calX^{k+1}(t) + \calX^{k+1}(t) \le C_{\ve,m_1,m_2,\gamma_1,\gamma_2}\calX^k(t).
\end{align*}
Since $\calX^k(0) = 0$ for all $k\in \mathbb N$, this implies $\sum_{k=0}^\infty \calX^{k+1}(t)$ is convergent. In particular, for both $i=1,2$, the sequences $\{f_i^{k+1}\}_k$ are Cauchy in $L^\infty_{\rm loc}((0,\infty);L^1_2(\R^n \times \R^n))$. Consequently, we obtain limits $f_i^\ve$ such that 
\begin{align*}
    f_i^{k+1} \underset{k\to\infty}{\longrightarrow} f_i^\ve \quad \text{in}\quad L^\infty_{\rm loc}((0,\infty);L^1_2(\R^n \times \R^n)).
\end{align*}
Moreover, since the mapping $f \mapsto \calM_i^{(\ve)}[f]$ is Lipschitz continuous in the $L^1_2$ norm, the convergence also carries over to the Maxwellians:
\begin{align*}
    \calM_i[f^{k+1}] \underset{k\to\infty}{\longrightarrow} \calM_i[f^\ve] \quad \text{in}\quad L^\infty_{\rm loc} ((0,\infty);L^1_2(\R^n \times \R^n)).
\end{align*}
 We can therefore pass to the limit in the mild formulas for the $(f_1^{k+1},f_2^{k+1})$, in the same way as in \cite[Proposition 3.1]{KS24}, and find that the limit $f_i^\ve$ satisfies the mild formula \eqref{eq: reg mild}. Finally, the uniform estimates established in \eqref{eq: unif} pass to the limit due to the lower semicontinuity of convex integrals with respect to $L^1$ convergence. This completes the proof.
\end{proof}
%
%
%
%
%
%
%
%
%
%
%
%
%
\subsection{Passing to the limit}

In order to obtain compactness for the sequences $\{f_i^\ve\}$, we recall some useful lemmas.

\begin{lemma}
    \cite[Lemma 3.4]{KS24} Let $\{f^m\}$ a family of weak solutions to
    \begin{align*}
        \p_t f^m + v\cdot \nabla_x f^m = G^m,\\
        f^m(0,\cdot,\cdot) = {}^{\rm in}f(\cdot,\cdot),
    \end{align*}
    where $G^m\in L^\infty([0,\infty);L^1(\R^n \times \R^n))$ and
    \begin{align*}
        \intr G^m\,dv \le 0.
    \end{align*}
    Assume also that $\mathscr{E} := \iint_{\R^n \times \R^n} |v|^2 f^m\,dxdv < +\infty$. Then for any $\psi\in C^0(\R^n)$ with $|\psi|\lesssim (1+|v|)^{\sigma}$ and $\sigma\in [0,2)$, the function
    \begin{align*}
        \rho_{\psi}^m := \intr \psi(v) f^m\,dv
    \end{align*}
    satisfies
    \begin{align*}
        \sup_m \sup_{t\in [0,T]}\int_{|x|\ge 2R} |\rho_{\psi}^m|\,dx \le \mathscr{E}^{\sigma/2} \left(\int_{|x|\ge R} \intr {}^{\rm in}f(x,v)\,dvdx \right)^{1-\sigma/2}.
    \end{align*}
\end{lemma}
Notice that the above lemma applies to each $f_i^\ve$ with $G_i^\ve := \calM_i^{(\ve)}[f^\ve] - f_i^\ve$, since
\begin{align*}
    \intr G_i^\ve\,dv = \rho_{f_i^\ve}^{(\ve)} - \rho_{f_i^\ve} \le 0.
\end{align*}
Next to recall is the following averaging lemma.
\begin{lemma} \label{lem: avg}
    \cite[Proposition 4.1]{KS24} Let $p\in (1,\infty)$, and $f^m$ solutions to the transport equations
    \begin{align*}
        \begin{cases}
        \p_t f^m + v \cdot \nabla_x f^m = G^m_+ - G^m_-, \\f^m(0,\cdot,\cdot) = {}^{\rm in}f(\cdot,\cdot) \in L^p(\R^n \times \R^n),
        \end{cases}
    \end{align*}
    where $G^m_+, G^m_- \ge 0$. Assume the following:
    \begin{enumerate}[label=(\roman*)]
    \item $\{f^m\}$ is bounded in $L^\infty([0,\infty);L^p(\R^n \times \R^n))$.
    \item $\{(1+|v|^2)f^m\}$ is bounded in $L^\infty([0,\infty);L^1(\R^n \times \R^n))$.
    \item $\{G^m_+\}$ and $\{G^m_-\}$ are bounded in $L_{\rm loc}^1([0,\infty);L^1(\R^n \times \R^n))$.
    \item For each $T>0$, the family $\{\rho_{f^m}(t,\cdot)\}_{(m,t)\in\N\times[0,T]}$ is tight:
    \begin{align*}
    \lim_{R\to \infty} \sup_{0\le t \le T} \sup_{m} \int_{|x|\ge R} \rho_{f^m} \,dx = 0.
    \end{align*}
    \end{enumerate}
    Then, for each $\psi \in C(\R^n)$ verifying $|\psi(v)| \lesssim (1+|v|)^\sigma$ with $\sigma\in [0,2)$, the sequence $\left\{\int_{\R^n}f^{m}\psi(v)\, dv\right\}$ is relatively compact in $L_{\rm loc}^1([0,\infty);L^1(\R^n))$.
\end{lemma}

%
%
%
%
%
%
%
%
%
%
%
%
%
\subsection{Proof of Theorem \ref{thm: exist}}
We begin by observing that the uniform bounds on the kinetic entropy of the regularized solutions $f^\ve = (f_1^\ve, f_2^\ve)$ imply the existence of a limit function $f = (f_1, f_2)$ such that, up to a subsequence,
\begin{align*}
    f_i^\ve \rightharpoonup f_i \quad \text{weakly$^*$ in }L^\infty((0,\infty);L^{1+\frac{2}{d_i}}(\R^n \times \R^n)).
\end{align*}

A direct application of Lemma \ref{lem: avg} shows
\[
   n_{f_i^\ve} \to n_{f_i} \quad \mbox{and} \quad n_{f_i^\ve} {\bf{u}}_{f_i^\ve} \to n_{f_i} {\bf{u}}_{f_i}
\]
a.e. in $(t,x)$ and strongly in $L^{p_i}((0,T);L^{p_i}(\R^n))$ for any $p_i \in [1,\frac{2\gamma_i}{\gamma_i + 1})$ and all $T>0$. Recalling the definitions of the regularized quantities $n^{(\ve)}$, ${\bf{u}}^{(\ve)}$, and the truncated Maxwellians $\calM_i^{(\ve)}$, we deduce from the strong convergence of the moments that
\[
    \calM_i^{(\ve)}[f^\ve] \to \calM_i[f] \quad \text{a.e. in }P_{t,x} \times \R^n, 
\]
where 
\[
P_{t,x}:= \{(t,x)\in [0,T]\times\R^n: (n_{f_1} + n_{f_2})(t,x) > 0\}.
\]
On the complementary set $(([0,T]\times\R^n) \setminus P_{t,x}) \times \R^n$, we can directly estimate the convergence in $L^1$ using the fact that both Maxwellians vanish when densities vanish. Indeed,
    \begin{align*}
      \iiint_{(([0,T]\times\R^n) \setminus P_{t,x}) \times \R^n} \left|\calM_i^{(\ve)}[f^\ve] - \calM_i[f]\right|\, dxdvdt  
        &\le \iiint_{(([0,T]\times\R^n) \setminus P_{t,x}) \times \R^n} \calM_i^{(\ve)}[f^\ve] + \calM_i[f]\, dxdvdt \\
        &\le \iint_{([0,T]\times\R^n)\setminus P_{t,x}} n_{f_i}^{(\ve)} + n_{f_i} \, dxdt \\
        &\le \iint_{([0,T]\times\R^n)\setminus P_{t,x}} 2n_{f_i} \,dxdt \cr
        &=0,
    \end{align*}
 where we used the definition of $P_{t,x}$ and the monotonicity of the regularization $n_{f_i}^{(\ve)} \le n_{f_i}$.

Thus, combining the two regions, we conclude that up to a subsequence
\[
\calM_i^{(\ve)}[f^\ve] \to \calM_i[f] \quad \mbox{a.e. in } [0,T]\times\R^n \times \R^n
\]
for any $T>0$. A diagonal argument permits passing to a further subsequence such that 
\[
\calM_i^{(\ve)}[f^\ve] \to \calM_i[f] \quad \mbox{a.e. in } [0,\infty)\times\R^n \times \R^n.
\]
As a result of this pointwise convergence and the bounds
\begin{align*}
    \iint_{\R^n \times \R^n} h_i(\calM_i^{(\ve)}[f^\ve],v) \,dxdv &= \intr \frac{m_i}{2}n_{f_i^\ve}^{(\ve)}|{\bf{u}}_{f^\ve}^{(\ve)}|^2 + \frac{m_i^{\frac{n(\gamma_i-1)}{2}}}{\gamma_i-1} (n_{f_i^\ve}^{(\ve)})^{\gamma_i}\,dx\\
    &\le \intr \frac{m_i}{2}n_{f_i^\ve}|{\bf{u}}_{f^\ve}|^2 + \frac{m_i^{\frac{n(\gamma_i-1)}{2}}}{\gamma_i-1} (n_{f_i^\ve})^{\gamma_i}\,dx\\
    &= \iint_{\R^n \times \R^n} h_i(\calM_i[f^\ve],v)\,dxdv \\
    &\le \iint_{\R^n \times \R^n} h_i(f_i^\ve,v)\,dxdv \\
    &\le \iint_{\R^n \times \R^n} h_i({}^{\rm in}f_i^\ve,v)\,dxdv,
\end{align*}
we deduce from the Vitali convergence theorem that
\begin{align*}
    \calM_i^{(\ve)}[f^\ve] \to \calM_i[f] \quad \text{in}\quad L^1_{\rm loc}((0,\infty); L^1(\R^n \times \R^n)).
\end{align*}
With this strong convergence in hand, we may pass to the limit in the mild formulation of the regularized problem, see \cite[Lemma 4.3]{KS24} for details. That is, for almost every $(t,x,v)$,
\begin{align*}
    &f_i^\ve \to f_i \quad \text{in} \quad L^1_{\rm loc}((0,\infty);L^1(\R^n \times \R^n)),\\
    &f_i(t,x,v) = e^{-t} \cdot {}^{\rm in}f_i(x-vt,v) + \int_0^t e^{-t+s} \calM_i[f](s,x-v(t-s),v)\,ds.
\end{align*}
Hence, we obtain a global-in-time mild solution to the original BGK-type system \eqref{main}.

It remains to verify the additional properties listed in \eqref{eq: prop f}, including the kinetic entropy inequality and moment bounds. Since the derivation of these estimates follows the same line of reasoning as in \cite[Sections 4--5]{KS24}, we omit the details and refer the reader to that work for a complete analysis.
%
%
%
%
%
%
%
%
%
%
%
%
%



%
%
%
%


\begin{thebibliography}{10}

\bibitem{AAP02} P.  Andries, K. Aoki, and B. Perthame, A consistent BGK-type model for gas mixtures, J. Stat. Phys., 106, (2002), 993--1018.

\bibitem{ARS} U. M. Ascher, S. J. Ruuth, and R. J. Spiteri, Implicit-explicit Runge--Kutta methods for time-dependent partial differential equations, Appl. Numer. Math., 25, (1997), 151--167.

 
\bibitem{BKPY21}
G.-C. Bae, C. Klingenberg,  M. Pirner, and S.-B. Yun,  BGK model of the multi-constituents Uehling Uhlenbeck equation, Kinetic and related models, 14, (2021), 25--44.

\bibitem{BKPY22} G.-C. Bae, C. Klingenberg, M. Pirner, and S.-B. Yun,  BGK model for two-component gases near a global Maxwellian, SIAM J. Math. Anal., 55, (2023), 1007--1047.

\bibitem{BV05} F. Berthelin and A. Vasseur, From kinetic equations to multidimensional isentropic gas dynamics before shocks, SIAM J. Math. Anal., 36, (2005), 1807--1835.

\bibitem{BGK54} P. L. Bhatnagar, E. P. Gross, and M. L. Krook, A model for collision processes in gases. I. Small amplitude processes in charged and neutral one-component systems, Phys. Rev., 94, (1954), 511--525. 

\bibitem{BMS18}
M. Bisi, R. Monaco, and A. J. Soares, A BGK model for reactive mixtures of polyatomic gases with
continuous internal energy, J. Phys. A, 51, (2018),1--29.

\bibitem{BT20} M. Bisi and R. Travaglini,  A BGK model for mixtures of monoatomic and polyatomic gases with discrete internal energy, Physica A: Statistical Mechanics and its Applications, 547, (2020),  124441.


\bibitem{BBGSP18}
A. V. Bobylev, M. Bisi, M. Groppi, G. Spiga, and I. F. Potapenko,   A general consistent BGK model for
gas mixtures, Kinet. Relat. Models, 11, (2018), 1377--1393.

\bibitem{BPR} S. Boscarino, L. Pareschi, G. Russo, Implicit-explicit methods for evolutionary partial differential equations, Math. Model. Comput., 24, Society for Industrial and Applied Mathematics (SIAM), Philadelphia, PA, [2025], \copyright 2025. ix+323 pp

\bibitem{B99} F. Bouchut, Construction of BGK models with a family of kinetic entropies for a given system of conservation laws, J. Statist. Phys., 95, (1999), 113--170.

 
	\bibitem{Brull15} S. Brull,  An ellipsoidal statistical model for gas mixtures, Communications in Mathematical Sciences, 13, (2015), 1--13.
	
		\bibitem{BPS12}
	S. Brull, V. Pavan, and J. Schneider,  Derivation of a BGK model for mixtures, Eur. J. Mech. B-Fluid, 33, (2012), 74--86.
	
	\bibitem{CBGR22}
S.Y. Cho, S. Boscarino, M. Groppi, and G. Russo, Conservative semi-Lagrangian schemes for a general
consistent BGK model for inert gas mixtures, Commun. Math. Sci., 20, (2022), 695--725.

\bibitem{CCJ21} J. A. Carrillo, Y.-P. Choi, and J. Jung, Quantifying the hydrodynamic limit of Vlasov-type equations with alignment and nonlocal forces, Math. Models Methods Appl. Sci., 31, (2021), 327--408.

\bibitem{CCK16} J. A. Carrillo, Y.-P. Choi, and T. K. Karper, On the analysis of a coupled kinetic-fluid model with local alignment forces, Ann. Inst. H. Poincar\'e Anal. Non Lin\'eaire, 33, (2016), 273--307.

\bibitem{CG06} J. A. Carrillo and T. Goudon, Stability and asymptotic analysis of a fluid-particle interaction model, Comm. Partial Differential Equations, 31, (2006), 1349--1379.


 

\bibitem{CH24} Y.-P. Choi and B.-H. Hwang, Global existence of weak solutions to a BGK model relaxing to the barotropic Euler equations, Nonlinear Anal., 238, (2024), 113414. 
 
\bibitem{CJ21} Y.-P. Choi and J. Jung, Asymptotic analysis for a Vlasov-Fokker-Planck/Navier-Stokes system in a bounded domain, Math. Models Methods Appl. Sci., 31, (2021), 2213--2295.


\bibitem{CJ23} Y.-P. Choi and J. Jung, On the dynamics of charged particles in an incompressible flow: from kinetic-fluid to fluid-fluid models, Commun. Contemp. Math., 25, 2250012, (2023).

\bibitem{CKP20} 
A. Crestetto, C. Klingenberg, and M. Pirner, Kinetic/fluid micro-macro numerical scheme for a two
 component gas mixture, Multiscale Model. Simul., 18, (2020), 970--998.



\bibitem{D79} C. M. Dafermos, The second law of thermodynamics and stability, Arch. Rational Mech. Anal., 70, (1979), 167--179.

\bibitem{FJ} F. Filbet and S. Jin, An asymptotic preserving scheme for the ES-BGK model of the Boltzmann equation, J. Sci. Comput., 46, (2011), 204--224.

\bibitem{GJV04a} T. Goudon, P.-E. Jabin, and A. Vasseur, Hydrodynamic limit for the Vlasov-Navier-Stokes equations: I. Light particles regime, Indiana Univ. Math. J., 53, (2004), 1495--1515.

\bibitem{GJV04b} T. Goudon, P.-E. Jabin, and A. Vasseur, Hydrodynamic limit for the Vlasov-Navier-Stokes equations: II. Fine particles regime, Indiana Univ. Math. J., 53, (2004), 1517--1536. 

	\bibitem{GK56}
E. P. Gross and M. Krook,  Model for collision processes in gases: Small-amplitude oscillations of charged two-component systems, Phys. Rev., (1956), 102:593.

\bibitem{GMS11} 
M. Groppi, S. Monica, and G. Spiga, A kinetic ellipsoidal BGK model for a binary gas mixture, Europhysics Letters, 96, (2011), 64002.

	\bibitem{GRS09}
M. Groppi, S. Rjasanow, and G. Spiga,  A kinetic relaxation approach to fast reactive mixtures: shock wave structure, J. Stat. Mech.-Theory Exp., (2009), P10010.

\bibitem{GRS18}
M. Groppi, G. Russo, and G. Stracquadanio, Semi-Lagrangian Approximation of BGK Models for Inert and Reactive Gas Mixtures, In: P. Goncalves, A. Soares (eds.). "From Particle Systems to Partial Differential Equations V", Springer Proceedings in Mathematics and Statistics, 258, (2018), 53--80.

\bibitem{HKM24} D. Han-Kwan and D. Michel, On hydrodynamic limits of the Vlasov-Navier-Stokes system, Mem. Amer. Math. Soc., 302, (2024), no. 1516, v+115 pp.

 \bibitem{HHM17} J. R. Haack, C. D. Hauck, and M. S.  Murillo, A conservative, entropic multispecies BGK model, J. Stat. Phys., 168, (2017), 826--856.
 
 \bibitem{ZH} J. Hu and X. Zhang, On a class of implicit-explicit Runge-Kutta schemes for stiff kinetic equations preserving the Navier-Stokes limit, J. Sci. Comput., 73, (2017), 797--818.

\bibitem{HLY24} B.-H. Hwang, M.-S. Lee, and S.-B. Yun, Relativistic BGK model for gas mixtures, J. Stat. Phys., 191, (2024), 59.

 
	
 \bibitem{KLY21} D. Kim, M.-S. Lee, and S.-B. Yun,  Stationary BGK models for chemically reacting gas in a slab, J. Stat. Phys., 184, (2021), 1--33.
	
 

\bibitem{KMT15} T. K. Karper, A. Mellet, and K. Trivisa, Hydrodynamic limit of the kinetic Cucker-Smale flocking model, Math. Models Methods Appl. Sci., 25, (2015), 131--163.


 
	\bibitem{KP18} C. Klingenberg and  M. Pirner, Existence, uniqueness and positivity of solutions for BGK models for mixtures, J. Differential Equations, 264, (2018), 702--727.

\bibitem{KPP17}
C. Klingenberg, M. Pirner, and G. Puppo, A consistent kinetic model for a two-component mixture with an application to plasma, Kinet. Relat. Models, 10, (2017), 445--465. 


 
\bibitem{KS24} D. Koo and S. Song, Global mild solutions to a BGK model for barotropic gas dynamics, SIAM J. Math. Anal., to appear. 


\bibitem{LPT91} P.-L. Lions, B. Perthame, and E. Tadmor, Formulation cin\'{e}tique des lois de conservation scalaires multidimensionnelles,
C.R. Acad. Sci. Paris, S\'{e}rie I Math., 312,  (1991) 97--102.


\bibitem{LPT94} P.-L. Lions, B. Perthame, and E. Tadmor, A kinetic formulation of multidimensional scalar conservation laws and related equations, J. Amer. Math. Soc., 7, (1994), 169--191.

\bibitem{LPT94-2} P.-L. Lions, B. Perthame, and E. Tadmor, Kinetic formulation of the isentropic gas dynamics and $p$-systems, Commun. Math. Phys., 163, (1994), 415--431

 
\bibitem{MV08} A. Mellet and A. Vasseur, Asymptotic analysis for a Vlasov-Fokker-Planck/compressible Navier-Stokes system of equations, Commun. Math. Phys., 281, (2008), 573--596.

\bibitem{PP} S. Pieraccini and G. Puppo, Implicit-explicit schemes for BGK kinetic equations. J. Sci. Comput., 32, (2007), 1--28. 	

\bibitem{Pirner18} M. Pirner,  Existence and uniqueness of mild solutions for BGK models for gas mixtures of polyatomic molecules, arxiv:1806.10603.

\bibitem{CWS} C.-W. Shu, Essentially non-oscillatory and weighted essentially non-oscillatory schemes for hyperbolic conservation laws. In: Advanced Numerical Approximation of Nonlinear Hyperbolic Equations, pp. 325--432, Springer-Verlag, Berlin, (1998).


\bibitem{SWYZ23} Y. Su, G. Wu, L. Yao, and Y. Zhang, Hydrodynamic limit for the inhomogeneous incompressible Navier-Stokes-Vlasov equations, J. Differential Equations, 342, (2023), 193--238.














\end{thebibliography}
\end{document}